\documentclass[12pt]{article}

\usepackage[utf8]{inputenc} 
\usepackage[T1]{fontenc}          

\usepackage[largesc,osf]{newtxtext}
\usepackage{cite,enumitem,graphicx,mathtools,slashed}
\usepackage[amsthm,nosymbolsc,vvarbb]{newtxmath}
\usepackage[cal=euler,frak=euler]{mathalpha} 
\usepackage[dvipsnames]{xcolor}

\usepackage[margin=3cm]{geometry} 

\input xy
\xyoption{all}

\usepackage[colorlinks]{hyperref}
\hypersetup{linkcolor=blue, citecolor=green!75!black}

\title{Characterization of principal bundles:\\
the noncommutative algebraic case}

\author{William J. Ugalde
\\[12pt]
{\small
Escuela de Matemática, Universidad de Costa Rica,
11501 San José, Costa Rica%
\thanks{Email: william.ugalde@ucr.ac.cr}}}

\DeclareMathOperator{\Alg}{Alg}     
\DeclareMathOperator{\End}{End}     
\DeclareMathOperator{\Hom}{Hom}     
\DeclareMathOperator{\pr}{pr}       
\DeclareMathOperator{\tsum}{{\textstyle\sum}} 

\newcommand{\al}{\alpha}            
\newcommand{\bt}{\beta}             
\newcommand{\Dl}{\Delta}            
\newcommand{\dl}{\delta}            
\newcommand{\eps}{\varepsilon}      
\newcommand{\ga}{\gamma}            
\newcommand{\sg}{\sigma}            
\renewcommand{\th}{\theta}          
\newcommand{\vf}{\varphi}           

\newcommand{\bC}{\mathbb{C}}        
\newcommand{\bF}{\mathbb{F}}        
\newcommand{\bN}{\mathbb{N}}        
\newcommand{\bQ}{\mathbb{Q}}        
\newcommand{\bR}{\mathbb{R}}        
\newcommand{\bS}{\mathbb{S}}        
\newcommand{\bZ}{\mathbb{Z}}        

\newcommand{\g}{\mathfrak{g}}       
\newcommand{\gsl}{\mathfrak{sl}}    

\newcommand{\can}{\mathrm{can}}     
\newcommand{\co}{\mathrm{co}}       
\newcommand{\copp}{{\mathrm{cop}}}  
\newcommand{\id}{\mathrm{id}}       
\newcommand{\inv}{\mathrm{inv}}     
\newcommand{\linspan}{\mathrm{span}} 
\newcommand{\opp}{{\mathrm{op}}}    

\newcommand{\sA}{\mathcal{A}}       
\newcommand{\sB}{\mathcal{B}}       
\newcommand{\sC}{\mathcal{C}}       
\newcommand{\sF}{\mathcal{F}}       
\newcommand{\sH}{\mathcal{H}}       
\newcommand{\sO}{\mathcal{O}}       
\newcommand{\sU}{\mathcal{U}}       

\renewcommand{\geq}{\geqslant}      
\newcommand{\hookto}{\hookrightarrow} 
\newcommand{\isom}{\simeq}          
\newcommand{\lt}{\triangleright}    
\newcommand{\otto}{\leftrightarrow} 
\newcommand{\ox}{\otimes}           
\newcommand{\rt}{\triangleleft}     
\newcommand{\rx}{\rtimes}           
\newcommand{\x}{\times}             
\renewcommand{\.}{\cdot}            
\renewcommand{\:}{\colon}           

\newcommand{\hatox}{\mathrel{\widehat\otimes}} 
\newcommand{\longto}{\mathop{\longrightarrow}\limits} 
\newcommand{\oxyox}{\otimes\cdots\otimes} 
\newcommand{\stroke}{\mathbin|}     
\newcommand{\wt}{\widetilde}        
\newcommand{\xyx}{\times\cdots\times} 

\newcommand{\half}{{\mathchoice{\thalf}{\thalf}{\shalf}{\shalf}}}

\newcommand{\shalf}{{\scriptstyle\frac{1}{2}}} 
\newcommand{\thalf}{\tfrac{1}{2}}   

\newcommand{\mot}[1]{\enspace\text{#1}\enspace} 
\newcommand{\set}[1]{\{\,#1\,\}}    
\newcommand{\Set}[1]{\bigl\{\,#1\,\bigr\}} 
\newcommand{\SET}[1]{\biggl\{\,#1\,\biggr\}} 
\newcommand{\val}[1]{\langle{#1}\rangle} 
\newcommand{\word}[1]{\quad\text{#1}\quad} 

\newcommand{\braket}[2]{\langle#1\stroke#2\rangle} 
\newcommand{\duo}[2]{\langle#1,#2\rangle} 
\newcommand{\switch}[2]{{}_{#1}\mathord{\frown}_{#2}} 

\newcommand{\twobytwo}[4]{\begin{pmatrix} 
   #1 & #2 \\ #3 & #4 \end{pmatrix}}
\newcommand{\twobytwodet}[4]{\begin{vmatrix} 
   #1 & #2 \\ #3 & #4 \end{vmatrix}}

\makeatletter
\renewcommand{\section}{\@startsection{section}{1}{\z@}%
							 {-3.25ex \@plus -1ex \@minus -.2ex}%
							 {1.5ex \@plus.2ex}%
							 {\normalfont\large\bfseries}}
\renewcommand{\subsection}{\@startsection{subsection}{2}{\z@}%
							 {-3.25ex \@plus -1ex \@minus -.2ex}%
							 {1.5ex \@plus .2ex}%
							 {\normalfont\normalsize\bfseries}}
\renewcommand{\@dotsep}{200} 
\makeatother

\numberwithin{equation}{section}        

\theoremstyle{plain}
\newtheorem{theorem}{Theorem}[section]  
\newtheorem{proposition}[theorem]{Proposition} 
\newtheorem{lemma}[theorem]{Lemma}      
\newtheorem{corollary}[theorem]{Corollary} 

\theoremstyle{definition}
\newtheorem{example}[theorem]{Example}  
\newtheorem{exercise}{Exercise}[section] 

\theoremstyle{remark}  
\newtheorem*{remark}{Remark}

\DeclareRobustCommand{\qef}{
  \ifmmode 
  \else \leavevmode\unskip\penalty9999 \hbox{}\nobreak\hfill
  \fi
  \quad\hbox{\qefsymbol}}
\newcommand{\qefsymbol}{$\circledcirc$} 

\newcommand{\hideqef}{\renewcommand{\qef}{}} 

\DeclareRobustCommand{\qeo}{
  \ifmmode 
  \else \leavevmode\unskip\penalty9999 \hbox{}\nobreak\hfill
  \fi
  \quad\hbox{\qeosymbol}}
\newcommand{\qeosymbol}{$\lozenge$} 

\begin{document}

\maketitle

\begin{abstract}
We review Hopf--Galois extensions, in particular faithfully flat ones,
accepted to be the noncommutative algebraic dual of a principal
bundle. We also make a short digression into how quantum groups relate
to Hopf–-Galois extensions. Several examples are given, in order to
provide a satisfactory understanding of each topic.
\end{abstract}

\begin{small}



\tableofcontents 

\end{small}


\section{Introduction} 
\label{sec:intro}

The Serre--Swan theorem gives a functorial connection between finitely
generated projective modules and smooth sections over a vector bundle
(see \cite[Chap.~2]{PepeJoeHector}, for example). It gives a
noncommutative version of a `commutative' object: a vector bundle or
its sections gets replaced by a finitely generated projective module.
However, when it comes to principal bundles, there is no a concrete
way of passage, regardless of a consolidated and nowadays accepted
interpretation of what is a `noncommutative principal bundle', namely,
a Hopf--Galois extension. The search for the most adequate substitute
or generalization of a principal bundle in noncommutative geometry has
had two main approaches: the noncommutative algebraic approach and the
$C^*$-algebraic formulation. This classification of the development of
the subject in two components coincides with that given
in~\cite{Tobolski}. These notes comprise a review of Hopf--Galois
extension, accepted or understood as a noncommutative principal
bundle, from the algebraic point of view.

The driving idea is to extend the injectivity of the canonical map
(which is trivially surjective)
$\can \: X \x G \to X \x_G X : (x,g) \mapsto (x,x\.g)$ associated to
an action, as follows. Its injectivity is equivalent (in the compact
or locally compact situations) to the surjectivity of the homomorphism
$\can^* \: C(X \x_G X) \to C(X \x G) : f \mapsto f \circ \can$, which
is now a map acting between algebras. The choice of the type of
algebra under study shows the kind of extension one wants to achieve.
Two of the main possibilities are Hopf algebras (related to the
replacement of the group $G$ for a more general structure), and
$C^*$-algebras, which was suggested by Connes in the direction of
$C(X \x G)$ being replaced by $C(X) \rx G$. Provided that the base
manifold $X$ \textit{and} the symmetry Lie group $G$ are both compact,
one can work algebraically, since the matrix elements of irreducible
representations of~$G$ form a Hopf algebra. Out of this compactness
assumption on $G$, many things can go out of control and open a whole
different universe. This is what is meant in \cite{Tobolski} by saying
``the setting of Hopf--Galois extensions cannot be used to obtain a
generalization of principal bundles beyond the compact case''.
However, simple examples of Hopf fibrations such as the principal
bundle $(\bS^7,\sg,\bS^4,SU(2))$ show that the compact case is
interesting enough to work out its algebraic formulation carefully.

Before ``going noncommutative'', one should first understand the
commutative case in algebraic terms. The following interpretation goes
along the lines of~\cite{BaumHMS}. If $G$ is a compact Lie group, let
$H = \sO(G)$ be the algebra of (complex-valued) \textit{representative
functions} on~$G$, which consists of linear combinations of the matrix
elements $a(g) = \braket{\eta}{\rho(g)\xi}$ of irreducible
representations $\rho$ of~$G$. These are smooth functions, so that
$\sO(G) \subset C^\infty(G)$. Under the coproduct
$\Dl \: H \to H \ox H$, counit $\eps \: H \to \bC$ and antipode
$S \: H \to H$, given by%
\footnote{The definition of $\Dl$ implicitly uses the relation
$\sO(G) \ox \sO(G) \isom \sO(G \x G)$.}
\[
\Dl a(g,h) := a(gh),  \qquad
\eps(a) := a(1),  \qquad
Sa(g) := a(g^{-1}),
\]
the algebra $H$ is also a coalgebra for which $\Dl$ and $\eps$ are
algebra homomorphisms; and indeed $H$ is a \textit{Hopf algebra} (a
bialgebra with antipode).

Given a principal $G$-bundle $\sg \: P \to M$, consider the algebra 
homomorphism 
$$
\dl \: C^\infty(P) \to C^\infty(P \x G)
\isom C^\infty(P) \hatox C^\infty(G)
$$
given by $\dl f(u,g) := f(ug)$. The tensor product $\hatox$ that
appears here is the so-called \textit{projective} tensor product of
locally convex topological vector spaces \cite{Treves67}, not the
\textit{algebraic} tensor product (over~$\bC$) that is written with
$\ox$ unornamented. When $P = G$ and $M$ is a point, this extends the
coproduct $\Dl$ on $\sO(G)$ to the larger space $C^\infty(G)$, but
this larger space is generally not itself a Hopf algebra. The
paper~\cite{BaumHMS} introduced the following concept, which allows to
use the algebraic instead of the projective tensor product. With $G$
compact, let $\sC^\infty(P)$ be the subalgebra of $C^\infty(P)$ given
by
$$
\sC^\infty(P)
:= \set{f \in C^\infty(P) : \dl f \in C^\infty(P) \ox \sO(G)}.
$$
$\sC^\infty(P)$ is called the algebra of \textit{upright smooth
functions} on~$P$. It turns out that if $\dl f = \sum_j f_j \ox a_j$
for $f \in \sC^\infty(P)$, then each $f_j$ lies in $\sC^\infty(P)$
also. Thus, by restriction, one obtains an algebra homomorphism
\[
\dl \: \sC^\infty(P) \to \sC^\infty(P) \ox \sO(G),
\]
which is a \textit{right coaction} of $H = \sO(G)$ on
$A := \sC^\infty(P)$, namely,
$$
(\dl \ox \id_H)\,\dl = (\id_A \ox \Dl)\,\dl
: \sC^\infty(P) \to \sC^\infty(P) \ox \sO(G) \ox \sO(G),
$$
as well as $(\id_A \ox \eps)\,\dl = \id_A$. The last display reduces
to the identity $f((ug)g') \equiv f(u(gg'))$, so it simply encodes the
right action of $G$ on~$P$.

In these notes we explore the venue that goes through Hopf algebras to
understand what is a Hopf--Galois extension, in particular a
faithfully flat one, accepted to be the noncommutative algebraic dual
of a principal bundle.

The first step we shall take is to address the meaning of a
`noncommutative space'. Out of many inspiring references, the notes by
Kassel \cite{Kassel1} help to build the discussion in
Section~\ref{sec:NCspaces} in which we try to motivate the
understanding of a noncommutative algebra as a replacement for an
ordinary space. Even though this could warrant replacing every
`commutative' object by a noncommutative substitute, it is important
to keep in mind situations like \cite{Ellwood}, which presents an
example ``of a commutative base and a commutative total space but it
is the way in which the fibers are entangled that produces a
noncommutative bundle.'' Meaning that the real greatness of
noncommutative geometry comes from how it intertwines traditional
spaces of analysis and geometry with a far-reaching point of view in
which commutative and noncommutative objects merge.

There are many possibilities to study, with more or less generality, a
Hopf--Galois extension. We chose to study extensions of algebras by
bialgebras based on a commutative associative unital ring. In
Section~\ref{sec:HG-for-dummies}, we assemble all the required
concepts to arrive at the definition of a Hopf--Galois extension. In
Section~\ref{sec:HG-extensions} we study the concept of a Hopf algebra
and its relation with Hopf--Galois (and faithfully flat Hopf--Galois)
extensions, nowadays considered as noncommutative principal bundles
\cite{AschieriFioresiLatini, AschieriBieliavskyPaganiSchenkel,
AschieriLandiPagani, AschieriLandiPagani2, BrzezinskiSzymanski,
Zanchettin}.

Section~\ref{sec:Qgroups-HG-ext} makes a short digression into how
quantum groups relate to Hopf–Galois extensions. In particular, for
every proper quantum group (a noncommutative noncocommutative Hopf
algebra) $H$, $k \subseteq H$ is a Hopf--Galois extension by the
bialgebra $H$. The last part of this section gives a short account of
what is accepted as a `quantum principal bundle': first, based of the
original pioneering works \cite{Brzezinski94,BrzezinskiMajid}, and
then addressing the recent papers \cite{AschieriLandiPagani,
AschieriBieliavskyPaganiSchenkel}.

Throughout the text, we offer several examples, aiming to provide a
satisfactory understanding of each topic.

In subsection~\ref{ssc:tracking-the-idea} we give a brief account, in
chronological order, of some of the more appealing references about
the concept of a Hopf--Galois extension as a noncommutative
replacement for a principal bundle. Out of all sources available, we
cite those that we were fortunate enough to come across.


\subsection{Tracking down the idea of a Hopf--Galois extension} 
\label{ssc:tracking-the-idea}

In the beginning, \cite{Schneider1990} stands as the point of
departure providing the main tool for the passage from modules to
comodules and from algebras to coalgebras: ``Once there is a general
theorem like Theorem~I, one can formally dualize it, i.e., turn all
the arrows around.''
One of the first to spot the potential importance of that reference
was \cite{Durdevic1993} which proposes ``a quantum generalization of
the theory of principal bundles, in which quantum groups play the role
of structure groups, and quantum spaces the role of base manifolds.''
And in \cite{Durdevic95}: 
``In the framework of noncommutative differential geometry, the
Hopf--Galois extensions are understandable as analogues of principal
bundles. [\dots] The freeness of the action of the structure group is
equivalent to the injectivity of the introduced canonical map.''
More involved in the subject is Brzeziński \cite{Brzezinski94}, who
constructed a noncommutative version of a translation map by
dualization and showed that the notion of a quantum principal bundle
is equivalent to the existence of this generalized translation map.

In \cite{BrzezinskiMajid}, Brzeziński and Majid proposed
noncommutative or `quantum' versions of homogeneous spaces, analogues
of trivial principal and associated bundles and of gauge
transformations on trivial quantum vector bundles, using the
dualization process. They chose as algebras Hopf algebras. Around the
same early years, Pflaum, exploring the categorical side of the
subject \cite{Pflaum}, studied noncommutative vector and principal
bundles within the language of quantum spaces. Also, Hajac
\cite{HajacThesis} analyzed a class of `strong' connections on quantum
principal bundles that provide a link between the two approaches to
noncommutative differential geometry based on quantum principal
bundles and on projective modules.

The first years of the viewpoint of considering a Hopf--Galois
extension as a dual notion of a principal bundle (up to 1999) are
completed with \cite{BrzezinskiMajid2, Schauenburg1998, Brzezinski99,
HajacMajid}. During the first ten years of the new century, the
following references stand out: \cite{Schauenburg2002, Grunspan,
SchauenburgFIC, SchauenburgSchneider, BrzezinskiHajac2004, BaumHMS,
Bohm07, BrzezinskiZielinski, BrzezinskiHajac2009}. Among the references
developing specific examples we cite~\cite{LandiSuijlekom}.

A refreshed effort in this direction is found in studies of the
invertibility of the canonical map. Of special interest is the
sequence \cite{AschieriBieliavskyPaganiSchenkel, AschieriLandiPagani,
AschieriLandiPagani2}, together with \cite{BrzezinskiSzymanski}.

As mentioned in \cite{Tobolski}, the noncommutative algebraic setting
(of Hopf--Galois extensions) has limitations when trying to extend
principal bundles beyond the compact case. There is where the
$C^*$-algebraic approach enters. Even though we do not walk this path
in the present work, it is worth mentioning a few significant
references. One of the earliest is \cite{Rieffel1} that studies proper
actions in this setting (Ellwood \cite{Ellwood} contains an equivalent
and newer perspective). Then \cite{BudzinskiKondracky} introduces the
notion of locally trivial quantum principal bundles using compact
quantum spaces (unital $C^*$-algebras) for the base and the total
space. In \cite{BaumDeComerHajac} it is proved that an action in the
$C^*$-algebraic setting is free if and only if the canonical map
(obtained from the underlying Hopf algebra of the compact quantum
group) is an isomorphism, see also \cite{BaumHajac}. The references
\cite{Wagner12, Wagner13} contain an important attempt to translate
the geometric concept of principal bundles to noncommutative
differential geometry, by means of dynamical systems and
representations of the corresponding transformation groups. Last, an
excellent account of the development of quantum groups is given
in~\cite{VanDaele}.


\section{Noncommutative spaces} 
\label{sec:NCspaces}

A driving idea in noncommutative geometry is to replace a space $X$ by
a commutative algebra of functions $\sO(X)$ from $X$ to $\bC$, in such
a way that knowledge of $\sO(X)$ allows one to recover $X$ up to an
isomorphism depending on the category. One can then replace the
algebra $\sO(X)$ by a possibly noncommutative algebra $\sA$ and
consider it as the algebra of functions for a putative noncommutative
space. As stated by Pareigis~\cite{Pareigis}, ``the \textit{dream of
geometry} would be that these two categories, the category of
(certain) spaces and the category of (certain) algebras, are dual to
each other.''

\begin{example} 
\label{eg:first}
Consider as a first example a \textit{finite set}
$X = \{x_1,\dots,x_n\}$ and for $\sO(X)$ take the commutative algebra
$\set{f \: X \to \bC}$, with point-wise operations
$(f + c g)(x) = f(x) + c g(x)$, $(fg)(x) = f(x) g(x)$ and unit the
constant function $1$ over~$X$.

Denoting $\dl_x(y) = \dl_{x,y}$ for any $x \in X$: the function that
is equal to $1$ if $y = x$ and $0$ otherwise, then any $f \in \sO(X)$
is written as $f = f(x_1) \dl_{x_1} +\cdots+ f(x_n) \dl_{x_n}$, so
that $\{\dl_{x_1}, \dots, \dl_{x_n}\}$ is a basis for
$\sO(X) \equiv \bC^n$, where $\dim \sO(X) = n$ is precisely the
cardinality of $X$.
\end{example}

\begin{example} 
\label{eg:second}
As a second example, let $X = \set{(x_1,\dots,x_n) \in \bC^n
: P(x_1,\dots,x_n) = 0}$ for every $P \in \sF\}$ be the
\textit{algebraic variety} corresponding the family of polynomials
$\sF \subset \bC[X_1,\dots,X_n]$. To this space one associates the
commutative \textit{coordinate algebra}
$\sO(X) = \bC[X_1,\dots,X_n]/I_\sF$, where $I_\sF$ is the ideal
generated by~$\sF$. Evidently, $X$ is determined as the set of points
$(x_1,\dots,x_n) \in \bC^n$ that are common solutions of the equations
$P(x_1,\dots,x_n) = 0$.

The previous example of a finite set $X = \{x_1,\dots,x_n\}$, can be
recovered from the one-element family
$\sF = \{(X_1 - x_1) \cdots (X_n - x_n)\}$. Also, with $\sF = \{0\}$,
$\sO(X) = \bC[X_1,\dots,X_n]$ gives $X = \bC^n$. In particular,
$\sO(X) = \bC[X_1,X_2]$ gives $X = \bC^2$. 
\end{example}

Let us replace this commutative algebra with a noncommutative one, and
see what sort of object we recover. This example is called in the
literature the \textit{quantum plane}.

\begin{example} 
\label{eg:quantum-plane}
Inside the noncommutative algebra $\bC\val{X_1,X_2}$ of non-commuting
polynomials in two variables $X_1$ and~$X_2$, define for any
$q \in \bC \setminus \{0\}$ the ideal $I_q$ generated by the
polynomial $X_1 X_2 - q X_2 X_1$ and the noncommutative algebra
$\bC_q[X_1,X_2] := \bC\val{X_1,X_2}/I_q$. Evidently, if $q = 1$, we
recover $\bC[X_1,X_2]$. One says that $\bC_q[X_1,X_2]$ is a
\textit{one-parameter, noncommutative deformation (or quantization)}
of $\bC[X_1,X_2]$.

The set $\set{X_1^i X_2^j : i,j \geq 0}$ forms a basis for
$\bC_q[X_1,X_2]$. In order to identify the space~$X$ for which
$\sO(X) = \bC_q[X_1,X_2]$, we need some sort of procedure that we do
not go into here, although we make a little digression in
subsection~\ref{ssc:non-existing-map}. Nevertheless, the point is that
the set $\set{(x_1,0), (0,x_2) \in \bC^2 : x_1,x_2 \in \bC}$ is
intuitively contained in our desired $X$, because its points are zeros
of $X_1X_2 - q X_2X_1$ as with algebraic varieties, but how do we know
if this is precisely~$X$? One possibility is to consider the
\textit{characters} of the algebra $\bC_q[X_1,X_2]$. In general, for
an algebra $\sA$ its sets of characters is
\[
\Alg(\sA,\bC) := \set{\chi \in \Hom_\bC(\sA,\bC)
: \chi(aa') = \chi(a)\chi(a') \mot{and} \chi(1) = 1;\ a,a'\in \sA}.
\]
In the case $\sA = \bC_q[X_1,X_2]$, each character gets determined by
its complex values $\chi(X_1) = x_1$ and $\chi(X_2) = x_2$, and this
induces a bijection $\Alg(\bC_q[X_1,X_2],\bC) 
\equiv \set{(x_1,x_2) \in \bC^2 : x_1x_2 - qx_2x_1 = 0}$. As a result,
the ``noncommutative space'' known as the quantum plane is indeed
$\set{(x_1,0), (0,x_2) \in \bC^2 : x_1,x_2 \in \bC}$.

A word of caution is needed. Noncommutative geometers usually call a
(noncommutative) algebra a ``noncommutative space'', leaving backstage
the possible existence (or not) of a space represented by the given
algebra. Thus, in the case of the quantum plane, what really matters
is the algebra $\bC_q[X_1,X_2]$ itself, and not the union of the two
coordinates axes in the complex plane.
\end{example}

Now that we have settled, at least intuitively, this first level of
replacing spaces by algebras of functions, let us pass to a more
involved situation. We wish to transfer properties of the space~$X$ to
its associated algebra~$\sO(X)$.

As before, consider two sets $X$ and $Y$ and their Cartesian product
$X \x Y$ with its two projections $\pi_1 \: X \x Y \to X$ and
$\pi_2 \: X \x Y \to Y$, with associated algebra morphisms
$\pi_i^* \: \sO(X) \to \sO(X \x Y) : f \mapsto f \circ \pi_i$ or any
$f \in \sO(X)$ or $\sO(Y)$, and $i = 1,2$.

Here the vector spaces (underlying the algebras) are generated by the
functions $\dl_x$, $\dl_y$ and $\dl_{(x,y)}$, with $x \in X$,
$y \in Y$. The product $X \x Y$ (with its two projections) enjoys the
following universal property: if $Z$ is a third set and
$h_1 \: Z \to X$, $h_2 : Z \to Y$ are maps, there is a unique map
$h \: Z \to X \x Y$ such that $\pi_1 \circ h = h_1$ and
$\pi_2 \circ h = h_2$, namely $h = (h_1,h_2)$.%
\footnote{This universal property for the Cartesian product holds in
many other situations using the appropriate maps.}
These two relations imply
$h^* \circ \pi_1^* = h_1^* : \sO(X) \to \sO(Z)$ and 
$h^* \circ \pi_2^* = h_2^* : \sO(Y) \to \sO(Z)$.
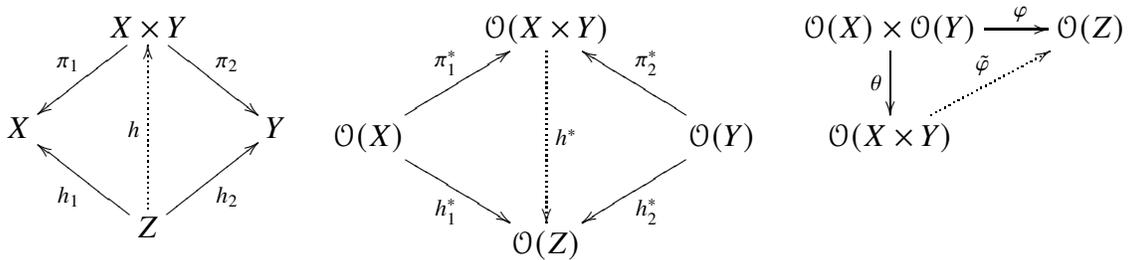
\[
\xymatrix{
& X \x Y \ar[dl]_{\pi_1} \ar[dr]^{\pi_2} &
\\
X && Y
\\
& Z \ar[ul]^{h_1} \ar@{.>}[uu]^{h} \ar[ur]_{h_2} & }
\quad
\xymatrix{
& \sO(X \x Y) \ar@{.>}[dd]^{h^*} &
\\
\sO(X) \ar[ur]^{\pi_1^*} \ar[dr]_{h_1^*} &&
\sO(Y) \ar[ul]_{\pi_2^*} \ar[dl]^{h_2^*} 
\\
& \sO(Z) & }
\quad
\xymatrix{
\sO(X) \x \sO(Y) \ar[r]^(0.65)\vf \ar[d]_\th & \sO(Z)
\\
\sO(X \x Y) \ar@{.>}[ur]^{\tilde\vf} & }
\]
The morphism $\th \: \sO(X) \x \sO(Y) \to \sO(X \x Y)$, sending each
pair $(f,g)$ to the map $(x,y) \mapsto f(x) g(y)$, for $x \in X$,
$y \in Y$, $f \in \sO(X)$, $g \in \sO(Y)$, allows to factor each
bilinear morphism $\vf \: \sO(X) \x \sO(Y) \to \sO(Z)$ in a unique way
as $\wt\vf \circ \th = \vf$, with $\wt\vf \: \sO(X \x Y) \to \sO(Z)$.
To prove it, it is enough to consider basic elements: for $x \in X$,
$y \in Y$ take $\wt\vf(\dl_{(x,y)}) := \vf(\dl_x,\dl_y)$. Because
$\dl_{(x,y)}(x',y') = \dl_x(x') \dl_y(y')$, we can write
$\th(\dl_x,\dl_y) := \dl_{(x,y)}$ to obtain
$\wt\vf \circ \th(\dl_x,\dl_y) = \wt\vf(\dl_{(x,y)})
= \vf(\dl_x,\dl_y)$. The result is that $\sO(X \x Y)$ enjoys the same
universal property as $\sO(X) \ox \sO(Y)$, and thus we identify both
algebras.

Taking advantage of this identification, for a pair of maps
$f \: X \to Z$ and $g \: Y \to W$, the product map
$f \x g \: X \x Y \to Z \x W$ produces the algebra morphism 
$$
(f \x g)^* \: \sO(Z \x W) \equiv \sO(Z) \ox \sO(W) \to \sO(X \x Y)
\equiv \sO(X) \ox \sO(Y),
$$
which we can now write as $(f \x g)^* = f^* \ox g^*$.
 
It is often useful to study the algebra of functions on a group $G$
instead of the group itself. To motivate some of the work ahead,
consider now the algebra of functions $\sO(G)$ (also denoted $F^G$) of
all $f \: G \to \bC$ with pointwise operations
$(f + l)(g) := f(g) + l(g)$ and $(fl)(g) := f(g)\,l(g)$. Again it has
a basis $\set{\dl_g : g \in G}$. If $G$ is finite,
$\sO(G) \equiv \bC^{|G|}$ with dimension~$|G|$. We want to transfer
the features of $G$ to properties of~$\sO(G)$.

Observe that the (pointwise) multiplication in $\sO(G)$ translates
into a multiplication map 
$\mu \: \sO(G) \ox \sO(G) \to \sO(G) : f \ox f' \mapsto ff'$, extended
bilinearly.

The group $G$ has an associative product $m$, an inverse map $\inv$
and an identity element $1_G$. The product map $m \: G \x G \to G$
induces an algebra morphism
$m^* \: \sO(G) \to \sO(G \x G) \equiv \sO(G) \ox \sO(G)$ such that
$m^*(f)(g,g') = f \circ m(g,g') = f(gg')$. The associativity of~$m$
can be written as $m \x \id = \id \x m : G \x G \x G \to G$, which
implies
\[
m^*\ox \id^* = \id^* \ox m^* : \sO(G) \to \sO(G \x G \x G)
\equiv \sO(G) \ox \sO(G) \ox \sO(G),
\]
where $\id^*(f) = f \circ \id = f$, meaning that 
$\id_G^* = \id_{\sO(G)}$. Later on, we shall denote
$\Dl = m^* \: \sO(G) \to \sO(G) \ox \sO(G)$ and call it
\textit{comultiplication}. It is a homomorphism of algebras, since
$m^*(ff') = m^*(f) m^*(f')$ and 
$m^*(1_{\sO(G)}) = 1_{\sO(G)} \ox 1_{\sO(G)} = 1_{\sO(G) \ox \sO(G)}$.

The inverse map in $G$, $\inv \: g \mapsto g^{-1}$, produces the
algebra morphism $\inv^* \: \sO(G) \to \sO(G)$ satisfying
$\inv^*(f)(g) = f \circ \inv(g) = f(g^{-1})$. The map $S = \inv^*$ is
an example of an \textit{antipode}.

To deal with the identity element, consider first the diagonal map (or
canonical inclusion) $i \: G \to G \x G : g \mapsto (g,g)$ with
associated algebra morphism $i^* \: \sO(G \x G) \to \sO(G)$ given by
$i^*(h)(g) = h(g,g)$ for any $h \in \sO(G \x G)$. That $1_G$ is the
identity element for the group $G$ is equivalent to the fact that the
map $m \circ (\id \x \inv) \circ i$ is a constant map from $G$ to the
set $\{1_G\}$, with $\sO(\{1_G\}) \equiv \bC$ by $z(1_G) = z$. As a
result, $\eta := (m \circ (\id \x \inv) \circ i)^* \: \bC \to \sO(G)$
is such that
$\eta(z)(g) = z \circ m \circ (\id \x \inv) \circ i(g) = z(1_G) = z$.
The morphism $\eta \: \bC \to \sO(G) : z \mapsto [g \mapsto z]$, for
every $g \in G$, will be called the \textit{unit} of $\sO(G)$ since
$\eta(1)(f) = f$ in this algebra.

There is also a trivial inclusion $j \: \{1_G\} \to G$ which implies 
$\eps := j^* \: \sO(G) \to \sO(\{1_G\}) \equiv \bC$, with
$\eps(f) = f(1_G)$, called the \textit{counit} of $\sO(G)$. Again,
$\eps$ is an algebra homomorphism:
$\eps(ff') = \eps(f) \eps(f')$ and $\eps(1_{\sO{G}}) = 1$. The
entanglement of these morphisms is represented in the following
diagrams:
\[
\xymatrix{\ar @{} [dr]
\sO(G) \ox  \sO(G) \ar[r]^(.6)\mu \ar[d]_{\eps \ox \eps} &
\sO(G) \ar[d]^\eps  \\
\bC \isom \bC \ox  \bC  \ar[r]^(.65){\cdot} & \bC}
\quad
\xymatrix{\ar @{} [dr]
\sO(G) \ox  \sO(G) \ar[r]^\mu \ar[d]_{\Dl \ox \Dl} & \sO(G) \ar[d]^\Dl
\\
\sO(G) \ox  \sO(G) \ox  \sO(G) \ox  \sO(G) \ar[r]^(.65){\mu_\ox} &
\sO(G)\ox  \sO(G),}
\]
where $\mu_\ox = \mu \ox \mu \circ \switch{2}{3}$ (flipping the second
and third factors as an initial step) is the multiplication on
$\sO(G) \ox \sO(G)$. These two relations will be referred to as
\textit{compatibility conditions}. The third relation shows the role
of the antipode as the ``convolution inverse'' for the algebra
morphism $\eta \circ \eps$. Its meaning will be made precise later on;
for now, it is enough to observe that
\[
(S \ox \id_{\sO(G)}) \circ \Dl(f)
= (\inv^* \ox \id^*_G) \circ m^*(f)
= \bigl( m \circ (\inv \x \id_G) \bigr)^*(f) = f(1_G)
\]
for every $f \in \sO(G)$.

The package made up of an algebra $H$, four algebra morphisms
$(\mu,\eta,\Dl,\eps)$, and one antihomomorphism $S$ satisfying the
three properties listed above will be called a \textit{Hopf algebra}.
In the case of $\sO(G)$, the antihomomorphism property of~$S$ gets
hidden by the commutativity of the algebra.

\medskip

Let us see an example of a group with an extra structure.

\begin{example} 
\label{eg:SL2C}
The \textit{special linear group} $SL_2(\bC)$ is an algebraic variety
as well.%
\footnote{With a little more elaboration, the same will hold for
$SL_n(\bC)$.}
This is the group of all invertible $2 \x 2$ complex matrices of
determinant~$1$, with its usual matrix multiplication. We can write
\[
\twobytwodet{a}{b}{c}{d} = ad - bc = 1,
\]
and consider its (commutative) \textit{coordinate algebra} 
\[
SL(2) = \sO(SL_2(\bC)) := \bC[a,b,c,d]/(ad - bc -1),
\]
with coproduct $\Dl \: SL(2) \to SL(2) \ox SL(2)$ written in the form
\[
\Dl \twobytwo{a}{b}{c}{d}
= \twobytwo{a}{b}{c}{d} \ox  \twobytwo{a}{b}{c}{d}
\equiv \twobytwo{a \ox a + b \ox c}{a \ox b + b \ox d}
{c \ox a + d \ox c}{c \ox b + d \ox d}.
\]
Its counit and antipode are
\[
\eps \twobytwo{a}{b}{c}{d} = \twobytwo{1}{0}{0}{1}, \qquad
S \twobytwo{a}{b}{c}{d} = \twobytwo{a}{b}{c}{d}^{-1}
= \twobytwo{d}{-b}{-c}{a}.
\]
In this example is easier to see that $S$ is in general an
antihomomorphism.

All these relations are compact versions of
$\Dl(a) = a \ox a + b \ox c$, $\eps(a) = 1$, $S(a) = d$, and so on.
This algebra is not cocommutative, meaning $\tau \circ \Dl \neq \Dl$,
where $\tau(x \ox y) := y \ox x$.
\end{example}

\begin{example} 
\label{eg:SLq2}
A noncommutative deformation of $SL(2)$ is given by the algebra
$SL_q(2)$, with $q \in \bC \setminus \{0\}$. It is the algebra with
four generators $a,b,c,d$ subject to the relations:
\[
ba = q ab,  \quad  ca = q ac,  \quad  db = q bd,  \quad  dc = q cd,
\]
\[
bc = cb,  \quad  da = ad + (q - q^{-1}) bc,  \quad  ad - q^{-1}bc = 1.
\]
With $q = 1$ we recover $SL(2)$ which is evidently non-isomorphic to
any other $SL_q(2)$, because one is commutative and the other is not.
For a motivation on how to select these relations, see \cite{Manin,
KlimykSchmudgen, KasselBook}.

The algebra $SL_q(2)$ becomes a Hopf algebra with the same
comultiplication and counit of $SL(2)$, and with the antipode
\[
S_q \twobytwo{a}{b}{c}{d} = \twobytwo{d}{-qb}{-q^{-1}c}{a},
\word{which satisfies}
S_q \twobytwo{a}{b}{c}{d} \twobytwo{a}{b}{c}{d}
= \twobytwo{1}{0}{0}{1}. 
\]
Note that $SL_q(2)$ is neither commutative nor cocommutative. Also,
observe that
\[
S_q^n \twobytwo{a}{b}{c}{d}
= \twobytwo{\half(1 + (-1)^n)a + \half(1 - (-1)^n)d}
{(-q)^n b}{(-q^{-1})^n c}{\half(1 - (-1)^n)a + \half(1 + (-1)^n)d}.
\]
In the even case, it follows that
\[
S_q^{2n} \twobytwo{a}{b}{c}{d} = \twobytwo{q^n}{0}{0}{q^{-n}}
\twobytwo{a}{b}{c}{d} \twobytwo{q^{-n}}{0}{0}{q^n},
\]
and in the odd case
\[
S_q^{2n+1} \twobytwo{a}{b}{c}{d} = \twobytwo{q^n}{0}{0}{q^{-n}}
\twobytwo{d}{-qb}{-q^{-1}c}{a} \twobytwo{q^{-n}}{0}{0}{q^n}.
\eqno \qef
\]
\hideqef
\end{example}

In Section~\ref{sec:Qgroups-HG-ext} we will come back to $SL_q(2)$ 
as an example of a \textit{quantum group}.

\medskip

The last motivational example that we want to mention extends the
action of a group on a space to the noncommutative interpretation we
just introduced.

\begin{example} 
\label{eg:action-on-SL2C}
The special linear group $SL_2(\bC)$ acts on the vector space $\bC^2$
by matrix multiplication, $\al \: SL_2(\bC) \x \bC^2 \to \bC^2$.
Because $\al$ is an action, it satisfies
\[
\al(\al \ox \id_{SL_2(\bC)}) = \al(\mu \ox \id_{\bC^2})
: SL_2(\bC) \ox SL_2(\bC) \ox \bC^2 \to \bC^2, 
\]
and $\al\bigl( 1_{SL_2(\bC)} \ox \id_{\bC^2} \bigr)
= \id_{\bC^2} \in \End(\bC^2)$. This action produces an algebra
morphism
\[
\al^* \: \sO(\bC^2) \equiv \bC[X_1,X_2]
\to \sO(SL_2(\bC) \x \bC^2) \equiv SL(2) \ox \bC[X_1,X_2]
\]
that satisfies
\begin{align*}
(\al^* \ox \id_{SL(2)}) \al^* &= (\Dl \ox \id_{\bC[X_1,X_2]}) \al^*
: \bC[X_1,X_2] \to SL(2) \ox SL(2) \ox \bC[X_1,X_2],
\\
(\eps \ox \id_{\bC^2})\al^* &= 1_\bC \ox \id_{\bC[X_1,X_2]}
\equiv \id_{\bC[X_1,X_2]} \in \End(\bC[X_1,X_2]),
\end{align*}
with $\bC[X_1,X_2] \ox \bC \equiv \bC[X_1,X_2]$. We say that $\al^*$
is a \textit{left coaction} of $SL(2)$ on $\bC[X_1,X_2]$. Actually,
\[
\binom{\al^*(X_1)}{\al^*(X_2)} = \al^* \binom{X_1}{X_2}
= \twobytwo{a}{b}{c}{d} \ox \binom{X_1}{X_2}
= \binom{a \ox X_1 + b \ox X_2}{c \ox X_1 + d \ox X_2},
\]
meaning that $\al^*(X_1) = a \ox X_1 + b \ox X_2$ and 
$\al^*(X_2) = c \ox X_2 + d \ox X_2$.

Of particular interest will be the \textit{coinvariant subalgebra} 
$\bC[X_1,X_2]^{\co SL(2)}$ defined by:
\[
\bC[X_1,X_2]^{\co SL(2)} := \set{p(X_1,X_2) \in \bC[X_1,X_2]
: \al^*(p(X_1,X_2)) = 1_{SL(2)} \ox p(X_1,X_2)}.
\]
Since on $\bC[X_1,X_2]$ there is a natural \textit{grading} given by
the order of the homogeneous polynomial expressions in $X_1,X_2$, it
follows that $\bC[X_1,X_2]^{\co SL(2)} \equiv \bC$, which comes from
the origin being the only point with a singleton orbit for the
action~$\al$. A hint for the calculations goes as follows. With
$p(X_1,X_2) = \sum_{i,j} a_{ij} X_1^i X_2^j$, a finite sum,
\[
\al^*\bigl( p(X_1,X_2) \bigr) 
= \sum_{i,j} a_{ij} \al^*(X_1)^i \al^*(X_2)^j 
= \sum_{i,j} a_{ij} (a \ox X_1 + b \ox X_2)^i 
(c \ox X_1 + d \ox X_2)^j.
\]
On the right-hand side $X_1$ and $X_2$ commute, but the generators
$a,b,c,d$ do not. Anyhow, using the grading by homogeneous polynomials
in $X_1,X_2$, one concludes, in the case $i+j = 1$, that
\[
(a_{10}\,a + a_{01}\,c) \ox X_1 + (a_{10}\,b + a_{01}\,d) \ox X_2
= a_{10}\,1_{SL(2)} \ox X_1 + a_{01}\,1_{SL(2)} X_2,
\]
for an element $p(X_i,X_2) \in \bC[X_1,X_2]^{\co SL(2)}$. This is
equivalent to the $2 \x 2$ homogeneous linear system
$a_{10}(a - 1) + a_{01}\,c = 0$ and
$a_{10}\,b + a_{01}(d - 1) = 0$, which has $a_{10} = 0 = a_{01}$ as a
unique solution, because of the linear independence of the generators.
Evidently, computations are much more involved for greater degrees
$i + j$.

With a moderate effort, it is possible to provide the quantum plane
$\bC_q[X_1,X_2]$ with a coaction from $SL_q(2)$, using the very same
explicit formula for~$\al^*$. See once again
\cite{KlimykSchmudgen, KasselBook}.
\end{example}


\subsection{A possibly non-existing map} 
\label{ssc:non-existing-map}

Referring to an action of a group $G$ on a space $X$, in
\cite{BaumHajac} one reads that the action is free if and only if the
map $F \: X \x G \to X \x_{X/G} X : (x,g) \mapsto (x,x\.g)$ is a
homeomorphism, and that this is equivalent to the assertion that the
$*$-homomorphism $F^* \: C(X \x_{X/G} X) \to C(X \x G)$ obtained from
the above map by $F^*(f) = f \circ F$ is an isomorphism. The so-called
\textit{canonical map} $F$ is always surjective, so that the
$*$-homomorphism $F^*$ is always injective. For that, \cite{BaumHajac}
argues (using density and $C^*$-algebra arguments) that the
surjectivity of $F^*$ implies that the action is free. As a result, it
is a corollary that
$F \: X \x G \to X \x_{X/G} X : (x,g) \mapsto (x,x\.g)$ being a
homeomorphism is equivalent to the claim that 
$F^* \: C(X \x_{X/G} X) \to C(X \x G)$ is an isomorphism, where
attention is paid to the injectivity of~$F^*$. This is a very nice
example on how the space $X$ gets classified from its associated
algebra of functions, and vice versa.

Let us now try to understand the reach of the following statement:
``$f \: X \to Y$ is a bijection between spaces (of whatever kind) is
equivalent to the map $f^* \: A(Y) \to A(X)$ being an isomorphism of
algebras''. Let $f \: X \to Y$ be a map between spaces (for which we
do not yet specify a category). Over $X$ and $Y$, consider some sort
of function algebras $A(X)$ and $A(Y)$ -- e.g., $C(X)$ or
$C^\infty(X)$, etc. Depending on the properties of $X$ and $Y$, these
$A(X)$ and $A(Y)$ could be algebras, $*$-algebras, Hopf algebras,
$C^*$-algebras, and so forth. Define $f^* \: A(Y) \to A(X)$ by
$f^*(h) := h\circ f$.

If $f$ is bijective, then it follows that $f^*$ is bijective, too.
Indeed, if $g \in A(X)$, then $g = (g \circ f^{-1}) \circ f$ with
$g \circ f^{-1} \in A(Y)$, meaning that $f^*$ is surjective. Also, if
$h_1 \circ f = f^*(h_1) = f^*(h_2) = h_2 \circ f$, then
$h_1 = h_2$ in $A(Y)$, saying that $f^*$ is injective.

Conversely, that is, would $f^*$ being a bijection imply that $f$ is a
bijection? The answer to that is the procedure we posited before to
deduce from the algebra $A(X)$ -- then denoted $\sO(X)$ -- the space
$X$, up a bijective equivalence. Recall that a driving idea in
noncommutative geometry is to replace a space $X$ by an algebra of
functions $A(X)$, which in turn may be replaced by a possibly
noncommutative algebra~$A'$. This new algebra $A'$ consists of
``functions over a possibly non-existing space'' $X'$, that will play
the role of $A' = A(X')$. In this way, having an algebra map
$F \: A \to B$, what matters is to regard it as $F = f^*$ where
$f \: X \to Y$ is a (possibly non-existing) map between (possibly
non-existing) spaces $X$ and~$Y$.

Let us look at the possibility that $f^*$ bijective implies $f$ 
bijective. Assume $f^* \: A(Y) \to A(X)$ is a bijection. If
$f(x_1) = f(x_2)$ for $x_1,x_2 \in X$, then
$f^*(h)(x_1) = h(f(x_1)) = h(f(x_2)) = f^*(h)(x_2)$, for all
$h \in A(Y)$. Because $f^*$ is surjective, $g(x_1) = g(x_2)$ for all
$g \in A(X)$. If $A(X)$ has a separating property with respect to $X$,
it follows that $x_1 = x_2$, and so $f$ is injective. Remember that a
family $\sF$ of functions with domain $X$ is said to \textit{separate
the points} of $X$ if for any two distinct $x,y \in X$, there exists a
function $\phi \in \sF$ such that $\phi(x) \neq \phi(y)$.

\begin{remark}
The property of separating points occurs in a great many algebras. The
best known examples are Urysohn's lemma and the Hahn--Banach
separation theorem. The first states that a topological space $X$ is
normal if and only if, for any two nonempty closed disjoint subsets
$C$ and $D$ of $X$, there exists a continuous map $f \: X \to [0,1]$
such that $f(C) = \{0\}$ and $f(B) = \{1\}$. Thereby, the algebra
$C(X)$ of continuous functions on $X$ with real (or complex) values
separates points on~$X$. The second statement avers that, if $X$ is a
locally convex Hausdorff topological vector space over $\bR$ or~$\bC$,
then the continuous linear functionals on $X$ separate points.

The several versions of the Stone--Weierstrass theorem provide another
rich source of examples. The original version proved by Marshall H.
Stone states that a unital subalgebra of the algebra $C(X,\bR)$ of
real-valued functions on a compact Hausdorff space $X$ is dense in
$C(X,\bR)$, with the topology of uniform convergence, if and only if
it separates the points of~$X$. The complex version, better suited to
our work ahead, states that the complex unital $*$-algebra generated
by a separating subset of $C(X,\bC)$, is dense in $C(X,\bC)$. There is
also a version of the theorem for $X$ locally compact: a subalgebra
$A$ of the Banach algebra $C_0(X,\bR)$, of real-valued continuous
functions on $X$ that vanish at infinity, is uniformly dense in
$C_0(X,\bR)$ if and only if it separates points and vanishes nowhere
(not all of the elements of $A$ vanish simultaneously at some point).
Lastly, we mention Nachbin's theorem~\cite{Nachbin49}. On a
finite-dimensional smooth manifold $X$, if a subalgebra $A$ of the
algebra $C^\infty(X)$ of smooth functions separates the points of $X$
and also separates its tangent vectors: for each point $x \in X$ and
each vector $v \in T_x X$, there is an $f \in A$ such that
$df(x)(v) \neq 0$; then $A$ is dense in $C^\infty(X)$.
\end{remark}

Now let us discuss the surjectivity of~$f$. Let $y \in Y$ and assume
there is some $g_y \in A(Y)$ that can ``distinguish $y$ from the other
elements of~$Y$''. Then $h_y = f^*(g_y) = g_y \circ f \in A(X)$, and
now the bijectivity of~$f^*$ would mean that whatever property $g_y$
enjoys (with respect to~$Y$), $h_y$ enjoys also (with respect to~$X$).

\begin{remark}
One example of a situation in which a member of an algebra of
functions over a given space can ``distinguish a point from the other
elements'' is the Riesz representation theorem: for every continuous
linear functional $\phi \in H^*$ on a given Hilbert space $H$, there
is a unique element $x_\phi \in H$ such that
$\phi(y) = \duo{x_\phi}{y}$ for every $y \in H$. One step deeper in
this direction is given by the so-called \textit{dual systems} in
which two spaces mutually distinguish or separate points of each
other, of which the former is a particular case: let $H$ and $K$ be
two vector spaces over a field $k$ and let $b \: H \x K \to k$ be a
nondegenerate bilinear form: for all nonzero $x \in H$, there exists
$y \in K$ such that $b(x,y) \neq 0$; and, for all nonzero $y \in K$,
there is an $x \in H$ such that $b(x,y) \neq 0$.
\end{remark}


\section{Hopf--Galois extensions: preliminaries} 
\label{sec:HG-for-dummies}

In \cite{Bohm07} one reads: ``a Hopf--Galois extension (that is
faithfully flat) can be interpreted as a (dual version of a)
noncommutative principal bundle'', and in \cite{Tobolski}, referring
to a rough classification of compact principal bundles in
noncommutative algebraic geometry: ``In this setting, the role of a
topological group and a topological space are played by a Hopf algebra
and its unital comodule algebra respectively.'' We want to introduce
the theory of Hopf--Galois extensions and how these can be considered
as noncommutative algebraic analogues of a principal bundle.

To define what a Hopf--Galois extension is, one must first deal with
several concepts from algebras and coalgebras: actions and coactions,
Hopf algebras, algebra and coalgebra extensions by bialgebras,
balanced tensor products, canonical maps, arriving finally at the
definition of a Hopf--Galois extension.


\subsection{Algebras, coalgebras and bialgebras} 
\label{ssc:all-algebras}

We denote by~$k$ a commutative associative unital ring. A \textit{left
$k$-module} over the ring $k$ is an abelian group $(M,+)$ together
with an operation $k \x M \to M$ (just written by juxtaposition) such
that $r(x + y) = rx + ry$, $( r+ s)x = rx + sx$, $(rs)x = r(sx)$,
$1x = x$ for all $r,s \in k$, $x,y \in M$; here $1$ is the unit
of~$k$. Similarly one defines a \textit{right $k$-module}.%
\footnote{%
In the context of modules over a ring, special attention must be given
to torsion in tensor products.} 
Since $k$ is commutative, a left $k$-module can be regarded as a right
$k$-module and vice versa, so we just speak of a \textit{$k$-module}.

If $M$ is a left $k$-module and $N$ is a right $k$-module, their
tensor product $M \ox N$ (omitting the ornament subindex~$k$) is an
abelian group together with a bilinear map
$\ox \: M \x N \to M \ox N$ satisfying $\ox(mr,n) = \ox(m,rn)$, with a
universal property: for any abelian group $Z$ and any bilinear map
$\phi \: M \x N \to Z$, there is a unique additive map
$\tilde\phi \: M \ox N \to Z$ such that $\tilde\phi \circ \ox = \phi$.
This defines the tensor product uniquely up to isomorphism (although
existence requires a separate argument).

A \textit{$k$-algebra} is a $k$-module $A$ equipped with a $k$-linear
multiplication map $\mu \: A \ox A \to A$ and a $k$-linear unit map
$\eta \: k \to A$ satisfying the obvious conditions of associativity
and unitality:
\[
\mu(\mu \ox \id_A ) = \mu(\id_A \ox \mu) : A \ox  A \ox  A \to A;
\quad
\mu(\eta \ox \id_A) = \id_A = \mu(\id_A \ox \eta).
\]
Our algebras will all be unital algebras. We denote
$1 = 1_A = \eta(1_k) = \eta(1)$.

If $A$ is a $k$-algebra then $A \ox A$ is a $k$-algebra, too, with its
multiplication defined by
$(a_1 \ox a_2)(b_1 \ox b_2) = a_1b_1 \ox a_2b_2$, extended by
linearity. The tensor product of unital algebras is also a unital
algebra.

To each ``structure'' in use, one can associate a ``co-structure''.
To begin with, the associative multiplication and the unit on a
$k$-algebra $A$ can be pictured with commutative diagrams:
\[
\xymatrix{\ar@{}[dr] & k \ar[d]^\eta  \\
A \ox A \ar[r]^(.6)\mu & A}
\qquad
\xymatrix{\ar@{}[dr]
A \ox A \ox A \ar[d]_{\mu\ox \id_A} \ar[r]^(.6){\id_A\ox\mu} &
A \ox A \ar[d]^\mu \\
A \ox A \ar[r]^(.6){\mu} & A}
\qquad
\xymatrix{\ar@{}[dr]
k \ox A \isom A \ox k \ar[d]_{\eta\ox\id_A} \ar[r]^(.65){\id_A\ox\eta}
\ar[dr]^{\id_A} & A \ox A \ar[d]^\mu \\
A \ox A \ar[r]^\mu & A}
\]
with  the identifications
$A \ox A \ox A := (A \ox A) \ox A \isom A \ox (A \ox A)$ and 
$k \ox A \isom A \isom A \ox k$. ``Dualizing'' these diagrams, i.e.,
reversing the arrows, we obtain other commutative diagrams:
\[
\xymatrix{\ar@{}[dr] & k  \\
C \ox C & C \ar[l]_(.3) \Dl \ar[u]_\eps}
\quad
\xymatrix{\ar@{}[dr]
C \ox C \ox C & C \ox C \ar[l]_(.4){\id_C\ox\Dl} \\
C \ox C \ar[u]^{\Dl\ox\id_C} & C \ar[u]_\Dl \ar[l]_(.4)\Dl}
\qquad
\xymatrix{\ar@{}[dr]
k \ox C \isom C \ox k & C \ox C \ar[l]_(.35){\id_C\ox\eps} \\
C \ox C \ar[u]^{\eps\ox\id_C} &
C \ar[l]_\Dl \ar[u]_\Dl \ar[lu]_{\id_C} }
\]
again with $C \ox C \ox C := (C \ox C) \ox C \isom C \ox (C \ox C)$
and $k \ox C \isom C \isom C \ox k$. In this way, we arrive at the
definition of a $k$-coalgebra.

Namely: a \textit{$k$-coalgebra} is a $k$-module $C$ equipped with a
$k$-linear \textit{comultiplication} map $\Dl \: C \to C \ox C$ and a
$k$-linear \textit{counit} map $\eps \: C \to k$. satisfying the
conditions of \textit{coassociativity} and \textit{counitality}:
\[
(\id_C \ox \Dl) \circ \Dl = (\Dl \ox \id_C) \circ \Dl
: C \to C \ox C \ox  C, \quad
(\id_C \ox \eps) \circ \Dl = \id_C = (\eps \ox \id_C) \circ \Dl.
\]

\begin{example} 
\label{eg:trigoncoalg}
Out of the standard trigonometric identities
\[
\sin(x + y) = \sin x \cos y + \cos x \sin y,  \qquad
\cos(x + y) = \cos x \cos y - \sin x \sin y,
\]
and the values at the origin $\sin 0 = 0$ and $\cos 0 = 1$, one can,
by suppressing the variables, define a two-dimensional $k$-coalgebra
$C$, called the \textit{trigonometric coalgebra}
\cite{BrzezinskibyWitkowski} as follows: $C$ has a basis
$\{\sin, \cos\}$, with comultiplication
\[
\Dl(\sin) := \sin \ox \cos + \cos \ox \sin,  \qquad
\Dl(\cos) := \cos \ox \cos - \sin \ox \sin,
\]
and counit $\eps(\sin) = 0$, $\eps(\cos) = 1$, both extended by
linearity over $k$. Direct computations show they satisfy
coassociativity and counitality.
\end{example}

\begin{remark}
Sweedler's notation \cite{Sweedler} is commonly used to represent the
action of $\Dl$ as a finite sum $\Dl(c) = c_{(1)}\ox c_{(2)}$, with
implicit summation understood. The counit property becomes
\[
\eps(c_{(1)}) c_{(2)} = c = c_{(1)} \eps(c_{(2)})
\word{for all} c \in C.
\]
Coassociativity entails
\[
c_{(1)} \ox c_{(2)(1)} \ox c_{(2)(2)} 
= c_{(1)(1)} \ox c_{(1)(2)} \ox c_{(2)},
\]
This common value 
$\Dl^2(c) := (\id_C \ox \Dl)\Dl(c) = (\Dl \ox \id_C)\Dl(c)$ is
rewritten as $c_{(1)} \ox c_{(2)} \ox c_{(3)}$.

\medskip

If $f \: C \x C \xyx C \to A$ is any $k$-multilinear map, and if 
$\bar f \: C \ox C \oxyox C \to A$ is the induced $k$-linear map on
the tensor product, then Sweedler's notation states that
\[
\bar f(c_{(1)} \ox c_{(2)} \oxyox c_{(n)}) 
= \bar f(\Dl^{n-1}(c)) = f(c_{(1)},c_{(2)},\dots,c_{(n)}).
\] 
The map $\tau \: C \ox C \to C \ox C$, given by linearly extending 
$\tau(c \ox c') = c' \ox c$, is the induced linear map 
corresponding to the $k$-bilinear map 
$t \: C\x C \to C \ox C : (c,c') \mapsto c' \ox c$. The map
$\tau\Dl \: C \to C \ox C$ satisfies
$\tau\Dl(c) = \tau (c_{(1)} \ox c_{(2)}) = t(c_{(1)},c_{(2)})
= c_{(2)} \ox c_{(1)}$, illustrating this convention.
\end{remark}

\begin{example} 
\label{eg:algebra-times-coalgebra}
Let $A$ be an algebra and $C$ be a coalgebra. Since both are
$k$-modules, we can consider their tensor product $A \ox C$.

We can define a coalgebra structure on $A \ox C$ by trivializing the
role of $A$ and defining 
$\Dl_{A\ox C} \: A \ox C \to A \ox C \ox A \ox C$ by
$a \ox c \mapsto a \ox c_{(1)} \ox a \ox c_{(2)}$, and
$\eps(a \ox c) := \eps(c)$, both extended by linearity. Indeed,
\begin{align*}
(\id_{A\ox C} \ox \Dl_{A\ox C}) \circ \Dl_{A\ox C} (a \ox c) 
&= (\id_{A\ox C} \ox \Dl_{A\ox C})(a \ox c_{(1)} \ox a \ox c_{(2)})
\\
&= a \ox c_{(1)} \ox a \ox c_{(2)} \ox a \ox c_{(3)}
\\
&= (\Dl_{A\ox C} \ox \id_{A\ox C})(a \ox c_{(1)} \ox a \ox c_{(2)})
\\
&= (\Dl_{A\ox C} \ox \id_{A\ox C}) \circ \Dl_{A\ox C} (a \ox c),
\\
\shortintertext{and}
(\id_{A\ox C} \ox \eps_{A\ox C}) \circ \Dl(a \ox c) 
&= (\id_{A\ox C} \ox \eps_{A\ox C}) (a \ox c_{(1)} \ox a \ox c_{(2)})  
\\
&= a \ox c_{(1)} \eps(c_{(2)}) = \id_{A\ox C} (a \ox c)
= a \eps(c_{(1)}) \ox c_{(2)}
\\
& = (\eps_{A\ox C} \ox \id_{A\ox C}) (a \ox c_{(1)} \ox a \ox c_{(2)})
\\
&= (\eps_{A\ox C} \ox \id_{A\ox C}) \circ \Dl (a \ox c).
\tag*{\qef}
\end{align*}
\hideqef
\end{example}

Once we have well-defined objects, it is natural to
consider morphisms relating two of them, according to their
structures. A \textit{coalgebra homomorphism} is a map
$f \: C \to C'$ between two coalgebras $(C,\Dl,\eps)$ and
$(C',\Dl',\eps')$, such that $\eps' \circ f = \eps$ and
$(f \ox f) \circ \Dl = \Dl' \circ f$, with $f \ox f$ acting on the
coalgebra tensor product $C \ox C$. In diagrams:
\[
\xymatrix{\ar@{}[dr]
C \ar[dr]_\eps \ar[rr]^f  & & C' \ar[dl]^{\eps'} \\
& k & }
\hspace*{6em}
\xymatrix{\ar@{}[dr]
C \ar[r]^f \ar[d]_\Dl & C' \ar[d]^{\Dl'}  \\
C \ox  C \ar[r]^{f\ox f} & C'\ox  C' }
\]
These two diagrams are dual to those for an algebra homomorphism
$\phi \: A \to A'$,
\[
\xymatrix{\ar@{}[dr]
A \ar[rr]^\phi & & A'   \\
& k \ar[ul]^\eta \ar[ur]_{\eta'} & }
\hspace*{6em}
\xymatrix{\ar@{}[dr]
A \ar[r]^\phi  & A'  \\
A \ox A \ar[r]^{\phi\ox \phi} \ar[u]^\mu & A'\ox  A' \ar[u]_{\mu'} }
\]
where $\phi \circ \eta = \eta'$ and 
$\phi \circ \mu = \mu' \circ (\phi \ox \phi)$.

\medskip

Many interesting properties of the structures we shall study rely on
the harmony of an algebra and a coalgebra, coalescing in a single
object called a bialgebra.

A \textit{bialgebra} is a $k$-module $H$, together with both a
$k$-algebra structure $(H,\mu,\eta)$ and a $k$-coalgebra structure
$(H,\Dl,\eps)$ in such a way that the counit $\eps$ and the coproduct
$\Dl$ are $k$-algebra homomorphisms:
\[
\eps(hh') = \eps(h) \eps(h'),  \quad  \eps(1) = 1;
\qquad
\Dl(hh') = \Dl(h)\Dl(h'),  \quad  \Dl(1) = 1 \ox 1.
\]
In Sweedler's notation, the third relation is written as
$(hh')_{(1)} \ox (hh')_{(2)} 
= h_{(1)} {h'}_{\!(1)} \ox h_{(2)} {h'}_{\!(2)}$. These are the
\textit{compatibility conditions} for the algebra and the coalgebra
structure on~$H$. In diagrams:
\[
\xymatrix{\ar@{}[dr]
H \ox H \ar[r]^(.6)\mu \ar[d]_{\eps\ox\eps} & H \ar[d]^\eps  \\
k \isom k \ox k  \ar[r]^(.65){\cdot} & k }
\qquad\qquad
\xymatrix{\ar@{}[dr]
H \ox H \ar[r]^\mu \ar[d]_{\Dl \ox \Dl}  & H \ar[d]^\Dl  \\
H \ox H \ox H \ox H \ar[r]^(.65){\mu_\ox} & H \ox H }
\]
where $\mu_\ox = \mu \ox \mu \circ \switch{2}{3}$ as before.

A \textit{homomorphism of bialgebras} is at once an algebra and a
coalgebra homomorphism.%
\footnote{%
In subsection~\ref{ssc:Hopf-algebras} we shall see how the vector
space $\Hom(C,A)$ of linear maps from a coalgebra $C$ to an algebra
$A$ is actually an algebra under the operation of convolution.}

The compatibility conditions are equivalent to the requirements that
the unit $\eta$ and the product $\mu$ are $k$-coalgebra homomorphisms:
\begin{gather*}
\mu(1 \ox 1) = 1,  \quad  \eta(1) = 1;  \qquad
\Dl(c_1 c_2) = \Dl(c_1)\Dl(c_2),  \quad
\Dl \eta(t) = \eta(t) \ox 1 = 1 \ox \eta(t),
\\
\xymatrix{\ar@{}[dr]
H \ox  H & H \ar[l]_(.35)\Dl   \\
k \ox  k \isom k \ar[u]^{\eta\ox\eta} & k \ar[u]_\eta \ar[l] }
\qquad\qquad
\xymatrix{\ar@{}[dr]
H \ox H  & H \ar[l]_\Dl   \\
H \ox H \ox H \ox H \ar[u]^{\mu_\ox} &
H \ox H \ar[l]_(.35){\Dl \ox \Dl} \ar[u]_\mu }
\end{gather*}
These diagrams are simply the reverses of the diagrams for the
compatibility conditions above.

\begin{example} 
\label{eg:bialgebra}
Consider the polynomial algebra $F[X]$ over a field $F$ with the
product defined on monomials by $\mu(X^k \ox X^l) := X^{k+l}$,
extended by linearity, and $\eta(1) = 1$. It is a coalgebra under the
coproduct $\Dl(X^n) := X^n \ox X^n$ and counit $\eps(X^n) := 1$ for
every $n \geq 0$, both extended by linearity. It is straightforward to
verify that $(F[X],\mu,\eta,\Dl,\eps)$ is a bialgebra.
\end{example}

\begin{example} 
\label{eg:Laurent-bialgebra}
The \textit{Laurent polynomials} algebra $F[X,X^{-1}]$ over a field
$F$ is obtained from the product defined on monomials by
$\mu(X^k \ox X^l) := X^{k+l}$ (now with $k,l \in \bZ$) and extended by
linearity, and $\eta(1) = 1$. The coproduct is given on generators by
$\Dl(X) := X \ox X$ and $\Dl(X^{-1}) := X^{-1} \ox X^{-1}$. On
extending multiplicatively to the whole $F[X,X^{-1}]$ one obtains
$\Dl(X^n) = X^n \ox X^n$, for any $n \in \bZ$. The counit is
$\eps(X^n) := 1$, also extended by linearity; all compatibility
conditions are trivial. As a result,
$(F[X,X^{-1}],\mu,\eta,\Dl,\eps)$ is a bialgebra.
\end{example}

\begin{example} 
\label{eg:bialgebra-bin}
Consider the polynomial algebra $F_b[X]$ over a field $F$ with the
product again defined on monomials by $\mu(X^k \ox X^l) := X^{k+l}$
(now with $k,l \in \bN$), and again $\eta(1) = 1$. It is a coalgebra,
too (called the \textit{binomial coalgebra}) under the coproduct
\[
\Dl(X^n) := \sum_{k=0}^n \binom{n}{k} X^k \ox X^{n-k}
\]
and counit $\eps(X^n) := 0$ for $n \neq 0$, $\eps(1) := 1$, also
extended by linearity. Here all except the third compatibility
conditions are immediate. As for this third one:
\begin{align*}
\Dl(X^m) \Dl(X^n) &= \sum_{j=0}^m \sum_{k=0}^n
\binom{m}{j} \binom{n}{k} X^j X^k \ox X^{m-j} X^{n-k}
\\
&= \sum_{r=0}^{m+n} \sum_{j+k=r}
\binom{m}{j} \binom{n}{k} X^{j+k} \ox X^{m+n-j-k}
\\
&= \sum_{r=0}^{m+n} \binom{m+n}{r} X^r \ox X^{m+n-r} = \Dl(X^{m+n}),
\end{align*}
thanks to Vandermonde's convolution identity. Thus,
$(F_b[X],\mu,\eta,\Dl,\eps)$ is a bialgebra. Note that $F_b[X]$ and
$F[X]$ have different coalgebra structures.
\end{example}

\begin{example} 
\label{eg:alg-grupo} 
For a finite group $G$ and a field $F$, the \textit{group algebra}
$F[G]$ is the vector space over $F$ of dimension $|G|$ whose elements
are formal sums $\sum_{g\in G} a_g g$, with $a_g \in F$. Its product
is given by
\[
\biggl( \sum_{g\in G}  a_g g \biggr)\,
\biggl( \sum_{h\in G}  b_h h \biggr)
:= \sum_{k\in G} \biggl( \sum_{gh=k} a_g b_h \biggr) k.
\]
Equivalently, $\mu(g \ox h) := gh$ on the usual basis of
$F[G] \ox F[G]$. In this example, $\eta(1_F) = 1_G$. The coproduct
$\Dl$ and counit on $F[G]$ are defined on the basic elements $g \in G$
by
\begin{equation}
\Dl(g) := g \ox g,  \quad  \eps(g) := 1_F.
\label{eq:group-like} 
\end{equation}
Since $\Dl(gh) = gh \ox gh = (g \ox g)(h \ox h) = \Dl(g)\Dl(h)$ and
$\eps(gh) = 1_F = \eps(g) \eps(h)$, it follows that $F[G]$ is a
bialgebra.
\end{example}

\begin{remark}
In any bialgebra $H$ over a field $F$, any $g \in H$ satisfying
\eqref{eq:group-like} is called a \textit{grouplike} element of~$H$.
\end{remark}

If $G$ is not abelian, $F[G]$ is not commutative, thus the relation
$\mu(a \ox b) = \mu(b \ox a)$ does not hold for every $a,b \in A$.
This relation can be written as $\mu \circ \tau = \mu$. Dually, a
coalgebra $C$ is \textit{cocommutative} if and only if
$\tau \circ \Dl = \Dl \: C \to C \ox C$. In Sweedler's notation,
\[
\tsum a_{(1)} \ox a_{(2)} = \tsum a_{(2)} \ox a_{(1)} 
\word{for every} a \in C.
\]
Note that $F[G]$ is cocommutative.

\begin{example} 
\label{eg:bialg-func} 
Once again, let $G$ be a finite group and $F$ be a field. The
\textit{bialgebra of functions} $\sO(G)$ (also denoted $F^G$) is the
algebra of functions $f \: G \to F$ with pointwise operations
\[
(f + l)(g) := f(g) + l(g),  \qquad  (fl)(g) := f(g)\,l(g).
\]
This is an $F$-vector space $\sO(G) \isom F^{|G|}$ of dimension~$|G|$
with basis $\set{e_g : g \in G}$ where, using Kronecker's delta, 
$e_g(h) := \dl_{gh}$ and $m(e_g \ox e_h) = \dl_{gh}\,e_g$. Here,
$\eta(1)$ equals the constant function $g \mapsto 1$ for $g \in G$.

In this case, one may identify the tensor product $\sO(G) \ox \sO(G)$
with the algebra of functions $\sO(G \x G) \isom F^{|G|^2}$, as
follows. If $f,l \in \sO(G)$, then $f \ox l$ is identified with the
two-variable function $(g,h) \mapsto f(g)\,l(h)$. The map
$f \mapsto f(g) \in F$ is then a linear functional on~$\sO(G)$.

The coproduct $\Dl \: \sO(G) \to \sO(G \x G)$ is given by 
$\Dl(f)(g,h) := f(gh)$. The coassociativity of $\sO(G)$ is a
consequence of the associativity of~$G$:
\begin{align*}
(\Dl \ox \id)(\Dl(f))(g,h,k) &= \Dl(f)(gh,k) = f((gh)k) = f(g(hk))
\\
&= \Dl(f)(g,hk) = (\id \ox \Dl)(\Dl(f))(g,h,k),
\end{align*}
or $\Dl^2(f)(g,h,k) = f(ghk)$, for short. The counit in $\sO(G)$ is
given by $\eps(f) := f(1_G)$. If the group $G$ is not abelian, this
bialgebra is commutative but not cocommutative.

As a particular case, $G = \bZ_2 = \{\pm1\}$ produces a basis
$\{e_1,e_{-1}\}$ for $\sO(\bZ_2)$ with multiplication $e_1 e_1 = e_1$,
$e_1 e_{-1} = e_{-1}e_1 = 0$, $e_{-1}e_{-1} = e_{-1}$, and
comultiplication given by
$\Dl(e_1) =  e_1 \ox e_1 + e_{-1} \ox e_{-1}$ and 
$\Dl(e_{-1}) =  e_1\ox e_{-1} + e_{-1} \ox e_1$.
\end{example}

Originally introduced as generalizations of Galois extensions of
commutative rings by groups \cite{ChaseSweedler, KreimerTakeuchi,
GreitherPareigis}, Hopf--Galois extensions now play an important role
beyond its algebraic origin, in noncommutative geometry (see
subsection~\ref{ssc:tracking-the-idea} for a list of references).
Towards a definition of Hopf--Galois extensions, we require some more
concepts.


\subsection{Actions and coactions} 
\label{ssc:act-and-coact}

Recall that a left action of a group $G$ on a set $X$ is a map
$\al \: G \x X \to X$, such that $\al(g,\al(h,x)) = \al(gh,x)$ and
$\al(1_G,x) = x$ for all $g,h \in G$ and all $x \in X$. Similarly,
\textit{left action} of a $k$-algebra $A$ on $k$-module $M$ is a
$k$-linear map $\rho \: A \ox M \to M$ satisfying:
\begin{enumerate}
\item 
$\rho(\id_A \ox \rho) = \rho(\mu \ox \id_M) : A \ox A \ox M \to M$,
\item 
$\rho(\eta \ox \id_M) = \id_M \in \End(M)$,
\end{enumerate}
i.e., $\rho(a \ox \rho(a' \ox m)) = \rho(a a' \ox m)$ and
$\rho(1_A \ox m) = m$. In diagrams:
\[
\xymatrix{\ar@{}[dr]
A \ox A \ox M \ar[d]_{\mu\ox\id_M} \ar[r]^(.6){\id_A\ox\rho} &
A \ox M \ar[d]^\rho  \\
A \ox M \ar[r]_(.6)\rho  & M }
\hspace*{6em}
\xymatrix{\ar@{}[dr]
k \ox M \ar[rr]^{\eta\ox\id_M} \ar[dr]_\cdot & & 
A \ox M \ar[dl]^{\rho} \\
& M & }
\]
Likewise, one can define a \textit{right action} $M \ox A \to M$. If
there is a left [respectively, right] action of $A$ on $M$, we say
that $M$ is a \textit{left} [resp., \textit{right}] \textit{module}
of~$A$. Any algebra is evidently a module over itself, the action
being given by the multiplication map.

As before, we ``dualize'' the above diagrams to define the concept of
coaction. A \textit{right coaction} of a $k$-coalgebra $C$ on a
$k$-module $N$ is a $k$-linear map $\dl \: N \to N \ox C$ satisfying:
\begin{enumerate}
\item 
$(\dl \ox \id_C) \dl = (\id_N \ox \Dl) \dl : N \to N \ox C \ox C$;
\item 
$(\id_N \ox \eps) \dl = \id_N \ox 1_k \isom \id_N \in \End(N)$ with
$N \ox k \isom N$.
\end{enumerate}
In diagrams:
\[
\xymatrix{\ar@{}[dr]
N \ox C \ox C & N \ox C \ar[l]_(.4){\dl\ox\id_C} \\
N \ox C \ar[u]^{\id_N\ox\Dl} & N \ar[l]^(.4)\dl \ar[u]_\dl }
\hspace*{6em}
\xymatrix{\ar@{}[dr]
N \ox k & & N \ox C \ar[ll]_{\id_N\ox\eps} \\
& N \ar[ul]^{\id_N \ox 1_k} \ar[ur]_{\dl} & }
\]
Similarly, one can define a \textit{left coaction} of $C$ on~$N$. If
there is a right [respectively, left] coaction of $C$ on $N$, we say
that $N$ is a \textit{right} [resp., \textit{left}] \textit{comodule}
of~$C$. Every coalgebra is evidently a comodule over itself, the
coaction being given by the comultiplication map.

If $\dl(n) =: n_{(0)} \ox n_{(1)}$, with $n_{(0)}\in N$ and
$n_{(j)} \in C$ for $j > 0$, we obtain
\[
n_{(0)(0)} \ox n_{(0)(1)} \ox n_{(1)} 
= n_{(0)} \ox n_{(1)(1)} \ox n_{(1)(2)}
=: n_{(0)} \ox n_{(1)} \ox n_{(2)}
\]
and $n_{(0)} \eps(n_{(1)}) = n$. Sweedler's notation for a left
coaction will keep $n_{(0)} \in N$ but employ negative indices to
label elements of~$C$.

If $N$ and $N'$ are two right $C$-comodules with coactions
$\dl \: N \to N \ox C$ and $\dl' \: N' \to N' \ox C$, a $C$-comodule
morphism $\phi \: N \to N'$ is a $k$-module morphism that further
satisfies $\dl' \circ \phi = (\phi \ox \id_C) \circ \dl$.

\medskip

A homomorphism of unital algebras $f \: A \to B$ satisfies
$f(1_A) = 1_B$ and $f(aa') = f(a)f(a')$, for $a,a' \in A$. If
$H$ is a $k$-bialgebra and $A$ is a $k$-algebra with a left action
$\rho \: H \ox A \to A$ which is a unital algebra homomorphism,
that is, $\rho(1_H \ox 1_A) = \eps(1_H) 1_A = 1_A$ and
$\rho((h \ox a) \ox (h' \ox a')) = \rho(h \ox a) \rho(h' \ox a')$, we
say that $\rho$ is a \textit{left Hopf action} of $H$ on~$A$ and call
$A$ a \textit{left $H$-module algebra} provided that
$\rho(h \ox a a')  = \rho(h_{(1)} \ox a) \rho(h_{(2)} \ox a')$ and
$\rho(h \ox 1_A) = \eps(h) 1_A$, for all $h \in H$; $a,b \in A$. In
diagrams:
\[
\xymatrix{\ar@{}[dr]
H \ox A \ox A \ar[d]_{\id_H\ox\mu} \ar[r]  &
H \ox H \ox A \ox A \ar[r]^(.5){{}_2\frown_3} &
H \ox A \ox H \ox A \ar[d]^{\rho \ox\rho}  \\
H \ox A \ar[r]_(.6)\rho & A & A \ox A \ar[l]^\mu }
\quad
\xymatrix{\ar@{}[dr]
H \ox k \ar[d]_{\id_H\ox \eta} \ar[r]^{\eps\ox \eta} &
k \ox A \ar[d]^{\cdot}  \\
H \ox A \ar[r]_\rho & A }
\]
where the map on the top left is $\Dl \ox \id_A \ox \id_A$. In a
similar fashion, it is possible to define \textit{right Hopf actions}
and \textit{right $H$-module algebras}.

\begin{remark}
Two special kinds of element for $H$-module algebras should be
mentioned. If $h$ is \textit{grouplike}, then
\[
\rho(h \ox a a') = \rho(h \ox a) \rho(h \ox a'), \qquad
\rho(h \ox 1_A) = 1_A, 
\]
so $a \mapsto \rho(h \ox a)$ is an algebra \textit{automorphism}
of~$A$.

If $h$ is \textit{primitive}, that is to say,
$\Dl h = h \ox 1 + 1 \ox h$ and $\eps(h) = 0$ (see
Example~\ref{eg:enveloping-Lie-algebra}), then
\[
\rho(h \ox a a') = \rho(h \ox a) a' + a \rho(h \ox a'), \qquad
\rho(h \ox 1_A) = 0, 
\]
thus $a \mapsto \rho(h \ox a)$ is a \textit{derivation} of the
algebra~$A$.
\end{remark}

If $\rho \: H \ox A \to A$ is an algebra action of the bialgebra $H$
over the algebra $A$ (i.e., an action of the algebra $H$ over the
module structure of the algebra $A$ which is an algebra homomorphism)
the \textit{invariant subspace} $A^H$ of $A$ is defined as
\[
A^H := \set{a \in A : \rho(h \ox a) = \eps(h)a, \mot{for all} h\in H}.
\]
In case $\rho$ is a Hopf action, $A^H$ is  a subalgebra of $A$. We
call it the \textit{invariant subalgebra}%
\footnote{%
In the case that $G$ is a finite group and $H = F[G]$, since each
$\eps(g) = 1$ it is clear that $A^{F[G]} = A^G$ is just the
fixed-point subalgebra of $A$ under a linear group action of $G$
on~$A$. See Example~\ref{eg:alg-grupo}.}
of~$A$. Indeed, since $\eps(h) \in k$ for all $h \in H$,
\begin{align*}
\rho(h \ox a a') &= \rho(h_{(1)} \ox a) \rho(h_{(2)} \ox a') 
\\
&= \eps(h_{(1)}) \eps(h_{(2)}) a a' = \eps(\eps(h_{(1)})h_{(2)}) a a'
= \eps(h) a a'
\end{align*}
for all $a,a' \in A$ and $h \in H$. One might indeed think of the
above as a motivation to define Hopf actions.

If $H$ is a $k$-bialgebra and $C$ is a $k$-coalgebra, the coalgebra
$C \ox H$ has comultiplication given by
$\Dl_{C\ox H}(c \ox h) = c_{(1)} \ox h_{(1)} \ox c_{(2)} \ox h_{(2)}$,
and $\eps_{C\ox H}(c \ox h) = \eps_C(c) \eps_H(h)$. If a coaction
$\dl \: C \to C \ox H$ is a coalgebra homomorphism, that is,
$\eps_{C\ox H}\,\dl = \eps_C$ and
$(\dl \ox \dl) \Dl_C = \Dl_{C\ox H}\,\dl$, we say that $\dl$ is a
\textit{coalgebra coaction}, and that $C$ is a \textit{(right)
$H$-comodule coalgebra}. In commutative diagrams:
\[
\xymatrix{\ar@{}[dr]
C \ox C \ox H  & C \ox C \ox H \ox H \ar[l] &
C \ox H \ox C \ox H \ar[l]_(.47){{}_2\frown_3} \\
C \ox H \ar[u]^{\Dl\ox\id_H} & C \ar[l]_\dl \ar[r]^{\Dl_C} & 
C \ox C \ar[u]^{\dl\ox\dl} }
\quad
\xymatrix{\ar@{}[dr]
k \ox H & C \ox k \ar[l]_{\eps\ox \eta} \\
C \ox H \ar[u]^{\eps\ox\id_H} & C \ar[l]_(.4)\dl \ar[u]_{\cdot} }
\]
where the map at the top left corner is $\id_C \ox \id_C \ox \mu$. In
a similar fashion, it is possible to define \textit{left Hopf
coactions} and \textit{left $H$-comodule algebras}.

If $\dl \: A \to A \ox H$ is a coaction of the bialgebra $H$ on the
algebra $A$, the \textit{coinvariant subalgebra} $A^{\co H}$ of $A$ is
defined by:
\[
A^{\co H} := \set{a \in A : \dl(a) = a \ox 1_H}.
\]
Evidently,
$\dl(ab) = \dl(a) \dl(b) = (a \ox 1_H)(b \ox 1_H) = ab \ox 1_H$. 
Therefore, $ab \in A^{\co H}$ for all $a,b \in A^{\co H}$.

\begin{remark}
Because $\dl$ is an algebra homomorphism, 
$\dl(aa') = \dl(a) \dl(a') = (a \ox 1) ({a'}_{\!(0)} \ox {a'}_{\!(1)})
= (a {a'}_{\!(0)} \ox {a'}_{\!(1)}) = a \dl(a')$, meaning that
$$
A^{\co H} \subseteq 
\set{a \in A : \dl(aa') = a\dl(a') \mot{for all} a' \in A}.
$$
On the other hand, if 
$(a_{(0)} {a'}_{\!(0)}) \ox (a_{(1)} {a'}_{\!(1)}) = \dl(aa') 
= a \dl(a') = (a {a'}_{\!(0)}) \ox ({a'}_{\!(1)})$ holds for every
$a' \in A$, then $A^{\co H} 
= \set{a \in A : \dl(aa') = a \dl(a') \mot{for all} a' \in A}$.
\end{remark}

Two bialgebras $H$ and $H'$ are said to be \textit{in duality}%
\footnote{%
For Hopf algebras (see subsection~\ref{ssc:Hopf-algebras}), it is also
requested that the corresponding antipodes be dual to each other:
$\duo{S(h)}{h'} = \duo{h}{S'(h')}$.}
if there exists a nondegenerate bilinear pairing
$\duo{\.}{\.} : H \x H' \to k$ such that, for all $h,l \in H$ and
$h',l' \in H'$:
\[
\begin{aligned}
\duo{hl}{h'} &= \duo{h \ox l}{\Dl' h'}_\ox,  \\
\duo{h}{h'l'} &= \duo{\Dl h}{h' \ox l'}_\ox,
\end{aligned}  \qquad
\begin{aligned}
\eps'(h') &= \duo{1_H}{h'},  \\
\eps(h) &= \duo{h}{1_{H'}}.
\end{aligned} 
\]
Here $\duo{\.}{\.}_\ox : (H \ox H) \x (H' \ox H') \to k$ refers to the
nondegenerate bilinear pairing given by
$\duo{h \ox l}{h' \ox l'}_\ox := \duo{h}{h'} \duo{l}{l'}$. It is
usually denoted simply $\duo{\.}{\.}$ with a slight abuse of notation.

\begin{example} 
\label{eg:duality-pair} 
For a finite group $G$ and a field $F$, the bialgebra of functions
$\sO(G)$ and the group (bi)algebra $F[G]$ of
Examples~\ref{eg:bialg-func} and~\ref{eg:alg-grupo} are in duality.
The nondegenerate bilinear pairing 
$\duo{\.}{\.} \: \sO(G) \x F[G] \to F$ is given on basic elements by
$\duo{e_h}{g} := \dl_{h,g}$.

In those examples, $\eps_{\sO(G)}(f) = f(1_G) = \duo{f}{1_{F[G]}}$ and
$\eps_{F[G]}(a) = 1_F = 1_{\sO(G)}(a) = \duo{1_{\sO(G)}}{a}$, for 
$a \in F[G]$, where $1_{\sO(G)}$ is the constant function on~$G$
with value $1_F$. Also,
$\duo{f}{ab} = f(ab) = \duo{\Dl_{\sO(G)} f}{a \ox b}$ and
$\duo{fl}{a} = f(a) l(a) = \duo{f \ox l}{\Dl_{F[G]} a}$, for 
$a,b \in F[G]$ and $f,l \in \sO(G)$.
\end{example}

Given two bialgebras $H$ and $H'$ in duality, and an algebra coaction
$\dl \: A \to A \ox H'$, one can define a Hopf action
$\rho \: H \ox A \to A$ by setting
$\rho(h \ox a) := a_{(0)} \,\duo{h}{a_{(1)}}$, for $h \in H$,
$a \in A$. Clearly
$\rho(h \ox 1_A) = 1_A \,\duo{h}{1_{H'}} = \eps(h)\,1_A$, and
\begin{align*}
\rho(h l \ox a) &= a_{(0)} \,\duo{h l}{a_{(1)}}
= a_{(0)} \,\duo{h \ox l}{a_{(1)} \ox a_{(2)}}_\ox
= a_{(0)} \,\duo{h}{a_{(1)}} \,\duo{l}{a_{(2)}}
\\
&= \rho(h \ox a_{(0)}) \,\duo{l}{a_{(1)}} = \rho(h \ox \rho(l \ox a)),
\end{align*}
for $h,l \in H$ and $a \in A$. Thus the coaction $\dl$ transposes to
an algebra action $\rho$ of $H$ on~$A$. The homomorphism property
of~$\dl$ shows that this action has the Hopf properties:
\begin{align*}
\rho(h \ox ab) &= (ab)_{(0)} \,\duo{h}{(ab)_{(1)}}
= a_{(0)} b_{(0)} \,\duo{h}{a_{(1)} b_{(1)}}
\\
&= a_{(0)} b_{(0)} \,\duo{h_{(1)}}{a_{(1)}} \,\duo{h_{(2)}}{b_{(1)}}
= \bigl( a_{(0)} \,\duo{h_{(1)}}{a_{(1)}} \bigr)
\bigl( b_{(0)} \,\duo{h_{(2)}}{b_{(1)}} \bigr)
\\
&= \rho(h_{(1)} \ox a) \rho(h_{(2)} \ox b),
\end{align*}
and also $\rho(h \ox 1_A) = 1_A \duo{h}{1_{H'}} = \eps(h)\,1_A$ since
$\dl(1_A) = 1_A \ox 1_{H'}$.

\begin{example} 
\label{eg:graded-algebras}
As in Example~\ref{eg:duality-pair}, consider the bialgebras in
duality $\sO(G)$ and $F[G]$, for a finite group $G$. Let $A$ be a
\textit{$G$-graded} algebra, that is, an algebra such that
\[
A = \bigoplus_{g\in G} A_g, \qquad
A_g A_h \subseteq A_{gh} \mot{for all} g,h \in G, \qquad
1_A \in A_{1_G}\,.
\]
Define $\dl\: A \to A \ox F[G]$ by $\dl(a) := a \ox g$ if $a \in A_g$,
whereby $(\id_A \ox \eps_{F[G]})\dl (a)
= (\id_A \ox \eps_{F[G]})(a \ox g) = a \ox 1_F = a = \id_A(a)$, and
\begin{align*}
(\dl \ox \id_{F[G]})\dl(a) 
&= (\dl \ox \id_{F[G]})(a \ox g)
= a \ox g \ox g
\\
&= (\id_A\ox \Dl_{F[G]}) (a \ox g)
= (\id_A \ox \Dl_{F[G]}) \dl(a).
\end{align*}
Thus $A$ is a right comodule of $F[G]$.

Consider 
$\rho :\sO(G) \ox A \to A : f \ox a \mapsto a_{(0)} \duo{f}{a_{(1)}}$,
with $\dl(a) = a_{(0)} \ox a_{(1)}$. In the case $f = e_h$ and
$a \in A_g$, we obtain
$\rho(e_h \ox a) := a \duo{e_h}{g} = \dl_{h,g} a$. In particular, for
$a \in A_g$,
\[
\rho(1_{\sO(G)} \ox a) = \duo{1_{\sO(G)}}{g} \,a = 1_{\sO(G)}(g) a
= 1_G a = a. 
\]
As a result, $\rho(1_{\sO(G)} \ox a) = a$ for every $a \in F[G]$.
Furthermore, for $a \in A_g$,
\[
\rho(f \ox \rho(l \ox a)) =  \rho(f \ox (a \duo{l}{g}))
= a \duo{f}{g} \duo{l}{g}
= a \duo{fl}{g} = \rho(fl \ox a),
\]
for every $f, l \in \sO(G)$ because $\sO(G)$ and $F[G]$ are in
duality. By linearity this holds for all $a \in A$; hence
$\rho \: \sO(G) \ox A \to A$ is an action of $\sO(G)$ on~$A$.

Also, $\rho(e_h \ox 1_A) = \dl_{h,1_G} (1_A) 
= e_h(1_G) 1_A = \eps_{\sO(G)}(e_h) 1_A$, and then
$\rho(f \ox 1_A) = \eps_{\sO(G)}(f) 1_A$ for every $f \in \sO(G)$.
Once more, with $a\in A_g$, $a' \in A_{g'}$:
\begin{align*}
\rho(e_h \ox a a') &= a a' \duo{e_h}{g g'}
= a a' \duo{\Dl_{\sO(G)}(e_h)}{g \ox g'}_\ox
= a a' \duo{{e_h}_{(1)} \ox {e_h}_{(2)}}{g \ox g'}_\ox
\\
&
= a \duo{{e_h}_{(1)}}{g} a' \duo{{e_h}_{(2)}}{g'} 
= \rho({e_h}_{(1)}\ox a) \rho({e_h}_{(2)} \ox a'),
\end{align*}
where we have again identified $\sO(G)\ox\sO(G)$ with $\sO(G \x G)$
and $\Dl_{\sO(G)}(f)(g,g') := f(gg')$. By linearity,
$\rho(f \ox a a') = \rho(f_{(1)}\ox a) \rho(f_{(2)} \ox a')$ for all
$f \in \sO(G)$, meaning that $\rho$ is a left Hopf action of $\sO(G)$ 
on~$A$.

Now we provide the graded algebra $A$ with a coalgebra structure such
that the linear map $\dl \: A \to A \ox F[G]$ given by $\dl(a) = a \ox
g$ if $a \in A_g$ is an algebra coaction of $F[G]$ on $A$. For every
basic element of the algebra $A$ define $\eps(a) := 1_F$ and
$\Dl_A(a) := a \ox a$. Easily, $\eps_A(aa') = \eps_A(a) \eps_A(a')$
for every $a,a' \in A$ and $\Dl_A(1_A) = 1_A \ox 1_A$. Easily,
$\eps_A(aa') = \eps_A(a) \eps_A(a')$ for every $a,a' \in A$ and
$\Dl_A(1_A) = 1_A \ox 1_A$. Also, it follows that
$(\id_A \ox \eps_A) \circ \Dl_A = \id_A 
= (\eps_A \ox \id_A) \circ \Dl_A$. Furthermore, 
$(\id_A \ox \Dl_A) (a \ox a) = a \ox a \ox a 
= (\Dl_A \ox \id_A) (a \ox a)$. Thus, $(A,\Dl_A,\eps_A)$ is a 
coalgebra, and this $A$ is a bialgebra, since
$\Dl(aa') = (aa') \ox (aa') = (a\ox a) (a' \ox a') =  \Dl(a)\Dl(a')$,
for all $a,a' \in A$.

To verify that $\dl$ is a coalgebra homomorphism, note first that
since $1_A \in A_{1_G}$, then $\dl(1_A) = 1_A \ox 1_G
= 1_A \ox 1_{F[G]}$. Also, for a basic element $a \in A_g$,
$\eps_{A \ox F[G]} \dl(a) = \eps_{A \ox F[G]}(a \ox g) = 1_F
= \eps_A(a)$ and
\[
(\dl \ox  \dl) \Dl_A (a) = (\dl \ox  \dl) (a \ox a) 
= (a \ox g) \ox (a \ox g)
= \Dl_{A \ox F[G]} (a \ox g)
= \Dl_{A \ox F[G]} \dl (a).
\]
Hence, $A$ is a right $F[G]$-comodule algebra, once we extend these
identities by linearity.
\end{example}

If two bialgebras $H$ and $H'$ in duality, and there is a coaction
$\dl \: A \to A \ox H' : a \mapsto a_{(0)} \ox a_{(1)}$,
with corresponding Hopf action
$\rho \: H \ox A \to A$ as defined above, then $A^{\co H'} = A^H$. 

In fact, $A^{\co H'} \subseteq A^H$ since $\dl(a) = a \ox 1_{H'}$ for
$a \in A^{\co H'}$. Conversely, if $a \in A^H$, then
$$
a_{(0)} \duo{h}{a_{(1)}} = \rho(h \ox a) = \eps_H(h) a 
= \eps_H(h) a_{(0)} \eps_{H'}(a_{(1)}),
$$
meaning that $\duo{h}{a_{(1)}} = \eps_H(h) \eps_{H'}(a_{(1)})
= \duo{h}{1_{H'}} \duo{1_H}{a_{(1)}}$, for $h \in H$, $a \in A$. From
the nondegeneracy of the bilinear form, it follows that
$a_{(1)} = \duo{1_H}{a_{(1)}} 1_{H'}$ and
$$
\dl(a) = (a_{(0)} \,\duo{1_H}{a_{(1)}}) \ox 1_{H'}
= \rho(1_H \ox a) \ox 1_{H'} = a \ox 1_{H'}\,.
$$

\begin{example} 
\label{eg:un-named}
In Example~\ref{eg:graded-algebras}, for $\rho \: \sO(G) \to A$, we
can write $a = \sum_{g\in G} a_g$ with $a_g \in A_g$, and also
$f = \sum_g f_g e_g$, to obtain
\[
\rho(f \ox a) = \sum_{g,h\in G} \rho(f_h e_h \ox a_g) 
= \sum_{g,h\in G} f_h a_g \,\duo{e_h}{g} 
= \sum_{g,h\in G} f_h a_g \,\dl_{h,g} = \sum_{g\in G} f_g a_g.
\]
Since $\eps(f) a = f(1_G) a = \sum_g f_g e_g (1_G) a = f_{1_G} a$, we
conclude that $A^{\sO(G)} = A_{1_G}$.

On the other hand, for $\dl \: A \to A \ox F[G]$, we may put 
$\dl(a) =: \sum_{g\in G} a_g \ox g$ which equals $a \ox 1_{F[G]}$ if
and only if $a \in A_g$, meaning that
$A^{\co H} = A_{1_G} = A^{\co  \sO(G)}$.
\end{example}


\subsubsection*{Modules}
\label{sss:modules}

Let $A$ be a (unital) $k$-algebra with unit $\eta$. We know
$ra = \eta(r)a = a\eta(r) = ar$, for $r \in k$, $a \in A$. If a
$k$-module $M$ is also an $A$-module in the sense of these notes, 
there is an action $\rho \: A \ox M \to M$. For $M$ to be an
$A$-module in the usual sense, besides the straightforward $A$-linear
properties, necessarily $r(am) = a(rm)$ for all $r \in k$, $a \in A$,
$m \in M$. If we define $\lt = \rho \circ \ox : A \x M \to M$, because
of the $k$-linearity of $\rho$ we obtain $r(a \lt m) = a \lt(rm)$,
meaning that $M$ is an $A$-module in the usual sense. Indeed,
\begin{align*}
r(\rho \circ \ox)(a,m) 
&= r\rho(a \ox m) = \rho(r(a \ox m)) = \rho(ra \ox m) = \rho(ar \ox m)
\\
&= \rho(a \ox rm) = \rho \circ \ox(a,rm).
\end{align*}

On the other hand, assume that $M$ is an $A$-module in the usual
sense. We wish to define an action $\rho \: A \ox M \to M$. Using the
$A$-module structure $\lt \: A \x M \to M$ we define
$a \ox m \mapsto a \lt m$ (extended by linearity). That this is an
action is a consequence of the universal property of tensor products:
$$
A \x M \longto^\lt M = A \x M \longto^\ox A \ox M
\longto^{\tilde\lt} M.
$$

We also need $k$-linearity. Distributivity of the sum is easy, and
$\rho(r(a \ox m)) = \rho(ra \ox m) = ra \lt m = r \rho(a \ox m)$,
thanks to the module structure of~$M$.


\section{Hopf--Galois extensions and Hopf algebras} 
\label{sec:HG-extensions}

A free (and continuous) action of a topological group $G$ on a
topological space $X$ is equivalent to the existence of a (continuous)
canonical map $X \x G \to X \x_G X : (x,g) \mapsto (x,x\.g)$. A
principal bundle (i.e., a free and proper action of $G$ on~$X$) is
equivalent to a Hausdorff quotient space $X/G$ and a homeomorphic
canonical map (see \cite{Liguria} for example). Besides its original
motivation as a generalization for Galois extensions in field theory
to more general structures \cite{ChaseHarrisonRosenberg,
ChaseSweedler}, the concept of a Hopf--Galois extension might be
considered as an extension to the noncommutative realm of the
bijectivity of the canonical map. The original idea comes from
\cite{Schneider1990}, it is given a concrete geometrical
interpretation in~\cite{HajacThesis} (see also
\cite{BrzezinskiHajac2009}) and generalized in
\cite{BrzezinskiMajid2}, among others. If $G$ is a compact Lie group,
the matrix elements of irreducible representations of~$G$ form a Hopf
algebra, making it interesting enough to work out its algebraic
formulation carefully.

Let $A$ and $B$ be algebras over $k$ such that $B \subseteq A$ and
there is an algebra action $\rho$ of a bialgebra $H$ over the algebra
$A$ that leaves $B$ invariant, i.e., $\rho(h \ox b) = \eps(h)b$ for
all $b \in B$. Then we say that $A$ is an \textit{algebra extension of
$B$ by the bialgebra $H$}. If there is a Hopf action of the bialgebra
$H$ over the algebra $A$ then, $A$ is an algebra extension of $A^H$
by~$H$.

Let $A$ and $B$ be algebras over $k$ such that $B \subseteq A$ and
there is an algebra coaction $\dl$ of a bialgebra $H$ over the algebra
$A$ that leaves $B$ coinvariant, i.e., $\dl(b) = b\ox 1_H$ for all $b$
in $B$. Then we say that $A$ is a \textit{coalgebra extension of $B$
by the bialgebra $H$}. If there is an algebra coaction of the
bialgebra $H$ over the algebra $A$, then $A$ is a coalgebra extension
of $A^{\co H}$ by $H$.

For $B$ a subalgebra of an algebra $A$, the \textit{balanced tensor 
product} $A \ox_B A$ is the quotient
\[
A \ox_B A := (A \ox A)/J_B
\] 
with $J_B$ the ideal generated by the elements of the form
$ab \ox c - a \ox bc$ for $a,c \in A$ and $b \in B$. We denote an
element in $A \ox_B A$ as $a \ox_B c = a \ox c + J_B$ for $a,c \in A$.
As a consequence, $ab \ox_B c = a \ox_B bc$ for all $b \in B$.

Let $\dl \: A \to A \ox H$ be an algebra coaction of the bialgebra $H$
on the algebra $A$, and let $B = A^{\co H}$ be the corresponding
coinvariant subalgebra. The \textit{canonical map} for the
coaction~$\dl$ is the $k$-linear map
\[
\can \: A \ox_B A \to A \ox H
: a \ox_B c \mapsto a\dl(c) = a c_{(0)} \ox c_{(1)} \,.
\]
where $\dl(c) = c_{(0)} \ox c_{(1)}$ as usual in Sweedler's notation. 
That `$\can$' is well defined can be seen as follows. If
$a \ox_B c = a' \ox_B c'$ then $a \ox c - a' \ox c'$ belongs to the
ideal $J_B$ arising from $B = A^{\co H}$. Thus, $a \ox c - a' \ox c'$
is a finite sum of elements
$a''b'' \ox c'' - a'' \ox b''c''$ with $a'',c'' \in A$ and 
$b'' \in A^{\co H}$ (i.e., $\dl(b'') = b'' \ox 1_H$). In this way,
with summation understood,
\begin{align*}
(a'' b'')\dl(c'') - a'' \dl(b'' c'') 
&= (a''b'')({c''}_{\!(0)} \ox  {c''}_{\!(1)}) - a'' \dl(b'')\dl(c'')             
\\
&= (a''b''{c''}_{\!(0)}) \ox  {c''}_{\!(1)} 
- a'' (b''{c''}_{\!(0)} \ox {c''}_{\!(1)})  = 0,
\end{align*}
where we use that $\dl$ is an algebra homomorphism and that $A \ox H$
is a left $A$-module. Clearly, $\can$ is left $A$-linear, i.e.,
$\can(ax) = a\,\can(x)$ for $a \in A$, $x \in A \ox_B A$. In other
words, $\can$ intertwines the obvious left-module actions of $A$ on
$A \ox_B A$ and on $A \ox H$. It is also clear that $\can$ intertwines
the $H$-coactions $\id_A \ox_B \dl$ on $A \ox_B A$ and $\id_A \ox \Dl$
on $A \ox H$, that is,
\begin{align*}
(\can \ox \id_H)(\id_A \ox_B \dl)(a \ox_B c)
&= (\can \ox \id_H)(a \ox_B c_{(0)} \ox c_{(1)})
\\
&= a c_{(0)} \ox c_{(1)} \ox c_{(2)} 
= (\id_A \ox \Dl)(\can(a \ox_B c)).
\end{align*}
We therefore say that $\can$ is \textit{right $H$-colinear}.

\begin{remark}
To verify that can is well defined, one could also use
$\can := (\mu \ox \id_H)(\id_A \ox \dl)$. 
If $b \in B$, then $b_{(0)} \ox b_{(1)} = \dl(b) = b \ox 1$ and thus
$c (ba)_{(0)} \ox (ba)_{(1)} = c b_{(0)}a_{(0)} \ox b_{(1)}a_{(1)}
= cb a_{(0)} \ox a_{(1)}$, so that the linear map
$\wt{\can}  \: A \ox A \to A \ox H
: c \ox a \mapsto c a_{(0)} \ox a_{(1)}$ satisfies
$\wt{\can}(c \ox ba) = \wt{\can}(cb \ox a)$ for $b \in B$, so
$\wt{\can}$ vanishes on $J_B$, i.e., $J_B \subseteq \ker(\wt{\can})$.

As usual, $(A \ox A)/\ker(\wt{\can}) \equiv \wt{\can}(A \ox A)
\subseteq A \ox H$, invoking the first isomorphism theorem, so that
$\can \: A \ox_B A \to A \ox H$ is an isomorphism if and only if
$\wt{\can}$ is an epimorphism with $J_B = \ker(\wt{\can})$.
\end{remark}

A coalgebra extension $B \subseteq A$ by a bialgebra $H$ is said to be
a \textit{Hopf--Galois} (or simply an $H$-Galois) \textit{extension}
if the canonical map is bijective.

\begin{remark}
It is evident that the previous definition of a Hopf--Galois extension
is what should be called a \textit{right} Hopf--Galois extension, and
that \textit{left} Hopf--Galois extensions arise from a similar
definition.
\end{remark}

\begin{example} 
\label{eg:H-over-H}
If $H$ be a bialgebra; its coproduct $\Dl$ is an algebra coaction of
the bialgebra $H$ over itself with
$H^{\co H} = \set{h \in H : \Dl(h) = h \ox 1_H}$. If
$h \in H^{\co H}$, then $h = (\eps \ox \id_H)(\Dl h)
= (\eps \ox \id_H)(h \ox 1_H) = \eps(h) 1_H$, which implies 
$H^{\co H} = k$. With $B = H^{\co H} = k$, it follows that
$H \ox_B H = H \ox H$ and
$\can(h \ox h') = h\Dl(h') = h {h'}_{\!(1)} \ox {h'}_{\!(2)}$. As a
result, $\can \in \End(H \ox H)$ in this particular case.
\end{example}

This example is far from trivial. In fact, we shall see in
subsection~\ref{ssc:Hopf-algebras} how a bialgebra~$H$ whose canonical
map is invertible is actually a Hopf algebra (see
Proposition~\ref{pr:Hopf-and-can-map}). More generally, if a bialgebra
$H$ admits a coaction whose canonical map is bijective, then $H$ has a
Hopf algebra structure (see
Proposition~\ref{pr:bialgebra-with-coaction-and-bijective-map} for the
finite-dimensional case and
Proposition~\ref{pr:Hopf-and-faithfully-flat} for the flat case).

\begin{example} 
\label{eg:can-graded-algebra}
Back to Example~\ref{eg:graded-algebras}, consider the algebra
coaction $\dl \: A \to A \ox F[G]$ on the bialgebra $F[G]$ over a
graded algebra $A$, with $B = A^{\co F[G]} = A_1$. The elements of the
balanced tensor product $A \ox_{A_1} A$ can be written as the original
elements of $A \ox A$, keeping in mind that they enjoy the extra
property: $a a_1 \ox a' = a \ox a_1 a'$. In this way,
$\can \: A \ox_{A_1} A \to A \ox F[G]$ is defined, with $a_g \in A_g$,
by
\[
\can \biggl( \sum_{g,g'\in G} a_g \ox a_{g'} \biggr)
:= \sum_{g'\in G} \biggl( \sum_{g\in G} a_g \biggr) a_{g'} \ox g',
\]
so the kernel of $\can$ is given by sums
$\sum a_g \ox a_{g'} \in A \ox_{A_1} A$ such that
$\sum a_g a_{g'} = 0$. As for 
$\can(A \ox_{A_1} A) \subseteq A \ox F[G]$, to obtain a surjective map
we require that any element of $a \in A$ be written in the form
$a = \sum a_g a_{g'}$ with $a_g \in A_g$.

One possibility to achieve a bijective canonical map in this
particular situation is to request the graded algebra to be
\textit{strongly graded}, meaning $A_g A_{g'} = A_{gg'}$ for every
$g,g' \in G$, which immediately implies surjectivity of~$\can$. As for
injectivity, $0 = \sum a_g a_{g'} \equiv \sum a_{gg'}$ if and only if
each $a_{gg'} = 0$, since $A = \bigoplus A_g$.
\end{example}

\begin{example} 
\label{eg:O-of-A}
Let $G$ be a group with a free right action on a set $X$ (i.e.,
$x\.g = x$ if and only if $g = 1_G$). Since in
Example~\ref{eg:bialg-func} $G$ need not actually be finite, we can
take the bialgebra $H = \sO(G)$. Also consider the algebra
$A = \sO(X)$ of functions from $X$ to a field $F$, with pointwise
operations. Because $\al \: X \x G \to X$, we know that
$\al^* \: \sO(X) \to \sO(X \x G)$ is given by
$\al^*(f) = f \circ \al$, i.e., $\al^*(f)(x,g) = f(x\.g)$.

In this example we use several identifications. First, if
$f \in \sO(X)$ and $h \in \sO(G)$, then $f \ox h$ gets identified with
the two-variable function $(x,g) \mapsto f(x)\,h(g)$, so
$\sO(X) \ox \sO(G) \equiv \sO(X \x G)$.
 
There is a coaction $\dl \: \sO(X) \to \sO(X \x G)$ given by 
$(\dl f)(x,g) := f(x\.g) = \al^*(f)(x,g)$. An element $f$ of the
coinvariant subalgebra obeys
$f(x\.g) = \dl(f)(x,g) = f \ox 1_{\sO(G)}(x,g) = f(x)$. Hence,
\[
\sO(X)^{\co\sO(G)} = \set{f \in \sO(X)
: f(x\.g) = f(x) \mot{for all} (x,g) \in X \x G} \equiv \sO(X/G).
\]

With $B = \sO(X/G)$, let us show that the map 
$\can \: \sO(X) \ox_B \sO(X) \to \sO(X) \ox \sO(G)$ is a bijection.
Here, $\can(f \ox_B f')(x,g) = (f\,\dl(f'))(x,g) = f(x)\,f'(x\.g)$.

A free action implies that $X \x G \otto X \x_G X
= \set{(x,x') \in X \x X : x' = x\.g \,\text{for some}\, g \in G}$;
i.e., if $(x,x') \in X \x_G X$ there is a unique $g \in G$ with
$x' = x\.g$, and then $\sO(X \x G) \equiv \sO(X \x_G X)$. The
identification $f \ox f' \equiv [(x,y) \mapsto f(x)\,f'(y)]$ shows
that $\can(f \ox_B f') = (f \ox f')|_{X \x_G X}$, which makes
commutative the diagram
\[
\xymatrix{
\sO(X) \ox \sO(X) \ar[d]_\pi \ar[r]^\equiv &
\sO(X \x X) \ar[d]^{|_{\x_G}}  \\
\sO(X) \ox_B \sO(X) \ar[r]^(.3){\can} & 
\sO(X) \ox \sO(G) \equiv \sO(X \x_G X) \equiv  \sO(X \x G), }
\]
where the second vertical arrow is the restriction of a map 
$X \x X \to F$ to $X \x_G X \to F$.

From the Remark that $\can$ is right $H$-colinear, we know that in
general $J_{\sO(X/G)} \subseteq \ker(\wt{\can})$. As a result,
$\wt{\can} \: \sO(X) \ox \sO(X) \to \sO(X) \ox \sO(G)$, the extension
of~$\can$, is precisely either of the compositions in this diagram,
and our task is to show that the bottom map is bijective. As a
consequence, $\sO(X) \ox_B \sO(X)$ will consist of restrictions of
elements in $\sO(X \x X)$ to $X \x_G X$.

Since the map $|_{\x_G}$ is evidently surjective (for a given map in
$\sO(X \x_G X)$. consider an extension by zero to $X \x X$), we know
that $\wt{\can}$, and hence $\can$, is surjective. If
$0 = \wt{\can}(f \ox f') = (f \ox f')|_{X \x_G X}$, then for any
$a,c \in \sO(X)$ and $b \in B = \sO(X/G)$ there holds
$(f \ox f' + ab \ox c - a \ox bc) (x,x\.g) 
= 0 + a(x) b(x) c(x\.g) - a(x) b(x\.g) c(x\.g) = 0$, since
$b(x) = b(x\.g)$. Thus, $f \ox f' \in J_{\sO(O/X)}$. Therefore
$\ker(\wt{\can}) = J_{\sO(O/X)}$, and so the canonical map is
bijective; and $\sO(X/G) \subseteq \sO(X)$ is a Hopf--Galois extension
by $\sO(G)$.
\end{example}

So far, we have used only bialgebras; Hopf algebras have not yet been
mentioned in this context. Particular situations in which there exists
an antipode is the finite-dimensional case, see
Proposition~\ref{pr:bialgebra-with-coaction-and-bijective-map} below;
or when there exists a faithfully flat $k$-module $A$, see
Proposition~\ref{pr:Hopf-and-faithfully-flat}. These ideas are
introduced in subsection~\ref{ssc:Hopf-algebras}.

Before broaching Hopf algebras and antipodes, and following the line
of ideas in the foregoing examples, we know that in
Example~\ref{eg:H-over-H},
$\can_H(h \ox h') = h\Dl(h') = h {h'}_{\!(1)} \ox {h'}_{\!(2)}$. In
the next result, the bijectivity of this particular canonical map is
matched to the existence of an antipode.

\begin{lemma} 
\label{lm:antipode-vs-can}
Given a bialgebra $(H, \mu, \eta, \Dl, \eps)$, if there exists a 
linear map $S \: H \to H$ such that
$\mu(S \ox \id_H) \Dl = \eta \eps = \mu(\id_H \ox S) \Dl$, then the
canonical map corresponding to the coaction $\dl = \Dl$ is bijective,
meaning that $k \subseteq H$ is a Hopf--Galois extension by the
bialgebra~$H$.
\end{lemma}

\begin{proof}
If such a linear map $S$ exists, satisfying
$S(h_{(1)}) h_{(2)} = \eps(h) 1_H = h_{(1)} S(h_{(2)})$, define the
linear map $\vf \: H \ox H \to H \ox H$ by
$\vf(h \ox h') := h S({h'}_{\!(1)}) \ox {h'}_{\!(2)}$. Then
\begin{align*}
\can_H(\vf(h \ox h')) &= \can_H(h S({h'}_{\!(1)}) \ox {h'}_{\!(2)})
= h S({h'}_{\!(1)}) {h'}_{\!(2)} \ox {h'}_{\!(3)}
\\
& = h {h'}_{\!(1)} S({h'}_{\!(2)}) \ox {h'}_{\!(3)}
= \vf(h {h'}_{\!(1)} \ox {h'}_{\!(2)}) = \vf(\can_H(h \ox h')). 
\end{align*}
Since $h S({h'}_{\!(1)}) {h'}_{\!(2)} \ox {h'}_{\!(3)}
= h \eps({h'}_{\!(1)}) \ox {h'}_{\!(2)}
= h \ox \eps({h'}_{\!(1)}) {h'}_{\!(2)} = h \ox h'$, we conclude that
$\vf$~is precisely the inverse of~$\can_H$. Henceforthe we call it
$\can_H^{-1}$.
\end{proof}

\begin{example} 
\label{eg:bialgebra-bin-S}
Consider the polynomial algebra $F_b[X]$ of
Example~\ref{eg:bialgebra-bin}. Since $\eps(1) = 1$ and
$\eps(X^n) = 0$ for any $n > 0$, the same holds for the operator
$\eta\eps$. In this situation, $1 = \mu(S \ox \id_H) \Dl (1) = S(1)$.
Also $0 = \eta\eps(X) = \mu(S \ox \id_H) \Dl (X) = S(1) X + S(X) 1
= X + S(X)$, implying that $S(X) = -X$. If we assume inductively that
$S(X^k) = (-X)^k$ for $k < n$, then also
\begin{align*}
0 &= \mu(S \ox \id_H) \Dl(X^n)  
= \sum_{k=0}^{n-1} \binom{n}{k} S(X^k) X^{n-k} + S(X^n)
\\
&= \sum_{k=0}^{n-1} \binom{n}{k} (-1)^k X^n + S(X^n)
= -(-X)^n + S(X^n).
\end{align*}
As a result, the operator $S(X^n) = (-X)^n$, extended by linearity, 
satisfies $\mu(S \ox \id_H)\Dl = \eta\eps = \mu(\id_H \ox S)\Dl$, and
then $F \subseteq F_b[X]$ is a Hopf--Galois extension by the
bialgebra $F_b[X]$.
\end{example}

We postpone the reciprocal statement of Lemma~\ref{lm:antipode-vs-can}
until Proposition~\ref{pr:Hopf-and-can-map} below, after we introduce
convolution of linear maps (in this particular situation from $H$
to~$H$). There we shall see that the properties of $S$ are written as
$S * \id_H = 1_* = \id_H * S$.


\subsection{Hopf algebras} 
\label{ssc:Hopf-algebras}

In this section we review Hopf algebras to the extent they relate to
the concept of Hopf Galois extension.

Among the many references available, \cite{Sweedler, PepeJoeHector}
and \cite{JoesGalois} have been inspirational for these notes. Given
two $k$-linear maps $f,g \in \Hom(C,A)$ between a coalgebra $C$ and an
algebra $A$, the \textit{convolution} $f * g \in \Hom(C,A)$ is defined
as the composition of maps%
\footnote{%
On grouplike elements this is just pointwise multiplication: if
$\Dl c = c \ox c$, then $\mu \circ (f \ox g)(c \ox c) = f(c) g(c)$.}
\[
C \longto^\Dl  C \ox C \xrightarrow{f\ox g} A \ox A \longto^\mu A.
\]
$\Hom(C,A)$ is a unital algebra whose unit element is the composition
$\eta_A \eps_C = \eta\eps$. Indeed, associativity can be read out of
the following diagram:
\[
\xymatrix{\ar@{}[dr] 
& C \ox C \ar[dr]^{\Dl\ox\id_C} \ar[rrr]^{(f*g)\ox h} &&&
A \ox A \ar[dr]^\mu & 
\\
C \ar[ur]^\Dl \ar[dr]_\Dl && C \ox C \ox C \ar[r]^{f\ox g\ox h} &
A \ox A \ox A \ar[ur]^{\mu\ox\id_A} \ar[dr]_{\id_A\ox\mu} && A 
\\
& C \ox C \ar[ur]_{\id_C\ox\Dl} \ar[rrr]_{f\ox(g*h)} &&& 
A \ox A \ar[ur]_\mu & }
\]
Also, 
\begin{align*}
f * (\eta\eps)(c) &= \mu \circ (f \ox \eta\eps) \circ \Dl(c)
=  \mu \circ (f \ox \eta\eps) (c_{(1)} \ox c_{(2)})
\\
&= \mu \bigl( f(c_{(1)}) \ox \eta\eps(c_{(2)}) \bigr) 
= f(c_{(1)}) \eps(c_{(2)}) = f(c).
\end{align*}
Similarly, $(\eta\eps) * f = f$.

A bialgebra $H$ is a \textit{Hopf algebra} if there exists a
$k$-module homomorphism $S \: H \to H$ such that
\[
S * \id_H  = \mu(S \ox \id_H)\Dl = \eta\eps = \mu(\id_H \ox S)\Dl
= \id_H * S.
\]
Equivalently,
$S(h_{(1)}) h_{(2)} = \eps(h) 1_H = h_{(1)} S(h_{(2)})$ for all
$h \in H$. The homomorphism $S$ is a convolution inverse for the
identity map on $H$ and it is called the \textit{antipode}. In
diagrams:
\[
\xymatrix{\ar@{}[dr] 
& H \ox H \ar[rr]^{S\ox\id_H} && H \ox H \ar[dr]^\mu & 
\\
H \ar[ur]^\Dl \ar[dr]_\Dl \ar[rr]^\eps && k \ar[rr]^\eta && H
\\
& H \ox H \ar[rr]_{\id_H\ox S} && H \ox H \ar[ur]_\mu & }
\]
Note how this diagram is invariant under ``dualization''.

\begin{remark} 
Because of the relation
$S(h_{(1)}) h_{(2)} = \eps(h) 1_H = h_{(1)} S(h_{(2)})$, one could
omit $\eps$ in the data structure $(H,\mu,\eta,\Dl,\eps,S)$, since it
comes from the identity $\mu \circ (\id \ox S) \circ \Dl = \eps 1_H
= \mu \circ (S \ox \id) \circ \Dl$.
\end{remark}

\begin{remark} 
If $H$ is a Hopf algebra, any grouplike element is invertible with
$S(x) = x^{-1}$. Indeed, $1_H \eps(x) = \eta\eps(x) = \id_H * S(x)
= \mu(\id_H \ox S)(x \ox x) = x S(x)$, where $\eps(x)$ is nonzero by
counitality: $x = (\id_H \ox \eps)\Dl(x) = x \eps(x)$.
\end{remark}

\begin{example} 
\label{eg:alg-grupo-S}
Back to Example~\ref{eg:alg-grupo}: in the group bialgebra $F[G]$
every element of~$G$ is grouplike; it becomes a Hopf algebra with
$S(g) := g^{-1}$ for every $g \in G$.
\end{example}

\begin{example} 
\label{eg:not-Hopf}
A bialgebra that is not a Hopf algebra is given by $\bC[X]$, as in
Example~\ref{eg:bialgebra}.

In fact, if there were an antipode $S$ for this algebra, then 
$1 = \eta\eps(X) = \mu(S \ox \id)(X \ox X) = S(X) X$ would hold,
giving a contradiction.
\end{example}

\begin{example} 
\label{eg:group-poly}
In Example~\ref{eg:bialg-func}, $\sO(G)$ becomes a commutative Hopf
algebra with antipode $S\phi(g) := \phi(g^{-1})$.

This example and the previous one can be considered as the leading
models to define the Hopf algebra axiomatization.

A more involved example considers a compact Lie group $G$ (in
particular, $G$ could be a finite group, of dimension~$0$). For
$g \in G$, let $R_g$ denote the right-translation map of functions in
$C(G)$, $R_g\phi(h) := \phi(hg)$. We say that $\phi$ is a
\textit{representative function} on~$G$ if the translates
$\set{R_g\phi : g \in G}$ generate a finite-dimensional subspace of
$C(G)$. Such functions form a subalgebra of $C(G)$ -- under pointwise
multiplication -- that we also denote by $\sO(G)$; it is sometimes
called the \textit{polynomial algebra} of~$G$.

It is known that $\sO(G)$ is the vector space generated by the 
\textit{matrix elements} $D^\pi_{\xi\eta}$ defined by
$D^\pi_{\xi\eta}(g) := \braket{\xi}{\pi(g)\eta}$ where $\pi$ is a
(unitary) irreducible representation of~$G$, with 
$\xi,\eta \in \sH_\pi$. 
It is also known that the \textit{algebraic} tensor product of two 
copies of $\sO(G)$ is
\[
\sO(G) \ox \sO(G) \equiv \sO(G \x G)  \word{with}
(\phi \ox \psi)(g,h) := \phi(g)\,\psi(h).
\]
Now $\sO(G)$ is a commutative Hopf algebra. See
Example~\ref{eg:Hopf-alg-duality}.
\end{example}

\begin{example} 
\label{eg:enveloping-Lie-algebra}
Let $\g$ be a finite dimensional Lie algebra with ordered basis
$\{X_1,\dots,X_n\}$ and structure constants given by
$[X_i,X_j] = \sum_k c_{ij}^k X_k$. Its (universal) \textit{enveloping
algebra} $\sU(\g)$ is the algebra with generators%
\footnote{%
It is usual to denote with the same symbols the generators of
$\sU(\g)$, making use of the embedding $\g \hookto \sU(\g)$, since
this embedding is injective.}
$\{x_1,\dots,x_n\}$, subject to the relations 
$x_i x_j - x_j x_i = \sum_k c_{ij}^k x_k$ and no others. By the
Poincaré--Birkhoff--Witt theorem, the ordered monomials
$x_1^{r_1} x_2^{r_2} \cdots x_n^{r_n}$ form a vector space basis for
$\sU(\g)$. The algebra homomorphisms given on the generators by
\[
\Dl x = x \ox 1 + 1 \ox x,  \qquad
\eps(1) = 1, \qquad  \eps(x) = 0,  \qquad
S(x) = -x,
\]
make $\sU(\g)$ a Hopf algebra.

This Hopf algebra is \textit{cocommutative}: if $\tau$ denotes the
\textit{flip map} $\tau(x \ox y) := y \ox x$, then $\tau \Dl = \Dl$.
(Observe that commutativity of an algebra is the condition
$\mu \tau = \mu$.)

Recall that an element $x \in H$ of a bialgebra $H$ is called
\textit{primitive} if $\Dl x = x \ox 1 + 1 \ox x$. Thus, the vector
subspace $\g$ of $\sU(\g)$ consists of primitive elements.
\end{example}

\begin{example} 
\label{eg:tensor-product-extension}
Let $A$ be an algebra and $H$ be a Hopf algebra. From
Example~\ref{eg:algebra-times-coalgebra} we know that $A \ox H$ is a
coalgebra. The coaction $\dl^A = \id_A \ox \Dl$ from $A \ox H$ to
$A \ox H \ox H$ makes $A \ox H$ an $H$-comodule.

Also, $(A \ox H, \id_A \ox \Dl)$ is an $H$-comodule algebra. Recall
that in this case,
\[
\Dl_{A\ox H\ox H}(a \ox h \ox h') 
= a \ox h_{(1)} \ox h_{(2)} \ox a \ox {h'}_{\!(1)} \ox {h'}_{\!(2)},
\]
and $\eps_{A\ox H\ox H}(a \ox h \ox h') = \eps(h) \eps(h')$, extended
by linearity. Because
\begin{align*}
\eps_{A\ox H\ox H}\, \dl^A(a \ox h) 
&= \eps_{A\ox H\ox H}(a \ox h_{(1)} \ox h_{(2)})
\\
&= \eps(h_{(1)}) \eps(h_{(2)}) = \eps(h_{(1)}\,\eps(h_{(2)}))
= \eps_{A\ox H}(a \ox h),
\\
(\dl^A \ox \dl^A) \Dl_{A\ox H}(a \ox h)
&= (\dl^A \ox \dl^A)(a \ox h_{(1)} \ox a \ox h_{(2)})
\\
&= a \ox h_{(1)(1)} \ox h_{(1)(2)} \ox a \ox h_{(2)(1)} \ox h_{(2)(2)}
\\
&= \Dl_{A\ox H\ox H}(a \ox h_{(1)} \ox h_{(2)})
= \Dl_{A\ox H\ox H}\, \dl^A(a \ox h),
\end{align*} 
$A \ox H$ is an $H$-comodule algebra.

Now,
\[
(A \ox H)^{\co H} = \set{a \ox h \in A \ox H
: a \ox h_{(1)} \ox h_{(2)} = a \ox h \ox 1_H} \equiv A.
\]
The last equivalence follows as in Example~\ref{eg:H-over-H} where
$\Dl(h) = h \ox 1_H$ implies $h = \eps(h) 1_H$.

With $(A \ox H) \ox_A (A \ox H) \equiv A \ox H \ox H$, the canonical
map is given by
\[
\can \: A \ox H \ox H \to A \ox H \ox H
: a \ox h \ox h' \mapsto a \ox h {h'}_{\!(1)} \ox {h'}_{\!(2)},
\]
with inverse
\[
\can^{-1} \: A \ox H \ox H \to A \ox H \ox H
: a \ox h \ox h' \mapsto a \ox h S({h'}_{\!(1)}) \ox {h'}_{\!(2)}.
\]
As a result, $A \ox H$ is a Hopf--Galois extension of~$A$.
\end{example}

The following is part of Proposition~4.0.1 in \cite{Sweedler}; see
also \cite{Majid} and \cite{Underwood}.

\begin{proposition} 
\label{pr:properties-of-S}
The antipode $S$ of a Hopf algebra is unique and satisfies, for all
$h,h' \in H$:
\begin{enumerate}
\item 
$S(1_H) = 1_H$;
\item 
$S(h h') = S(h') S(h)$;
\item 
$\eps(S(h)) = \eps(h)$;
\item 
$(S \ox S)\Dl h = \tau \Dl S(h)$.
\end{enumerate}
\end{proposition}

\begin{remark}
The first two properties tell us that $S$ is a unital algebra
antihomomorphism. The third and fourth properties tell us that $S$ is
a coalgebra antihomomorphism. We review the proof since it affords us
the opportunity to learn techniques about the role of the antipode as
a substitute for an inverse.
\end{remark}

\begin{proof}
Using the counit, $h = h_{(1)} \eps(h_{(2)}) = \eps(h_{(1)}) h_{(2)}$
for all $h$ in $H$. Thus~(a) follows easily, with $1 = 1_H$, 
$1 = 1_{(1)} S(1_{(2)}) = 1 S(1) = S(1)$.

Also, if $S$ and $S'$ are both antipodes for $H$, then
$$
S'(h) = S'(h_{(1)} \eps(h_{(2)})) = S'(h_{(1)}) \eps(h_{(2)}) 
= S'(h_{(1)}) h_{(2)(1)} S(h_{(2)(2)}).
$$
Similarly,
$$
S(h) =  S(\eps(h_{(1)})h_{(2)}) = \eps(h_{(1)}) S(h_{(2)})
= S'(h_{(1)(1)}) h_{(1)(2)} S(h_{(2)}).
$$
By coassociativity of $\Dl$, both right-hand sides might be written as
\[
\mu((S' \ox \id_H \ox S)(h_{(1)} \ox h_{(2)(1)} \ox h_{(2)(2)}))
= \mu((S' \ox \id_H \ox S)(h_{(1)(1)} \ox h_{(1)(2)} \ox h_{(2)})),
\]
where $\mu \: H \ox H \ox H \to H$ makes sense thanks to the
associativity of~$\mu$. Thus the antipode is unique. We also conclude
that $S(h_{(1)}) h_{(2)(1)}S(h_{(2)(2)}) = S(h) 
= S(h_{(1)(1)})h_{(1)(2)} S(h_{(2)})$ for all $h \in H$, and write it
as $S(h) = S(h_{(1)}) h_{(2)} S(h_{(3)})$.

For~(c) we find that 
\[
\eps(Sh) = \eps(Sh_{(1)}\eps(h_{(2)})) 
= \eps(Sh_{(1)})\eps(h_{(2)}) = \eps((Sh_{(1)}) h_{(2)}) 
= \eps(h) \eps(1_H) = \eps(h).
\]

For~(b), we consider the convolution product on $\Hom(H \ox H, H)$
with unit $\eta\eps = \eta_H \eps_{H\ox H}$ and 
$\eps_{H\ox H}(h \ox h') := \eps(h) \eps(h')$. Recall that on the
tensor product of coalgebras $H \ox H$, $\Dl_\ox = \Dl_{H\ox H}$ is
given by $\Dl_\ox(h \ox h') 
= h_{(1)} \ox {h'}_{\!(1)} \ox h_{(2)} \ox {h'}_{\!(2)}$.

Consider as in \cite{Sweedler} the homomorphisms
$P,N \: H \ox H \to H$ given by $P(h \ox h') := S(hh')$ and
$N(h \ox h') := S(h') S(h)$. We shall show that
$P * \mu = \eta\eps = \mu * N$. In this way, 
$P = P * \eta\eps = P * \mu * N = \eta\eps * N = N$.

We observe that 
\begin{align*}
P * \mu(h \ox h') &= \mu \circ (P \ox \mu) \circ \Dl_\ox(h \ox h')
= \mu(P \ox \mu)(h_{(1)} \ox {h'}_{\!(1)} \ox h_{(2)} \ox {h'}_{\!(2)})
\\
&= \mu(S(h_{(1)} {h'}_{\!(1)}) \ox h_{(2)} {h'}_{\!(2)}) 
= S(h_{(1)} {h'}_{\!(1)}) h_{(2)}{h'}_{\!(2)} 
= S((hh')_{(1)}) (hh')_{(2)}
\\
&= \eps(hh') 1_H = \eps(h) \eps(h') 1_H.
\\
\shortintertext{and}
\mu * N(h \ox h')
&= \mu \circ (\mu \ox N) \circ \Dl_\ox (h \ox h')
= \mu(\mu \ox N)(h_{(1)} \ox {h'}_{\!(1)} \ox h_{(2)} \ox {h'}_{\!(2)})
\\
&= \mu(h_{(1)} {h'}_{\!(1)} \ox S({h'}_{\!(2)}) S(h_{(2)}))
= h_{(1)} {h'}_{\!(1)} S({h'}_{\!(2)}) S(h_{(2)})
\\
&= h_{(1)} \eps(h') S(h_{(2)}) = h_{(1)} S(h_{(2)}) \eps(h')
= \eps(h) \eps(h') 1_H.
\end{align*}
Hence, $P * \mu(h \ox h') = \eta\eps(h \ox h') = \mu * N(h \ox h')$,
as claimed.

For~(d), we consider the convolution product on $\Hom(H, H \ox H)$
with unit $\eta\eps = \eta_{H \ox H} \eps_H$ and
$\eta_{H \ox H}(h \ox h') = \eta(h) \eta(h')$. Remember that on
$H \ox H$, the multiplication $\mu_\ox = \mu_{H\ox H}$ is given by 
$\mu_\ox((g \ox g') \ox (h \ox h')) = gh \ox g'h'$.

Consider, as suggested in \cite{Sweedler}, the homomorphisms
$P',N' \: H \to H \ox H$ given by
$P' = \Dl \circ S : h \mapsto (Sh)_{(1)} \ox (Sh)_{(2)}$; and
$N' = \tau \circ(S \ox S)\circ \Dl 
: h \mapsto S(h_{(2)}) \ox S(h_{(1)})$. We shall show that
$P' * \Dl = \eta\eps = \Dl * N'$, and therefore
$P' = P' * \eta\eps = P' * \Dl * N' = \eta\eps * N' = N'$. Now,
\begin{align*}
\Dl * N'(h)
&= \mu_\ox \circ (\Dl \ox N') \circ \Dl(h)
= \mu_\ox (\Dl \ox N') (h_{(1)} \ox h_{(2)})
= \mu_\ox \bigl( \Dl(h_{(1)}) \ox N'(h_{(2)}) \bigr)
\\
&= \mu_\ox \bigl( h_{(1)} \ox h_{(2)} \ox S(h_{(4)})
\ox S(h_{(3)})\bigr) = (h_{(1)} S(h_{(4)})) \ox (h_{(2)} S(h_{(3)}))
\\ 
&= h_{(1)} S(h_{(3)}) \ox \eps(h_{(2)}) 1
= h_{(1)} \eps(h_{(2)}) S(h_{(3)}) \ox 1
\\
&= h_{(1)} S(h_{(2)}) \ox 1 = \eps(h) 1 \ox 1  = \eta\eps(h),
\\
\shortintertext{and}
P' * \Dl(h)
&= \mu_\ox \circ (P' \ox \Dl) \circ \Dl(h)
= \mu_\ox (P' \ox \Dl) (h_{(1)} \ox h_{(2)})
= \mu_\ox \bigl( P'(h_{(1)}) \ox \Dl(h_{(2)}) \bigr)
\\
&= \mu_\ox \bigl( S(h_{(1)})_{(1)} \ox S(h_{(2)})_{(2)}
\ox h_{(3)} \ox h_{(4)} \bigr)
\\
&= (S(h_{(1)})_{(1)} h_{(3)}) \ox (S(h_{(2)})_{(2)} h_{(4)})
= (S(h_{(1)}) h_{(2)})_{(1)} \ox (S(h_{(2)}) h_{(2)})_{(2)} 
\\
&= (\eps(h)1)_{(1)} \ox (\eps(h)1)_{(2)} 
= \eps(h) 1 \ox 1 = \eta\eps(h),
\end{align*}
which completes the proof.
\end{proof}

\begin{remark}
In general it is not required that $S^2 = \id_H$, but this does happen
if $H$ is commutative or cocommutative: see Proposition~4.0.1(6)
of~\cite{Sweedler} or Theorem~III.3.4.(c) of~\cite{KasselBook}.

The antipode need not be invertible. However, one can use an
invertible antipode $S$ to convert a right action of a Hopf algebra
$H$ on an algebra~$A$ into a left action, and vice versa (see
\cite{GoverZhang} and~\cite{Dabrowskietal}), through the relation
$h \lt a := a \rt S^{-1}h$, employing the standard notations
$h \lt a = \rho(h \ox a)$ and $a \rt h = \rho(a \ox h)$, for left and
right actions respectively. Indeed,
\[
h \lt (h' \lt a) = (a \rt S^{-1}(h')) \rt S^{-1}(h)
= a \rt (S^{-1}(h') S^{-1}(h)) = a \rt S^{-1}(h h') = (h h') \lt a,
\]
and $1_H \lt a = a \rt 1_H = a$, for every $a \in A$ and $h,h' \in H$.
\end{remark}

\medskip

Two Hopf algebras $H$ and $H'$ are said to be \textit{in duality} if
the corresponding bialgebras are in duality, and furthermore
$\duo{S(h)}{h'} = \duo{h}{S'(h')}$ for all $h \in H$, $h' \in H'$.

\begin{example} 
\label{eg:duality-pair-H} 
Back to Example~\ref{eg:duality-pair}, for a finite group $G$ and a
field~$F$: both $\sO(G)$ and $F[G]$ are Hopf algebras with 
respective antipodes $S_{\sO(G)}\phi(g) := \phi(g^{-1})$ and
$S_{F[G]}(g) = g^{-1}$ (see Examples~\ref{eg:group-poly} and
\ref{eg:alg-grupo-S}), which satisfy
$\duo{S_{\sO(G)}(e_h)}{g} = e_h(g^{-1}) = \dl_{h,g^{-1}}
= \duo{e_h}{S_{F[G]}(g)}$, for all $g,h \in G$. It follows that they
are also in duality as Hopf algebras.
\end{example}

\begin{example} 
\label{eg:Hopf-alg-duality}
Let $G$ be a compact Lie group and $\g$ its Lie algebra. Let
$H' = \sO(G)$ be its Hopf algebra of representative functions on~$G$
with structure $\Dl_{\sO(G)} f(g,h) := f(gh)$,
$\eps_{\sO(G)}(f) := f(1)$, $S_{\sO(G)}f(g) := f(g^{-1})$, as in the
finite-dimensional case. If $H = \sU(\g)$, with
$\Dl_{\sU(\g)} X = X \ox 1_{\sU(\g)} + 1_{\sU(\g)} \ox X$,
$\eps_{\sU(\g)}(1) = 1$, $\eps_{\sU(\g)}(X) = 0$ and
$S_{\sU(\g)}(X) = -X$, then both algebras are in duality under
$\duo{X}{f} := Xf(1_G)$ when $X \in \g$,%
\footnote{%
Here $X f(1_G) = \frac{d}{dt}\bigr|_{t=0} f(\exp tX)$.}
extended to $\sU(\g)$ by setting $\duo{XY}{f} := X(Yf)(1_G)$
(see~\cite{JoeHNCG}). We must check that five relations hold. Firstly,
\begin{align*}
\duo{XY}{f} &= X(Yf)(1_G) 
= \frac{d}{dt}\biggr|_{t=0} (Yf)(\exp tX)
= \frac{d}{dt}\biggr|_{t=0} \frac{d}{ds}\biggr|_{s=0}
f(\exp tX \exp sY)
\\
&= \frac{d}{dt}\biggr|_{t=0} \frac{d}{ds}\biggr|_{s=0}
(f \circ \mu)(\exp tX \ox \exp sY) 
= \duo{X \ox Y}{f \circ \mu}_\ox
= \duo{X \ox Y}{\Dl_{\sO(G)}(f)}_\ox.
\end{align*}
Secondly, 
\begin{align*}
\duo{X}{f h} &= X(fh)(1_G) = \bigl( (Xf)h + f(Xh) \bigr)(1_G)
\\
&= Xf(1_G)\, h(1_G) + f(1_G)\, Xh(1_G)
\\
&= \duo{X \ox 1_{\sU(\g)}}{f \ox h}_\ox  
+ \duo{1_{\sU(\g)} \ox X}{f \ox h}_\ox 
\\
&= \duo{X \ox 1_{\sU(\g)} + 1_{\sU(\g)} \ox X}{f \ox h}_\ox
= \duo{\Dl_{\sU(\g)} X}{f \ox h}_\ox.
\end{align*}
Thirdly,
$\eps_{\sO(G)}(f) = f(1_G) = \duo{1_{\sU(\g)}}{f}$.
Fourthly, 
$\eps_{\sU(\g)}(X) = 0 = X(1_{\sO(G)}) = \duo{X}{1_{\sO(G)}}$.
Finally,
\begin{align*}
\duo{S_{\sU(\g)}(X)}{f}
&= \duo{-X }{f} = -X(f)(1_G) = \frac{d}{dt}\biggr|_{t=0} f(\exp -tX)
\\
&= \duo{X}{f \circ [g \mapsto g^{-1}]} =  \duo{X}{S_{\sO(G)}(f)}
\end{align*}
completes the assertion.
\end{example}

A \textit{morphism of Hopf algebras} $f \: H \to H'$ 
is a morphism of bialgebras
\[
f \circ \mu = \mu' \circ (f \ox f), \quad
f \circ \eta = \eta', \quad 
(f \ox f) \circ \Dl = \Dl' \circ f, \quad
\eps = \eps' \circ f,
\]
which further satisfies $S' \circ f = f \circ S$.

\begin{remark}
In the finite-dimensional case, the \textit{linear dual}
$H^* := \Hom(H,k)$ of $H$ becomes a Hopf algebra with the following
operations:
\begin{alignat*}{2}
\mu^*(\al \ox \bt)(h) &:= \al \ox \bt(\Dl(h)), \qquad &
\eps^*(\al) &:= \al(1_H),
\\
\Dl^*(\al)(h \ox h') &:= \al (hh'), &
S^*(\al) &:= \al \circ S.
\end{alignat*}
In the infinite-dimensional case, more caution is required:
see\cite{Sweedler}. The dual must be restricted to
$H^\circ := \set{\al \in H^*
: \al(I) = 0 \mot{for some ideal $I$ with} \dim(H/I) < \infty}$. Note
that $H^\circ = H^*$ in the finite-dimensional case.

If two Hopf algebras $H$ and $H'$ are in duality, then the map
$H \to H^\circ$ given by $h \mapsto \duo{h}{-}$ defines an isomorphism
of Hopf algebras.%
\footnote{%
A bijective homomorphism of Hopf algebras (bialgebra homomorphism
that intertwines the antipodes).}
\end{remark}

In Lemma~\ref{lm:antipode-vs-can} it was shown that the existence of
the antipode implies the invertibility of the canonical map, from the
explicit formula
$\can^{-1}(h \ox h') := h S({h'}_{\!(1)}) \ox {h'}_{\!(2)}$. It is now
time to show that if the canonical map in $\End(H \ox H)$ is
invertible, then $H$ is a Hopf algebra with explicit antipode
$S := (\id_H \ox \eps) \circ \can_H^{-1} \circ (1_H \ox \id_H)$.

\begin{proposition} 
\label{pr:Hopf-and-can-map}
A bialgebra $H$ is a Hopf algebra if and only if the canonical map
$\can_H$ is invertible.
\end{proposition}

We shall base the proof on Lemma~\ref{lm:Koppinen}
from~\cite{Koppinen}. For an algebra~$A$ and a coalgebra~$C$ over a
field~$F$, denote by $\End_A^C(A \ox C)$ the subalgebra of elements
$\psi \in \End(A \ox C)$ that commute with $\mu_A$ on the left and
with $\Dl_C$ on the right, as follows:
\begin{align*}
(\mu_A \ox \id_C)(\id_A \ox \psi) 
&= \psi(\mu_A \ox \id_C) : A \ox A \ox C \to A \ox C,
\\
(\psi \ox \id_C)(\id_A \ox \Dl_C) 
&= (\id_A \ox \Dl_C) \psi : A \ox C \to A \ox C \ox C.
\end{align*}
Since the first property may be written as
$a'\psi (a \ox c) = \psi(a' a \ox c)$ for $a,a' \in A$, $c \in C$,
each $\psi$ is determined by its values on $1_A \ox C$. Indeed,
$\psi(a \ox c) = a \psi(1_A \ox c)$.

The second
property indicates that
$\psi(a \ox c_{(1)}) \ox c_{(2)} = (\id_A \ox \Dl_C)(\psi(a \ox c))$,
for every $a \in A$, $c \in C$. That $\End_A^C(A \ox C)$ is a
subalgebra follows from this: if $a,b \in A$, $c \in C$ with
$\psi'(a \ox c) = a' \ox c'$, then
\begin{align*}
\MoveEqLeft{
(\mu_A \ox \id_C)(\id_A \ox (\psi \circ \psi'))(b \ox a \ox c)
= (\mu_A \ox \id_C)(\id_A \ox \psi) (b \ox \psi'(a \ox c))}
\\
&= \psi(\mu_A \ox \id_C)(b \ox a' \ox c')  
= \psi(b a' \ox c') = \psi(b\,\psi'(a \ox c)) 
\\
&= \psi(\psi'(b a \ox c))
= (\psi \circ \psi')(\mu_A \ox \id_C)(b \ox a \ox c),
\\
\MoveEqLeft{
((\psi \circ \psi') \ox \id_C)(\id_A \ox \Dl_C)(a \ox c)
= ((\psi \circ \psi') \ox \id_C)(a \ox c_{(1)} \ox c_{(2)})}
\\
&= (\psi \circ \psi')(a \ox c_{(1)}) \ox c_{(2)}
= (\psi \ox \id_C)(\psi' \ox \id_C)(\id_A \ox \Dl_C)(a \ox c)
\\
&= (\psi \ox \id_C)(\id_A \ox \Dl_C)\psi'(a \ox c) 
= (\id_A \ox \Dl_C)(\psi \circ \psi')(a \ox c).
\end{align*}

\begin{lemma}[Koppinen] 
\label{lm:Koppinen}
If $A$ is an algebra and $C$ is a coalgebra over a field $F$, then
there is an anti-isomorphism of algebras
$T \: \End_A^C(A \ox C) \to \Hom(A,C)$. 
\end{lemma}

Because on $\Hom(A,C)$ the product is given by convolution, $T$ should
satisfy $T(\psi \circ \psi') = T(\psi') * T (\psi)$ for
$\psi,\psi' \in \End_A^C(A \ox C)$.

\begin{proof}
For each $\phi \in \Hom(C,A)$, define $R\phi \in \End(A \ox C)$ by
$R\phi(a \ox c) := a \phi(c_{(1)}) \ox c_{(2)}$. That
$R\phi \in \End_A^C(A,C)$ follows from
\begin{align*}
\MoveEqLeft{
(\mu_A \ox \id_C)(\id_A \ox R\phi)(b \ox a \ox c)
= (\mu_A \ox \id_C)(b \ox a \phi(c_{(1)}) \ox c_{(2)})}
\\
&= b a \phi(c_{(1)}) \ox c_{(2)} = R\phi(b a \ox c)
= R\phi(\mu_A \ox \id_C)(b \ox a \ox c),
\end{align*}
and
\begin{align*}
\MoveEqLeft{
(R\phi \ox \id_C) (\id_A \ox \Dl_C) (a \ox c)
= (R\phi \ox \id_C) (a \ox c_{(1)} \ox c_{(2)})}
= a \phi(c_{(1)}) \ox c_{(2)} \ox c_{(3)}
\\
&= (\id_A \ox \Dl_C) (a \phi(c_{(1)}) \ox c_{(2)})
= (\id_A \ox \Dl_C) R\phi(a \ox c).
\end{align*}
Note also that, for every $c \in C$, since
$R\phi(1_A \ox c) = \phi(c_{(1)}) \ox c_{(2)}$:
\begin{align*}
(\id_A \ox \eps) R\phi (1_A \ox \id_C)(c) 
&= (\id_A \ox \eps) R\phi (1_A \ox c)
\\
&= \phi(c_{(1)})\eps(c_{(2)}) = \phi(c_{(1)}\eps(c_{(2)})) = \phi(c).
\end{align*}

Noe define a linear map $T \: \End_A^C(A \ox C) \to \Hom(C,A)$ by
$T(\psi) := (\id_A \ox \eps)\psi(1_A \ox \id_C)$, so that
$T(R\phi) = \phi$. Composing $R$ and $T$, we deduce that
\begin{align*}
R \circ T(\psi)(a \ox c)
& = a T(\psi)(c_{(1)}) \ox c_{(2)}
= a (\id_A \ox \eps)\psi(1_A \ox c_{(1)}) \ox c_{(2)}
\\
&= (\id_A \ox \eps)\bigl( a\psi (1_A \ox c_{(1)}) \bigr) \ox c_{(2)}
\\
&= (\id_A \ox \eps)\bigl( \psi (a \ox c_{(1)}) \bigr) \ox c_{(2)}
= (\id_A \ox \eps \ox \id_C)(\psi(a \ox c_{(1)}) \ox c_{(2)})
\\
&= (\id_A \ox \eps \ox \id_C)(\psi \ox \id_C)(a\ox c_{(1)}\ox c_{(2)})
\\
&= (\id_A \ox \eps \ox \id_C)(\psi \ox \id_C)(\id_A \ox \Dl)(a \ox c)
\\
&= (\id_A \ox \eps \ox \id_C)(\id_A \ox \Dl)(\psi(a \ox c))
= \psi(a \ox c),
\end{align*}
because $(\eps \ox \id_C)\Dl_C = \id_C$. As a result, $T$ is a linear
bijection with inverse~$R$.

Now, if $\phi,\phi' \in \Hom(C,A)$, then
\begin{align*}
R\phi \bigl( R\phi'(a \ox c) \bigr)
&= R\phi\bigl( a \phi'(c_{(1)})\ox c_{(2)}) \bigr)
= a R\phi\bigl( \phi'(c_{(1)})\ox c_{(2)}) \bigr)
\\
&= a \phi'(c_{(1)}) \phi(c_{(2)}) \ox c_{(3)}
= a \mu_A(\phi' \ox \phi) \Dl_C(c_{(1)}) \ox c_{(2)}
\\
&= a (\phi' * \phi)(c_{(1)}) \ox c_{(2)} = R(\phi' * \phi)(a \ox c).
\end{align*}
Thus, $R \: \Hom(C,A) \to \End_A^C(A \ox C)$ is an antihomomorphism.
Its inverse $T$ is an antihomomorphism, too, because from
$\phi = T(\psi)$ and $\phi' = T(\psi')$ we conclude that
$T(\psi \circ \psi') = T(R\phi \circ R\phi')
= T \bigl( R(T(\psi') * T(\psi) \bigr)) = T(\psi') * T(\psi)$.
\end{proof}

\begin{proof}[Proof of Proposition~\ref{pr:Hopf-and-can-map}.]
Taking $H = A = C$ in Lemma~\ref{lm:Koppinen}, we obtain two
mutually inverse antihomomorphisms
$R \: \End(H) \to \End_H^H(H \ox H)$ and 
$T \: \End_H^H(H \ox H) \to \End(H)$.

Observe that $\can_H \in \End_H^H(H \ox H)$. Indeed,
$\can_H(h \ox h') := h {h'}_{\!(1)} \ox {h'}_{\!(2)}$ and so
\begin{align*}
(\mu \ox \id) \circ (\id \ox \can_H) (h \ox h' \ox h'')
&= (\mu \ox \id)(h \ox h' {h''}_{\!(1)} \ox {h''}_{\!(2)})
= h h' {h''}_{\!(1)} \ox {h''}_{\!(2)}
\\
&= \can_H(h h' \ox h'') = \can_H \circ(\mu \ox \id)(h \ox h' \ox h''),
\\
\shortintertext{and}
(\can_H \ox \id) \circ (\id \ox \Dl) (h \ox h')
&= (\can_H \ox \id)(h \ox {h'}_{\!(1)} \ox {h'}_{\!(2)})
= h {h'}_{\!(1)} \ox {h'}_{\!(2)} \ox {h'}_{\!(3)}
\\
&= (\id \ox \Dl) (h {h'}_{\!(1)}\ox {h'}_{\!(2)})
= (\id \ox \Dl) \circ \can_H(h \ox h')
\end{align*}
for every $h,h', h'' \in H$.

Assume that $\can_H$ is bijective. Naturally, $\can_H^{-1}$ belongs
to $\End_H^H(H \ox H)$ as well. To check that directly, one may use
the three bijections $\can_H$, $(\id \ox \can_H)$ and
$(\can_H \ox \id)$ to get
\begin{align*}
\MoveEqLeft{
\can_H \bigl( (\mu \ox \id)(\id \ox \can_H^{-1})
- \can_H^{-1}(\mu \ox \id) \bigr) (\id \ox \can_H)}
\\
&= \can_H(\mu \ox \id) - (\mu \ox \id)(\id \ox \can_H) = 0,
\\
\MoveEqLeft{
(\can_H \ox \id) \bigl( (\can_H^{-1} \ox \id)(\id \ox \Dl)
- (\id \ox \Dl)\can_H^{-1} \bigr) \can_H}
\\
&= (\id \ox \Dl)\can_H - (\can_H \ox \id)(\id \ox \Dl) = 0.
\end{align*}

From the proof of Lemma~\ref{lm:Koppinen}, we then find
\[
T(\can_H)(h) = (\id \ox \eps) \can_H (1_H \ox \id)(h)
= (\id \ox \eps) (h_{(1)} \ox h_{(2)}) = h_{(1)} \eps( h_{(2)}) = h,
\]
entailing $T(\can_H) = \id_H$. Hence, there is a unique
$S \in \End(H)$ such that $T(\can_H^{-1}) = S$. From
$\can_H^{-1} \circ \can_H = \can_H \circ \can_H^{-1} = \id_{H\ox H}$
we conclude $\id * S = S * \id = \eta\eps$ in $\End(H)$, and $S$ is an
antipode for~$H$. From the proof of Lemma~\ref{lm:Koppinen}, we know
that $S = T(\can_H^{-1})
= (\id_A \ox \eps) \circ \can_H^{-1} \circ (1_A \ox \id_C)$.
\end{proof}

\begin{remark}
Theorem 1.3.18 of \cite{Timmermann} states that a unital bialgebra $H$
is a Hopf algebra if and only if the associated maps
\begin{align*}
T_1 &= \Dl(1 \ox \id_H) : H \ox H \to H \ox H 
: h \ox h' \mapsto h_{(1)} \ox h_{(2)} h'
\\
T_2 &= (\id_H \ox 1)\Dl : H \ox H \to H \ox H 
: h \ox h' \mapsto h {h'}_{\!(1)} \ox {h'}_{\!(2)}
\end{align*}
are bijective. Observe that in the particular case $\dl = \Dl$, we
obtain $T_2 = \can$ and $T_1 = \can'$, if we redefine the canonical
action from a left coaction $\dl' \: A \to H \ox A$, instead of a
right one, and take again $\dl' = \Dl$ in the particular case $A = H$.

As mentioned by \cite{Timmermann}, this characterization using the
bijectivity of $T_1$ and $T_2$ is adequate for generalizations to
nonunital algebras and was originally used in~\cite{VanDaele3} for the
definition of multiplier Hopf algebras.
\end{remark}

\begin{example} 
\label{eg:Laurent-bialgebra-Hopf}
Back to Example~\ref{eg:Laurent-bialgebra}, with $H = F[X,X^{-1}]$ and
$\Dl(X^k) = X^k \ox X^k$: we get
$\can_H(X^k \ox X^l) = X^{k+l} \ox X^l$ and
$\can_H^{-1} (X^k \ox X^l) = X^{k-l}\ox X^l$. From the previous proof
we know that
\[
S(X^k) = (\id_H \ox \eps) \circ \can_H^{-1} (1_H \ox X^k) 
= (\id_H \ox \eps)(X^{-k} \ox X^k) = X^{-k} \ox \eps(X^k) = X^{-k}.
\]
As a result, $F[X,X^{-1}]$ is a Hopf algebra and a Hopf--Galois
extension of~$F$.
\end{example}

\begin{example} 
\label{eg:alg-grupo-S-2}
From Example~\ref{eg:alg-grupo}, $F[G]$ has canonical map given by
$\can_H(g \ox g') = g g' \ox g'$, with inverse
$\can_H^{-1}(g \ox g') = g (g')^{-1} \ox g'$. Its antipode $S$ is
given by
$S(g) = (\id_H \ox \eps)(g^{-1} \ox g) = g^{-1} \eps(g) = g^{-1}$.
\end{example}

The following result is, in some sense, a generalization of
Proposition~\ref{pr:Hopf-and-can-map}. In
Proposition~\ref{pr:Hopf-and-faithfully-flat} we shall dispense with
the finite-dimensionality hypothesis and replace the field $F$ with a
commutative ring $k$, asking for some sort of flatness on the
algebras.%
\footnote{%
Flatness relates the bijectivity of
$f \ox \id : A \ox H \to A' \ox H$ with that of $f \: A \to A'$.}
When working with modules over rings, it is important to keep in mind
difficulties that may occur in case of torsion; for example,
$\bQ \ox_{\bZ} \bZ_n = 0$.

One of the key points in proving that result is the possibility of
deducing, from the bijectivity of an $F$-algebra homomorphism
$f \: A \to A'$, the bijectivity of
$f \ox \id \: {A \ox H} \to A' \ox H$ for a given algebra~$H$. If
$a' \ox h \in A' \ox H$, then because $f$ is surjective, there is some
$a \in A$ such that
$a' \ox h = \sum f(a) \ox h_j = (f \ox \id)(a \ox h)$, and so
$f \ox \id$ is surjective. If
$(f \ox \id)(a \ox h) = f(a) \ox h = 0$, it is not obvious whether
$a \ox h = 0$. In case $A$ and $H$ are finite-dimensional, with
respective bases $\{a_1,\dots,a_r\}$ and $\{h_1,\dots,h_s\}$, then
$0 = \sum \al_i \bt_j f(a_i) \ox h_j$ for some (unique) elements
$\al_i,\bt_j \in F$. Because the $f(a_i) \ox h_j$ are linearly
independent in $A'\ox H$, each coefficient $\al_i \bt_j$ must be~$0$.
As a result, $a \ox h = \sum \al_i \bt_j a_i \ox h_j = 0$, and
$f \ox \id$ is necessarily injective.

Another required result in the upcoming proof is the converse
possibility to conclude from the bijectivity of
$f \ox \id : A \ox H \to A' \ox H$ that $f \: A \to A'$ is bijective.

If $f(a) = 0$, then $(f \ox \id)(a \ox h) = f(a) \ox h = 0$, so that
$a \ox h = 0$ for each $h \in H$. Once more, assuming that $A$ and $H$
are finite-dimensional (with bases as above), we get
$\sum_i \al_i a_i \ox h_j = 0$ for each~$j$, implying that all
$\al_i = 0$, forcing $a = 0$. For each $a' \in A'$, the surjectivity of
$f \ox \id$ implies that for any basic $h_j \in H$, there is
$a_j \ox h_j \in A \ox H$ with
$a' \ox h_j = (f \ox \id)(a_j \ox h_j) = f(a_j) \ox h_j$. Taking in
$A'$ a finite basis $\{{a'}_{\!1},\dots, {a'}_{\!t}\}$, we write
$0 = (a' - f(a^j)) \ox h_j = {\al'}_{\!k,j} {a'}_{\!k} \ox h_j$, just
as before, to deduce that each ${\al'}_{\!k,j} = 0$. In particular,
$a = f(a_j)$, and $f$ must be surjective.

\begin{proposition} 
\label{pr:bialgebra-with-coaction-and-bijective-map}
If a finite-dimensional bialgebra $H$ over a field $F$ admits an
algebra coaction $\dl \: A \to A \ox H$ with $A$ finite-dimensional
over $F$, for which the canonical map $\can \: A \ox_B A \to A \ox H$
is bijective, then $H$ has an antipode, i.e., $H$ is a Hopf algebra.
\end{proposition}

\begin{proof} 
Consider the following diagram:
\[
\xymatrix@R=3pc@C=3pc{
(A \ox H) \ox_B A \ar[d]_{\can_{13}} &
A \ox_B A \ox_B A \ar[r]^{\id_A\ox\can} \ar[l]_{\can\ox\id_A} &
A \ox_B A \ox H \ar[d]^{\can\ox\id_H} 
\\
A \ox H \ox H \ar[rr]^{\id_A\ox\can_H} & &
A \ox H \ox H }
\]
where
$\can_{13} = \switch{2}{3} \circ (\can \ox \id_H) \circ \switch{2}{3}
= \switch{1}{2} \circ (\id_H \ox \can) \circ \switch{1}{2}$, using the
flip operator. In other words,
$\can_{13}(a \ox h \ox_B a') := a\,{a'}_{\!(0)}\ox h\ox {a'}_{\!(1)}$.
This diagram is commutative (departing from the middle top):
\begin{align*}
\MoveEqLeft{
(\can \ox \id_H) (\id_A \ox \can) (a \ox_B b \ox_B a')
= (\can \ox \id_H)\bigl( a \ox_B b {a'}_{\!(0)}\ox {a'}_{\!(1)} \bigr)}
\\
&= a b_{(0)} {a'}_{\!(0)(0)} \ox b_{(1)} {a'}_{\!(0)(1)}
\ox {a'}_{\!(1)}
= a b_{(0)} {a'}_{\!(0)} \ox b_{(1)} {a'}_{\!(1)(1)}
\ox {a'}_{\!(1)(2)}
\\
&= (\id_A \ox \can_H) \bigl(
a b_{(0)} {a'}_{\!(0)} \ox b_{(1)} \ox {a'}_{\!(1)} \bigr)
= (\id_A \ox \can_H) \can_{13} \bigl(
a b_{(0)} \ox b_{(1)} \ox_B a' \bigr)
\\
&= (\id_A \ox \can_H) \can_{13} (\can \ox \id_A) (a \ox_B b \ox_B a').
\end{align*}

Every map in this diagram (except from the one at the bottom) is of
the form $\can \ox \id$ or $\id \ox \can$, and each of these is also a
bijection because $\can$ is bijective.

Commutativity of the diagram implies that $\id_A \ox \can_H$ is
bijective, which in turn implies that
$\can_H \: H \ox H \to H \ox H$ is bijective. The result follows from
Proposition~\ref{pr:Hopf-and-can-map}.
\end{proof}

Let us now focus on how Hopf algebras relate to Hopf--Galois
extensions. What follows is the main result of \cite{Schauenburg1997}.
Our proof is taken from \cite{SchauenburgFIC} (originally
from~\cite{Takeuchi}).

Given an algebra $A = (A, \mu, \eta)$, its \textit{opposite algebra}
$A^\opp = (A, \mu^\opp, \eta)$ has the reversed
multiplication operation $\mu^\opp(a \ox a') = \mu(a' \ox a)$ for all
$a,a' \in A$. Evidently $A^\opp$ is associative, and if $A$ has a
unit, it is also the unit for $A^\opp$. Note that $A$ is commutative
if and only if $A^\opp = A$ (or $\mu^\opp = \mu$).

If $(C, \Dl, \eps)$ is a coalgebra, its 
\textit{co-opposite coalgebra} $C^\copp = (C,\Dl^\opp,\eps)$ has the
\textit{opposite coproduct} $\Dl^\opp(c) = c_{(2)} \ox c_{(1)}$, i.e.,
$\Dl^\opp = \tau \circ \Dl$. $C^\copp$ is coassociative because
\[
(\id_C \ox \Dl^\opp) \circ \Dl^\opp(c)
= c_{(2)} \ox c_{(1)(2)} \ox c_{(1)(1)}
= c_{(2)(2)} \ox c_{(2)(1)} \ox c_{(1)}
= (\Dl^\opp \ox \id_C) \circ \Dl^\opp(c).
\]
A coalgebra is called \textit{cocommutative} if $C^\copp = C$.

If $A = (A, \mu, \eta, \Dl, \eps)$ is a bialgebra, then all of the
following are bialgebras, too \cite{KlimykSchmudgen}:
\begin{itemize}
\item
the \textit{opposite bialgebra} 
$A^\opp = (A, \mu^\opp, \eta, \Dl, \eps)$,
\item
the \textit{co-opposite bialgebra} 
$A^\copp = (A, \mu, \eta, \Dl^\opp, \eps)$,
\item
the \textit{opposite co-opposite bialgebra} 
$A^{\opp,\copp} = (A, \mu^\opp, \eta, \Dl^\opp, \eps)$.
\end{itemize}

Furthermore, if $S$ is an antipode for $A$, then it is also an
antipode for $A^{\opp,\copp}$, which is called the \textit{opposite
co-opposite Hopf algebra} (see \cite[Corollary III.3.5]{KasselBook}).
As for the other two bialgebras, one can show that the following are
equivalent:
\begin{itemize} 
\item
the antipode $S \: A \to A$ is an invertible linear map, with inverse
$S^{-1}$,
\item
$S^{-1} \: A^\opp \to A^\opp$ is an antipode,
\item
$S^{-1} \: A^\copp \to A^\copp$ is an antipode.
\end{itemize}

\begin{remark}
If $S$ is an invertible antipode for $H$, then $S(h) = l$ and
$S(h') = l'$ entail
\[
S^{-1}(ll') = S^{-1}(Sh\,Sh') = S^{-1}(S(h' h)) = h' h 
= S^{-1}(l') S^{-1}(l).
\]
Also, since $\eps(Sh) = \eps(h)$ and
$S(h_{(1)}) \ox S(h_{(2)}) = (Sh)_{(2)} \ox (Sh)_{(1)}$ -- see
Proposition~\ref{pr:properties-of-S},
\begin{align*}
S^{-1} * \id(l) 
&= \mu (S^{-1} \ox \id)(l_{(1)} \ox l_{(2)}) 
= \mu (S^{-1} \ox \id)((Sh)_{(1)} \ox (Sh)_{(2)}) 
\\
&= \mu (S^{-1} \ox \id)(S(h_{(2)}) \ox S(h_{(1)})) 
= h_{(2)}\, S(h_{(1)}) = \eps(h) = \eps(Sh) = \eps(l),
\end{align*}
and similarly $\id * S^{-1}(l) = \eps(l)$. Therefore, the first
statement above implies the second.
\end{remark}

\medskip

Now consider a coaction $\dl \: A \to A \ox H$. There is a
corresponding coaction $\dl^\opp \: A^\opp \to A^\opp \ox H^\opp$.
Since one only modifies $\mu_A \mapsto \mu_H^\opp$ and
$\mu_H \mapsto \mu_H^\opp$, all coaction properties of~$\dl$ remain
unchanged, except possibly for $\dl(a a') = \dl (a) \dl(a')$, which
becomes $\dl(a' a) = \dl(a') \dl(a)$. Thus, the same (undecorated)
$\dl$ defines a coaction $\dl \: A^\opp \to A^\opp \ox H^\opp$. Now
$B^\opp = (A^\opp)^{\co H^\opp} 
= \set{a \in A^\opp : \dl(a) = a \ox 1_{H^\opp}}$, without much
variation. However, the new canonical map
$\can^\opp \: A^\opp \ox_{B^\opp} A^\opp \to A^\opp \ox H^\opp$ is
given by
\[
\can^\opp(a \ox a') = (\mu^\opp \ox \id_H) \dl (a \ox a')
= {a'}_{\!(0)} a \ox {a'}_{\!(1)} \,.
\]

The following result comes from~\cite{Schneider1990}. Our proof is
indicated in~\cite{SchauenburgFIC}.

\begin{lemma}[Schneider] 
\label{lm:HS-to-HopS}
If the Hopf algebra $H$ has bijective antipode and $A$ is a
$H$-comodule algebra, then $A$ is a $H$-Galois extension if and only
if $A^\opp$ is a $H^\opp$-Galois extension.
\end{lemma}

\begin{proof}
As recently observed, 
$\can^\opp \: A^\opp \ox_{B^\opp} A^\opp \to A^\opp \ox H^\opp$ gets
identified with the linear map $\can' \: A \ox_B A \to A \ox H
: a \ox a' \mapsto a_{(0)} a' \ox a_{(1)}$. Consider the diagram
\[
\xymatrix{
A \ox_B A \ar[r]^\can \ar[dr]_{\can'} 
& A \ox H \ar[d]^{(\id_A\ox\mu_H) \circ (\dl \ox S)}  \\
& A \ox H, }
\]
where $(\id_A \ox \mu_H) \circ (\dl \ox S) \: A \ox H \to A \ox H
: a \ox h \mapsto a_{(0)} \ox a_{(1)}\,Sh$ is bijective, with inverse
given by $a \ox h \mapsto a_{(0)} \ox S^{-1}h\,a_{(1)}$. Indeed,
\begin{align*}
(\id_A \ox \mu_H) \circ (\dl \ox S)(a_{(0)} \ox S^{-1}h\,a_{(1)})
&= a_{(0)} \ox a_{(1)}  S(S^{-1}h\,a_{(2)}))
= a_{(0)} \ox a_{(1)} S(a_{(2)}) h
\\
&= a_{(0)} \ox \eps_H(a_{(1)}) h = a \ox h,
\\
\bigl( (\id_A \ox \mu_H) \circ (\dl \ox S) \bigr)^{-1}
(a_{(0)} \ox a_{(1)}\,Sh)
&= a_{(0)} \ox S^{-1}(a_{(1)}\,Sh) a_{(2)}  
= a_{(0)} \ox h S^{-1}(a_{(1)}) a_{(2)}
\\
&= a_{(0)} \ox h\,\eps_H(a_{(1)}) = a \ox h,
\end{align*}
where we use $S^{-1}(h h') = S^{-1}(h') S^{-1}(h)$ and
$S^{-1} * \id = \eta\eps = \id * S^{-1}$.

The commutativity of this diagram follows from the next relations.
Firstly,
\begin{align*}
(\dl \ox S) \can(a \ox a') 
&= (\dl \ox S) (a {a'}_{\!(0)} \ox {a'}_{\!(1)})
= (a{a'}_{\!(0)})_{(0)} \ox (a{a'}_{\!(0)})_{(1)} \ox S({a'}_{\!(1)})
\\
&= (a_{(0)} \ox a_{(1)})({a'}_{\!(0)} \ox {a'}_{\!(1)}) 
\ox S({a'}_{\!(2)})
\\
&= a_{(0)} {a'}_{\!(0)} \ox a_{(1)} {a'}_{\!(1)} \ox S({a'}_{\!(2)}),
\end{align*}
because $\dl(a{a'}_{\!(0)}) = \dl(a) \dl({a'}_{\!(0)})$. Secondly,
\begin{align*}
(\id_A \ox \mu_H) (\dl \ox S) \can(a \ox a')
&= (\id_A \ox \mu_H) \bigl(
a_{(0)} {a'}_{\!(0)} \ox a_{(1)}{a'}_{\!(1)} \ox S({a'}_{\!(2)}) \bigr)
\\
&= a_{(0)} {a'}_{\!(0)} \ox a_{(1)}{a'}_{\!(1)} S({a'}_{\!(2)})
= a_{(0)} {a'}_{\!(0)} \ox a_{(1)} \eps({a'}_{\!(1)})
\\
&= a_{(0)} {a'}_{\!(0)} \eps({a'}_{\!(1)}) \ox a_{(1)} 
= a_{(0)} a' \ox a_{(1)} = \can'(a \ox a'),
\end{align*}
since $(\id_A \ox \eps_H) \dl = \id_A$. As a result, $\can$ is
invertible if and only if $\can'$ is invertible.
\end{proof}

In the proof of
Proposition~\ref{pr:bialgebra-with-coaction-and-bijective-map} in the
finite-dimensional case, we showed that the bijectivity of 
$\id_A \ox \can_H \: A \ox H \ox H \to A \ox H \ox H$ implies that the
map $\can_H \: H \ox H \to H \ox H$ is also bijective. In
\cite{Schauenburg1997}, with a shorter proof due to \cite{Takeuchi},
see also~\cite{SchauenburgFIC}, this result does not require finite
dimensionality, but asks for a flatness property of the algebras.

In general, if $0 \to M' \to M \to M'' \to 0$ is a short exact
sequence of $k$ modules and $N$ is a $k$-module (or a $k$-algebra), it
follows that $M' \ox N \to M \ox N \to M'' \ox N \to 0$ is exact
\cite[Proposition~XVI.2.6]{Lang}. $N$ is said to be \textit{flat} (or
``tensor exact'') if moreover the tensored sequence
$0 \to M' \ox  N \to M \ox  N \to M'' \ox N \to 0$ is exact. If
$f \: M \to M'$ is a bijection and $N$ is flat, then 
$f \ox \id_N \: M \ox N \to M' \ox N$ is also a bijection.

A $k$-module $M$ is \textit{faithfully flat} if it is flat and for
every nonzero $k$-module $N$ the module $M \ox N$ is also nonzero,
i.e., the functor $T_M \: N \mapsto M \ox N$ is faithful. It may be
checked \cite{Lang} that the faithful flatness of~$M$ is equivalent
with the condition: a sequence of $k$ modules $N'' \to N \to N'$ is
exact if and only if the tensored sequence
$N'' \ox M \to N \ox M \to N' \ox M$ is exact. A $k$-algebra is
faithfully flat if it is so as a $k$-module. With $A$ faithfully flat,
$f \: M \to M'$ is a bijection if and only if
$f \ox \id_A \: M \ox A \to M' \ox A$ is bijective.

The following is \cite[Lemma~2.1.5]{SchauenburgFIC}.

\begin{proposition}[Schauenburg] 
\label{pr:Hopf-and-faithfully-flat}
Let $H$ be a flat $k$-bialgebra, and let $A$ be a (right) $H$--Galois
extension of its coinvariant subalgebra $A^{\co H}$, such that $A$ is
faithfully flat as a $k$-module. Then $H$ is a Hopf algebra.
\end{proposition}

\begin{proof}
To verify that $\can_H$ is invertible in $\End(H \ox H)$, we consider
the same diagram used in the proof of
Proposition~\ref{pr:bialgebra-with-coaction-and-bijective-map}. Since
$\can \: A \ox_B A \to A \ox H$ is bijective and $H$ is flat, every
arrow in that diagram other than the bottom one is bijective, and from
the commutativity of the diagram, we conclude that
$\id_A \ox \can_H \: A \ox H \ox H \to A \ox H \ox H$ is also a
bijection. Then, since $A$ is faithfully flat, we deduce that
$\can_H \: H \ox H \to H \ox H$ is invertible.
\end{proof}

\goodbreak 


\section{Quantum groups and Hopf algebras} 
\label{sec:Qgroups-HG-ext}

\begin{flushright}
\begin{minipage}[t]{8cm}
\small \itshape
The most interesting Hopf algebras are those which are neither
commutative nor cocommutative.
\par \hfill \upshape
--- V. G. Drinfel'd
\end{minipage}
\end{flushright}

What is a quantum group and how does it fit into our study of
Hopf--Galois extensions?

As indicated in \cite{VanDaele}, there is a general consensus that
Hopf algebras are the starting point for a study of quantum groups;
from there, many venues are possible. The author of \cite{VanDaele}
classifies the hierarchy of the subject as follows: finite quantum
groups at the beginning; discrete quantum groups; compact quantum
groups and algebraic quantum groups; to finally arrive at locally
compact quantum groups. In the historical record the compact quantum
groups appear before the discrete ones.

Compact and locally compact quantum groups are built in the setting of
unital $C^*$-algebras (with a coproduct satisfying a density property,
not mentioning a counit or antipode). Discrete quantum groups appear
as (multiplier) Hopf $*$-algebras that are dual to compact quantum
groups: see \cite{PodlesWoronowicz} and~\cite{VanDaele2}. Algebraic
quantum groups are (multiplier) Hopf $*$-algebras ``with positive
integrals''. Many of these requirements are not addressed here; thus,
we restrict to the ``algebraic'' setting. Our goal here is to offer an
algebraic interpretation of quantum groups within our framework.
For example, the initial level of finite quantum groups reduces simply
to direct sums of matrix algebras, the simplest instance being
Example~\ref{eg:bialg-func}. An interesting feature of some finite
quantum groups is that $S^2 = \id$, as in Example~\ref{eg:group-poly}.
In Example~\ref{eg:smallest-2}, notice that $S^2 \neq \id_H = S^4$.
See~\cite{Taft} for examples of finite-dimensional Hopf algebras over
a field, whose antipodes have arbitrary even orders greater than~$4$.

The most elegant definition of compact quantum group is the one given
by Woronowicz: a \textit{compact quantum group} is a pair $(A, \Dl)$,
where $A$ is a unital $C^*$-algebra and $\Dl \: A \to A \ox A$ is a
unital $*$-homomorphism, called comultiplication, such that
\begin{enumerate}[label={(\roman*)}]
\item 
$(\Dl \ox \id_A) \Dl = (\id_A \ox \Dl)\Dl : A \to A \ox A \ox A$, and
\item 
the vector spaces $\linspan \set{(a \ox 1)\Dl(b) : a,b \in A}$ and
$\linspan \set{\Dl(a)(1 \ox b) : a,b \in A}$, denoted respectively by
$(A \ox 1)\Dl(A)$ and $\Dl(A)(1 \ox A)$, are dense in $A \ox A$.
Compare, for example, \cite{VanDaele, NeshveyevTuset}.
\end{enumerate}

Property~(i) is just the coassociativity of the comultiplication as in
any bialgebra, and it is purely algebraic. Observe that no counit or
antipode appears in the definition. Remember that in a Hopf algebra
one could omit explicit mention of $\eps$ in the data structure
because of the relations $\mu \circ (\id \ox S) \circ \Dl = \eps 1_H 
= \mu \circ (S \ox \id) \circ \Dl$. Similarly for $S$; indeed, from
the proof of Proposition~\ref{pr:Hopf-and-can-map},
$T \: \End_H^H(H \ox H) \to \End(H)$ is an antihomomorphism satisfying
$T(\can_H) = \id_H$, with $\can_H \in \End_H^H(H \ox H)$ the canonical
map. Hence, there is a unique $S \in \End(H)$ such that
$T(\can_H^{-1}) = S$. As stated in Eq.~(4.7) after Theorem~4.8
of~\cite{BohmBook}: ``for a bialgebra to be a Hopf algebra is a
\textit{property, not a further structure}''. See also Proposition
1.3.22 of~\cite{Timmermann}.

Property~(ii) is the \textit{cancellation property}. The density
requirement appeals to an analytical setting (as is the case of
$C^*$-algebras, von~Neumann algebras, etc.). Nevertheless, we can say
something about this cancellation property from an algebraic point of
view.

\begin{proposition} 
\label{pr:Hopfalgebra-algebraic-quantum-group}
Any Hopf algebra $H$ satisfies the following properties, which are
algebraic versions of the requirements for a compact quantum group.
\begin{enumerate}[label={\textup{(\roman*$'$)}}]
\item 
$(\Dl \ox \id_H) \Dl = (\id_H \ox \Dl)\Dl : H \to H \ox H \ox H$, and
\item 
$(H \ox 1)\Dl(H) = H \ox H = \Dl(H)(1 \ox H)$. 
\end{enumerate}
We shall call \textup{(ii$'$)} the cancellation property as well.
\end{proposition}

\begin{proof}
From Proposition~\ref{pr:Hopf-and-can-map} we know that in every Hopf
algebra $H$, the canonical map
$\can_H \: H \ox H \to H \ox H : h \ox h' \mapsto (h \ox 1)\Dl(h')$ is
invertible. It follows that
$(H \ox 1)\Dl(H) = \can(H \ox H) = H \ox H$ in any Hopf algebra.
Similarly, $H \ox H = \Dl(H)(1 \ox H)$.
\end{proof}

In case the unital algebra $A$ is generated by elements $u_{ij}$, for
$i,j \in \{1,\dots,n\}$, where $[u_{ij}]_{i,j}$ is an invertible
matrix in $M_n(A)$ with inverse $[v_{kl}]_{k,l}$ and whose coproduct
satisfies $\Dl(u_{ij}) = \sum_k u_{ik} \ox u_{kj}$, the pair $(A,\Dl)$
is called a \textit{matrix quantum group}.

\begin{proposition} 
\label{pr:co-quant-mat-group}
Every matrix quantum group satisfies 
$(A \ox 1)\Dl(A) = \Dl(A)(1 \ox A) = A \ox A$.
\end{proposition}

\begin{proof}
The set 
$B = B(A) = \set{a \in A : 1 \ox a = \sum_i (a_i \ox 1)\Dl(b_i)
\mot{for some} a_i,b_i}$ is a subalgebra of~$A$. Indeed, 
\begin{align*}
1 \ox aa' &= (1 \ox a)(1 \ox a')
= \sum_i (1 \ox a)(a'_i \ox 1)\Dl(b_i')
= \sum_i ({a'}_{\!i} \ox a)\Dl({b'}_{\!i})
\\
&= \sum_{i,j} ({a'}_{\!i} \ox 1)(a_j \ox 1) \Dl(b_j) \Dl(b_i')
= \sum_{i,j} ({a'}_{\!i} a_j \ox 1) \Dl(b_j {b'}_{\!i}).
\end{align*}
Similarly, $B' = \set{a \in A : a \ox 1 = \sum_i \Dl(a_j)(1 \ox b_j) 
\mot{for some} a_j,b_j}$ is a subalgebra of~$A$.

Each generator $u_{ij}$ belongs to $B$. To see this, notice that
\[
1 \ox u_{ij} = \sum_k \dl_{ik} \ox u_{kj}
= \sum_{k,l} v_{il} u_{lk} \ox u_{kj}
= \sum_{k,l} (v_{il} \ox 1)(u_{lk} \ox u_{kj})
= \sum_l (v_{il} \ox 1) \Dl(u_{lj}).
\]
As a result, $B = A$. Similarly,
$u_{ij} \ox 1 = \sum_k \Dl(u_{ik}) (1 \ox v_{kj})$ and $B' = A$.

Lastly, the relation
\[
u_{ij} \ox u_{kl} = (u_{ij} \ox 1) (1 \ox u_{kl})
= \sum_r (u_{ij} \ox 1) (v_{kr} \ox 1) \Dl(u_{rl})
= \sum_r (u_{ij} v_{kr} \ox 1) \Dl(u_{rl})
\]
implies that $A \ox A$ is contained in $(A \ox 1) \Dl(A)$.
Similarly, $A \ox A \subseteq \Dl(A) (1 \ox A)$ and then
$A \ox A = (A \ox 1)\Dl(A) = \Dl(A)(1 \ox A)$.
\end{proof}

\begin{example} 
\label{eg:SLq2-cancellation}
In Example~\ref{eg:SLq2} we exhibited $SL_q(2)$, with
$q \in \bC \setminus \{0\}$, as a noncommutative deformation of the
coordinate algebra $SL(2)$ for the special linear group $SL_2(\bC)$.
Both $SL_q(2)$ and $SL(2)$ are matrix quantum groups since
\[
\Dl \twobytwo{a}{b}{c}{d}
= \twobytwo{a}{b}{c}{d} \ox \twobytwo{a}{b}{c}{d}, \word{with}
[v_{ij}]_{i,j} = \twobytwo{d}{-b}{-c}{a}.
\]
By Proposition~\ref{pr:co-quant-mat-group}, they satisfy the
cancellation property.
\end{example}

\begin{example} 
\label{eg:GL2-Hopf-algebra}
The group $SL_2(\bC)$ sits naturally inside the \textit{special linear
group} $GL_2(\bC)$, the $2 \x 2$ invertible complex matrices with
(commutative) coordinate algebra
\[
GL(2) = \sO(G_2(\bC)) := \bC[a,b,c,d,t]/\bigl( t(ad - bc) -1 \bigr),
\]
with (noncocommutative) Hopf algebra structure given by the morphisms 
\begin{align*}
\Dl \twobytwo{a}{b}{c}{d}
&= \twobytwo{a}{b}{c}{d} \ox \twobytwo{a}{b}{c}{d}, \quad
\Dl(t) = t \ox t,
\\
S \twobytwo{a}{b}{c}{d} &= t \twobytwo{d}{-b}{-c}{a},  \quad
S(t) = ad - bc,
\end{align*}
and $\eps(a) = \eps(d) = \eps(t) = 1$, $\eps(b) = \eps(c) = 0$.

$GL(2)$ is not a matrix quantum group, because it has five generators.
Nevertheless, we can write
\[
\Dl \begin{pmatrix} 
a & b & 0 \\ c & d & 0 \\ 0 & 0 & t \end{pmatrix}
= \begin{pmatrix} 
a & b & 0 \\ c & d & 0 \\ 0 & 0 & t \end{pmatrix}
\ox \begin{pmatrix} 
a & b & 0 \\ c & d & 0 \\ 0 & 0 & t \end{pmatrix}
\word{with} \begin{pmatrix} 
a & b & 0 \\ c & d & 0 \\ 0 & 0 & t \end{pmatrix}^{-1}
= \begin{pmatrix} 
dt & -bt & 0 \\ -ct & at & 0 \\ 0 & 0 & ad-bc \end{pmatrix},
\]
and apply the same formula
$1 \ox u_{ij} = \sum_l (v_{il} \ox 1) \Dl(u_{lj})$ as in the proof of
Proposition~\ref{pr:co-quant-mat-group},
to conclude that every generator is in $B(GL(2))$.

Also, the relation
\[
u_{ij} \ox u_{kl} = (u_{ij} v_{k1} \ox 1) \Dl(u_{1l})
+ (u_{ij} v_{k2} \ox 1) \Dl(u_{2l})
+ (u_{ij} v_{k3} \ox 1) \Dl(u_{3l})
\]
holds and implies that $B(GL(2))$ is a subalgebra of $GL(2)$. As a
result, $GL(2)$ satisfies the cancellation property, which can also be
deduced from $GL(2)$ being a Hopf algebra.
\end{example}

Let us give a deformation of $GL(2)$ that is neither commutative nor
cocommutative. Denote by $M_q(2)$ the complex algebra of $2 \x 2$
matrices on four generators $a,b,c,d$ subject to the relations:
\[
ba = q ab, \quad ca = q ac, \quad db = q bd, \quad dc = q cd,
\quad
bc = cb, \quad ad - q^{-1} bc = da - q cb. 
\]
With $q = 1$ we recover the complex algebra of $2 \x 2$ matrices on
four commuting generators $a,b,c,d$.

\begin{example} 
\label{eg:Mq2} 
Define the algebra
$GL_q(2) := M_q(2)[t] /\bigl( t(ad - q^{-1}bc) -1 \bigr)$, with
coproduct and counit given by the same formulas as for $GL(2)$. Its
antipode is given by
\[
S\begin{pmatrix}
a & b
\\
c & d
\end{pmatrix}
=
t\begin{pmatrix}
d & -qb
\\
-q^{-1}c & a
\end{pmatrix}, \quad
S(t) = ad - q^{-1}bc.
\]
Note that $GL_q(2)$ is not commutative or cocommutative.

As with $GL(2)$, $GL_q(2)$ is not a matrix quantum group. However,
with
\[
\begin{pmatrix}
a & b & 0 \\ c & d & 0 \\ 0 & 0 & t \end{pmatrix}^{-1}
= \begin{pmatrix}
dt & -qbt & 0 \\ -q^{-1}ct & at & 0 \\ 0 & 0 & ad-q^{-1}bc
\end{pmatrix},
\]
and, as before, $(GL_q(2) \ox 1)\Dl(GL_q(2)) 
= \Dl(GL_q(2)) (1 \ox GL_q(2)) = GL_q(2) \ox GL_q(2)$.
\end{example}

The very first definition that accommodates our algebraic approach is
the one due to Drinfel'd in his conference \cite{Drinfeld}, that was
the spark that ignited a gigantic interest in the subject. A
\textit{quantum group} is a noncommutative noncocommutative Hopf
algebra. Some authors, like \cite{Pareigis}, call them \textit{proper
quantum groups}. So far, $GL_q(2)$ and $SL_q(2)$ are proper quantum
groups. Observe that every proper quantum group, being a Hopf algebra,
satisfies the conditions of coassociativity and cancellation.

From Lemma~\ref{lm:antipode-vs-can}, for every Hopf algebra, in
particular every quantum group, the canonical map corresponding to the
coaction $\dl = \Dl$ is bijective, and we conclude the following
result.

\begin{corollary} 
\label{cr:k-HG-extnsion}
For every proper quantum group $H$, $k \subseteq H$ is a Hopf--Galois
extension by the bialgebra~$H$.
\end{corollary}

All of the examples in subsection~\ref{ssc:Hopf-algebras} about Hopf
algebras are commutative or cocommutative. The next, taken from
\cite{Pareigis}, is referred to as the smallest proper quantum group.

\begin{example} 
\label{eg:smallest}
Consider the algebra (attributed to M. Sweedler) in noncommuting
variables $H_4 := \bF\val{g,x}/(g^2 - 1, x^2, xg + gx)$, with
structure
\begin{alignat*}{3}
\Dl(g) &= g \ox g, & 
\eps(g) &= 1, & 
S(g) &= g^{-1} = g,
\\
\Dl(x) &= x \ox 1 + g \ox x, \quad &
\eps(x) &= 0, \quad &
S(x) &= -gx.
\end{alignat*}
In \cite{Pareigis} it is noted that this is, up to isomorphism, the
only noncommutative ($gx = -xg$) noncocommutative
($\tau \Dl(x) \neq \Dl(x)$) Hopf algebra of dimension~$4$.
\end{example}

The following example is from~\cite{Taft} (see also \cite{VanDaele}).

\begin{example} 
\label{eg:smallest-2}
Let $H'_4 := \bC\val{a,b}/(a^4 - 1, b^2, ab - iba)$ with structure
\begin{alignat*}{3}
\Dl(a) &= a \ox a, & 
\eps(a) &= 1, & 
S(a) &= a^{-1} = a^3,
\\
\Dl(b) &= a \ox b + b \ox a^{-1}, \quad & 
\eps(b) &= 0, \quad & 
S(b) &= ib.
\end{alignat*}
With $\bF = \bC$, $g = a^2$ and $x = ab$, $H_4$ is a Hopf subalgebra
of $H'_4$. Here $S^2 \neq \id_H = S^4$.
\end{example}

In Example~\ref{eg:enveloping-Lie-algebra}, we noted that the
universal enveloping algebra $\sU(\g)$ of a finite-dimensional Lie
algebra $\g$ is a cocommutative Hopf algebra. As a particular case,
consider the Lie algebra of $2 \x 2$ matrices denoted by
$\gsl(2) = \gsl_2(\bC)$, generated by
\[
H = \twobytwo{1}{0}{0}{-1}, \quad
E = \twobytwo{0}{1}{0}{0}, \quad
F = \twobytwo{0}{0}{1}{0},
\]
which satisfy the relations $[H,E] = 2E$, $[H,F] = -2F$ and
$[E,F] = H$. Its universal enveloping algebra is $\sU(\gsl(2)) 
= \bC\val{h,e,f}/(he - eh - 2e, hf - fh + 2f, ef - fe - h)$, with Hopf
algebra structure given on the generators by the morphisms
\[
\Dl x = x \ox 1 + 1 \ox x,  \quad  \eps(1) = 1, \quad
\eps(x) = 0,  \quad  S(x) = -x.
\]

\begin{example} 
\label{eg:slq2}
For $q \in \bC \setminus \{-1,0,1\}$, define $\sU_q(\gsl(2))$ as the
noncommutative algebra generated by $h, h^{-1}, e, f$ subject only to
the relations
\[
h e h^{-1} = q^2 e, \quad
h f h^{-1} = q^{-2} f, \quad 
ef - fe = \frac{h - h^{-1}}{q - q^{-1}}\,.
\]
It is an infinite-dimensional Hopf algebra with structure given by
\begin{gather*}
\Dl(h) = h \ox h, \quad
\Dl(e) = 1 \ox e + e \ox h, \quad
\Dl(f) = h^{-1} \ox f + f \ox 1,
\\
\eps(h) = 1, \quad \eps(e) =  \eps(f) = 0,
\\
S(h) = h^{-1}, \quad  S(e) = -eh^{-1}, \quad S(f) = -hf.
\end{gather*}
It is noncocommutative and hence, a proper quantum group. Also note
that $S^2 \neq \id$.%
\footnote{%
If $H$ is either commutative or cocommutative, then $S^2 = \id$, 
without necessarily assuming that $S$ is invertible.}
\end{example}


\subsection{Quantum principal bundles} 
\label{ssc:Quantum-principal-bundles}

In \cite{Brzezinski94}, based on the work in \cite{BrzezinskiMajid},
the notion of a translation map is dualized to the noncommutative
setting, showing that the notion of a `quantum principal bundle' is
tantamount to the existence of this generalized translation map.
Recall that by \cite{Husemoller}, a continuous translation map is
equivalent to having a (topological) principal bundle (with Hausdorff
base space).%
\footnote{%
The relations between the translation map, the canonical map and the
Hausdorff topology of the base space are clarified in~\cite{Liguria}.}
Classically, with 
$X \x_G X := \set{(x,y) : x,y\in X,\ \sO_x = \sO_y} \subseteq X \x X$,
the \textit{canonical map} corresponding to the action of $G$ on~$X$
is the surjective map
$\ga \: X \x G \to X \x_G X : (x,g) \mapsto (x, x\.g)$. When the
action is free and thus the canonical map is invertible, its inverse
generates the important \textit{translation map}, as follows. Consider
the composition:
$$
X  \x_G X \longto^{\ga^{-1}} X  \x G \longto^{\pr_2} G
: (x,y) = (x,x\.g) \mapsto (x,g) \mapsto g.
$$
That is to say, if $x,y \in X$ belong to the same orbit there is a
unique group element $g \in G$ such that $y = x\.g$. The map
$\tau \: X \x_G X \to G$, that assigns to the pair $(x,x\.g)$ the
element~$g$, is called the \textit{translation map}.

The generalization in \cite{Brzezinski94} goes along these lines: let
$H$ be a Hopf algebra and $A$ a right $H$-comodule algebra with a
coaction $\dl_A \: A \to A \ox H$. The map
$\chi \: A \ox A \to A \ox H$ is defined by
\[
\chi = (\mu_A \ox \id) \circ (\id \ox \dl_A),  \word{or}
\chi(a \ox a') = a {a'}_{\!(0)} \ox {a'}_{\!(1)}
\mot{for any} a,a' \in A.
\]
Observe that $\chi = \wt{\can}$ from Section~\ref{sec:HG-extensions}.
With a nod to~\cite{BrzezinskiMajid}, the coaction $\dl_A$ is called
\textit{free} if $\chi$ is a surjection. Those authors define also a
map $^\sim \: A_2 \to A \ox \ker\eps$ by $^\sim := \chi|_{A_2}$.
There, for any algebra $A$, $A_2 \subset A \ox A$ denotes the kernel
of the multiplication $\mu_A$ in $A$. Let $B = A^{\co H}$ (denoted
$A^H$ in~\cite{Brzezinski94}). The coaction $\Dl_A$ of $H$ on $A$ is
called \textit{exact} if $\ker^\sim = A B_2 A$ and
\cite{BrzezinskiMajid} defines a \textit{quantum principal bundle}
with structure quantum group $H$ and base $B = A^{\co H}$ if the
coaction $\Dl_H$ is free and exact. Then, for a quantum principal
bundle $A(B,H)$, the linear map $\tau \: H \to A \ox_B A$, given by
$\tau(h) := \sum_{i\in I} a_i \ox_B {a'}_{\!i}$, where
$\sum_{i\in I} a_i \ox {a'}_{\!i} = \chi^{-1}(1 \ox h)$, is called a
translation map.%
\footnote{This $\tau$ should not be confused with the flip map
$\tau(x \ox y) := y \ox x$.}

Their Lemma~3.2 states that if $A$ is a right $H$-comodule algebra
with a free coaction $\dl_A \: A \to A \ox H$ and $B = A^{\co H}$ and
if there is a translation map $\tau \: H \to A \ox_B A$ in $A$, then
the coaction $\dl_A$ is exact and hence $A(B,H)$ is a quantum
principal bundle.

As an example, recall Example 4.2 of~\cite{BrzezinskiMajid}: let $H$
be a Hopf algebra and $A$ an $H$-comodule algebra with invariant
subalgebra $B$. Suppose that there exists a map $\Phi \: H \to A$ such
that $\Dl_A \circ \Phi = (\Phi \ox \id) \circ \Dl_A$,
$\Phi(1_H) = 1_A$, so that $\Phi$ is an intertwiner for the right
coaction. Then $A$ is a quantum principal bundle. They call
$A(B,H,\Phi)$ a \textit{trivial bundle with trivialization}~$\Phi$.

\begin{example} 
\label{eg:tensor-product-extension-2}
From Example~\ref{eg:tensor-product-extension} we know that $A \ox H$
is an $H$--Galois extension of an algebra $A$, whenever $H$ is a Hopf
algebra and the coaction from $A \ox H \to A \ox H \ox H$ is given by
$\dl^A = \id_A \ox \Dl_H$. The canonical map and its inverse are given
by
\[
\can \: a \ox h \ox h' \mapsto a \ox h {h'}_{\!(1)} \ox {h'}_{\!(2)};
\qquad
\can^{-1} \: a \ox h \ox h' 
\mapsto a \ox h S({h'}_{\!(1)}) \ox {h'}_{\!(2)},
\]
so that $\tau(h) = \can^{-1}(1_A \ox 1_H \ox h)
= 1_A \ox S(h_{(1)}) \ox h_{(2)}$. In this situation,
$\chi \: (A \ox H) \ox (A \ox H) \to A \ox H \ox H$ is defined by
$\chi(a \ox h \ox a' \ox h') = aa' \ox {h'}_{\!(1)} \ox {h'}_{\!(2)}$,
which is evidently surjective, making $\Dl_{A\ox H}$ free. It is quite
complicated to identify the kernel of $\mu_{A\ox H}$, as we do not
expect a general $(A \ox H) (A,H)$ to be a quantum principal bundle.
\end{example}

\begin{example} 
\label{eg:Laurent-bialgebra-Hopf-2}
From Example~\ref{eg:Laurent-bialgebra-Hopf}, with
$A = H = F[X,X^{-1}]$ and $\can_H^{-1}(X^k \ox X^l) = X^{k-l}\ox X^l$,
we know that $\tau(X^l) = X^{-l}\ox X^l$. Also,
$\chi \: H \ox H \to H \ox H$ is defined by
$\chi(X^{k-l} \ox X^l) = X^k \ox X^l$, which is surjective, making
$\Dl$ free.

To find the kernel of $\mu$ we observe that
\[
0 = \mu\biggl( \sum_{k,l} f_{k,l} X^{k-l} \ox X^l \biggr)
= \sum_{k,l} f_{k,l} X^k 
\word{if and only if each} \sum_l f_{k,l} = 0.
\]
Thus, $H_2 = \Set{\sum_{k,l} f_{k,l} X^{k-l} \ox X^l
: \sum_l f_{k,l} = 0 \mot{for each fixed} k \in \bZ}$. Also,
$\ker(\eps) = \Set{\sum_j f_j X^j : \sum f_j = 0}$, implying that
$^\sim \: H_2 \to H \ox \ker\eps$ is given by
\[
\chi\bigr|_{H_2} \biggl( \sum_{k,l} f_{k,l} X^{k-l} \ox X^l \biggr)
= \sum_k \biggl( X^k \ox \sum_l f_{k,l} X^l \biggr),
\]
where $\sum_l f_{k,l} = 0$ for each fixed~$k$. Because $B = H^{\co H}
= \set{ \sum_j g_j X^j : \sum_j g_j X^j \ox (X^j - 1) = 0}$, which is
the subalgebra $\{g_0\,1\}$ generated by~$1$, we obtain $B_2 = \{0\}$.
Since $\ker^\sim \neq H B_2 H =\{0\}$, the coaction fails to be exact.
\end{example}

\begin{example} 
\label{eg:bialgebra-bin-S-2}
In Example~\ref{eg:bialgebra-bin-S} we had $A = H = F_b[X]$,
$\Dl(X^n) := \sum_{l=0}^n \binom{n}{l} X^l \ox X^{n-l}$ and
\[
\chi(X^k \ox X^n) = \sum_{l=0}^n \binom{n}{l} X^{k+l} \ox X^{n-l}
= (X^k \ox 1) \sum_{l=0}^n \binom{n}{l} X^l \ox X^{n-l}.
\]
In particular, $\chi(X^k \ox 1) = X^k \ox 1$ and
$\chi(X^k \ox X) = X^k \ox X + X^{k+1} \ox 1$. The relation
\[
\chi(X^k \ox X^n) 
= X^k \ox X^n + \sum_{l=1}^n \binom{n}{l} X^{k+l} \ox X^{n-l}
\word{for} n > 0
\]
shows by induction that $\chi$ is surjective and $\Dl$ is free.

Similarly to Example~\ref{eg:Laurent-bialgebra-Hopf-2},
$\ker(\eps) = \set{\sum_{j>0} f_j X^j}$, and an element of $H_2$ is of
the form $\sum_{r \geq 0} \sum_{l=0}^r f_{r,l} X^{r-l} \ox X^l$, with
$\sum_{l=0}^r f_{r,l} = 0$ for each $r \in \bN$. In particular,
$f_{0,0} = 0$ and
\[
H_2 = \SET{ \sum_{r>0} \sum_{l=0}^r f_{r,l} X^{r-l} \ox X^l 
: \sum_{l=0}^r f_{r,l} = 0 \mot{for each fixed} r > 0}.
\]
In this case, $^\sim \: H_2 \to H \ox \ker\eps$ is given by
\begin{align*}
\MoveEqLeft{
\chi\bigr|_{H_2} \biggl( \sum_{r>0} f_{r,0} X^r \ox 1
+ \sum_{r>0} \sum_{l=1}^r f_{r,l} X^{r-l} \ox X^l \biggr)}
\\
&= \sum_{r>0} f_{r,0} X^r \ox 1 
+ \sum_{r>0} \sum_{l=1}^r f_{r,l} \biggl( X^{r-l} \ox X^l
+ \sum_{s=1}^l \binom{l}{s} X^{r-l+s} \ox X^{l-s} \biggr)
\\
&= \sum_{r>0} \sum_{l=0}^r f_{r,l} X^{r-l} \ox X^l
+ \sum_{r>0} \sum_{l=1}^r f_{r,l} \sum_{s=1}^l
\binom{l}{s} X^{r-l+s} \ox X^{l-s},
\end{align*}
where $\sum_l f_{r,l} = 0$ for each fixed $r > 0$; and we want to
compute its kernel. By the grading provided by the index $r$, it is
enough to check that
\[
0 = f_{r,r} 1 \ox X^r + \sum_{l=0}^{r-1} f_{r,l} X^{r-l} \ox X^l
+ \sum_{l=1}^r f_{r,l} \sum_{s=1}^l \binom{l}{s} X^{r-l+s} \ox X^{l-s}
\word{implies} f_{r,r} = 0;
\]
and then
\begin{align*}
0 &= \sum_{l=0}^{r-1} f_{r,l} X^{r-l} \ox X^l
+ \sum_{l=1}^{r-1} f_{r,l} 
\sum_{s=1}^l \binom{l}{s} X^{r-l+s} \ox X^{l-s}
\\
&= f_{r,r-1} X \ox X^{r-1} + \sum_{l=0}^{r-2} f_{r,l} X^{r-l} \ox X^l
+ \sum_{l=1}^{r-1} f_{r,l} 
\sum_{s=1}^l \binom{l}{s} X^{r-l+s} \ox X^{l-s},
\end{align*}
showing that $f_{r,r-1} = 0$; and so on. Induction shows that
$\ker^\sim = \{0\}$.

That $B = H^{\co H}$ is the subalgebra $\{f_0\,1\}$ generated by~$1$,
follows from
\[
\sum_r f_r \biggl( X^r \ox 1
+ \sum_{l=1}^r \binom{r}{l} X^{r-l} \ox X^l \biggr) 
= \sum_r f_r X^r \ox 1 
\iff  \sum_{r>0} f_r \sum_{l=1}^r \binom{r}{l} X^{r-l} \ox X^l = 0.
\]
Thus, the grading given by the terms $X^{r-l} \ox X^l$ implies that
$f_r = 0$ for each $r > 0$. As a result, $B_2 = \{0\}$, and
$\ker^\sim = H B_2 H$. By~\cite{BrzezinskiMajid},
$\{f_0\,1\} \subseteq F_b[X]$ is a quantum principal bundle.
\end{example}

\begin{example} 
\label{ex:OofA2}
Back to Example~\ref{eg:O-of-A}: we know that
$\sO(X/G) \subseteq \sO(X)$ is a Hopf--Galois extension by the
bialgebra $\sO(G)$, where the coaction of $\sO(G)$ over $\sO(X)$ is
given by $(\dl f)(x,g) := f(x\.g)$ and
$\sO(X)^{\co\sO(G)} = \sO(X/G)$. The map 
$\can \: \sO(X) \ox_B \sO(X) \to \sO(X) \ox \sO(G)$ is a bijection,
making $\chi = \wt{\can} \: \sO(X) \ox \sO(X) \to \sO(X) \ox \sO(G)
\equiv \sO(X \x_G X)$ a surjection and $\dl$ a free coaction. As
before, the identification
$f \ox f' \equiv [(x,y) \mapsto f(x)\,f'(y)]$ shows that
$\chi(f \ox f') = (f \ox f')|_{X \x_G X}$. Here
$\chi(f \ox f')(x,x\.g) = f(x) f'(x\.g) = f''(x\.g)
= 1 \ox f'(x, x\.g)$ if and only if $\tau(f'') = 1 \ox f''$. Because
in $\sO(X)$ the algebra structure is given by pointwise
multiplication, the coaction fails to be exact.
\end{example}

A refreshed treatment is given in
\cite{AschieriBieliavskyPaganiSchenkel} and
\cite{AschieriLandiPagani}. Given a Hopf algebra $H$, a
\textit{quantum subgroup} of $H$ is a Hopf algebra $H'$ together with
a surjective bialgebra homomorphism $\pi \: H \to H'$. The restriction
via $\pi$ of the coproduct of $H$,
$\dl_H := (\id \ox \pi)\circ \Dl \: H \to H \ox H'$, induces on $H$
the structure of a right $H'$-comodule algebra. The subalgebra $B :=
H^{\co H'} \subseteq H$ of coinvariants is called a \textit{quantum
homogeneous $H$-space}. When the associated canonical map 
$\chi \: H \ox_B H \to H \ox H'
: h \ox_B h' \mapsto h {h'}_{\!(1)} \ox \pi({h'}_{\!(2)})$ is 
bijective, i.e., $B \subseteq H$ is a Hopf--Galois extension, $H$ is
called a \textit{quantum principal bundle} over the quantum
homogeneous space~$B$. It is less restrictive than the requirement in
\cite{Brzezinski94} about the exactness of the coaction. Examples
\ref{eg:Laurent-bialgebra-Hopf-2}, \ref{eg:bialgebra-bin-S-2}
and~\ref{ex:OofA2} now fit the description (with $\pi = \id_H$).

In \cite{AschieriLandiPagani} (aiming at the notion of group of
noncommutative gauge transformations) the authors study conditions for
the canonical map $\chi$ to be an algebra map in the setting of
coquasitriangular Hopf algebras (where the category of $H$-comodule
algebras is a monoidal category). In that context, they show that the
canonical map is a morphism and its inverse is determined by its
restriction to $1 \ox H$, named \textit{translation map}, 
$\tau = \chi^{-1}|_{1\ox H'} \: H' \to H \ox_B H$.

\begin{example} 
\label{eg:smallest-3}
In Example~\ref{eg:smallest-2}, consider 
$H = H'_4 = \linspan_\bC\{1, a, a^2, a^3, b, a^2 b, a^3 b\}$ and its
subalgebra $H' = \linspan_\bC\{1, a^2\}$ generated by $a^2$, with
structure
\[
\Dl'(a^2) = a^2 \ox a^2, \qquad \eps'(a^2) = 1, \qquad S'(a^2) = a^2.
\]
With the surjective Hopf algebra homomorphism $\pi \: H \to H'$ given
by $\pi(a) = a^2$, $\pi(b) = 0$, $H'$ is a quantum subgroup of~$H'_4$.
In this case,
$\dl_{H'_4} := (\id \ox \pi) \circ \Dl \: H'_4 \to H'_4 \ox H'$ is
given by $\dl_{H'_4}(a) = a \ox a^2$ and $\dl_{H'_4}(b) = b \ox a^2$.
The subalgebra
\[
B := (H'_4)^{\co H'} := \set{h \in H'_4 : \dl_{H'_4}(h) = h \ox 1}
= \linspan_\bC \{1, a^2, a^3 b\}
\]
is the subalgebra of $H'_4$ generated by $a^2$ and $a^3b$, and is a
quantum homogeneous $H'_4$-space. Since
$H'_4 \ox_B H'_4 = \set{h \ox_B 1, h \ox_B a^2 : h \in H'_4}$, with
$\dim_\bC(H'_4 \ox_B H'_4) = 16 = \dim_\bC(H'_4 \ox H')$, where
$H'_4 \ox H' = \set{h \ox 1, h \ox a^2 : h \in H'_4}$, it is clear 
that the associated canonical map
$\chi \: H'_4 \ox_B H'_4 \to H'_4 \ox H'
: h \ox_B h' \mapsto h {h'}_{\!(1)} \ox \pi({h'}_{\!(2)})$ satisfies
$\chi(h \ox 1) = h \ox 1$ and $\chi(h \ox a^2) = ha^2 \ox a^2$; we
know that $\chi$ is bijective, showing that $B \subseteq H'_4$ is a
quantum principal bundle over the quantum homogeneous space~$B$.

In this case, $\tau(1) = \chi^{-1}(1 \ox 1) = 1 \ox 1$ and
$\tau(a^2)  = \chi^{-1}(1\ox a^2) = a^2 \ox a^2$.
\end{example}

\begin{example} 
\label{eg:pi-GL2-to-SL2}
Consider $\pi \: H = GL_2(\bC) \to H' = SL_2(\bC)$ where
$\pi(a) = a$, $\pi(b) = b$, $\pi(c) = c$, $\pi(d) = d$ and
$\pi(t) = 1$. We can write
\[
\dl_H = (\id \ox \pi) \circ \Dl \begin{pmatrix}
a & b & 0 \\ c & d & 0 \\ 0 & 0& t \end{pmatrix}
= (\id \ox \pi) \begin{pmatrix}
a & b & 0 \\ c & d & 0 \\ 0 & 0 & t \end{pmatrix}
\ox \begin{pmatrix}
a & b & 0 \\ c & d & 0 \\ 0 & 0 & t \end{pmatrix}
= \begin{pmatrix}
a & b & 0 \\ c & d & 0 \\ 0 & 0 & t \end{pmatrix}
\ox \begin{pmatrix}
a & b & 0 \\ c & d & 0 \\ 0 & 0 & 1 \end{pmatrix}
\]
to deduce that $B = H^{\co H'} = \{\dl_{GL_2(\bC)} (h) = h \ox 1\}$ is
the subalgebra generated by~$t$. By switching every appearance of $t$
in the right `leg' to the left one, $GL_2(\bC) \ox_B GL_2(\bC)$ is
generated by
$\set{h \ox_B h' : h \in \{a,b,c,d,t\},\ h' \in \{a,b,c,d\}}$. Because
$\pi(h') = h'$ for every $h' \in \{a,b,c,d\}$, the canonical map
$\chi \: GL_2(\bC) \ox_B GL_2(\bC) \to GL_2(\bC) \ox SL_2(\bC)$ is
given by $\chi(h \ox_B h') = h {h'}_{\!(1)} \ox \pi({h'}_{\!(2)})
= h {h'}_{\!(1)} \ox {h'}_{\!(2)} = (h \ox 1) \Dl_{SL_2(\bC)}(h')$. As
a result, it is bijective, and $GL_2(\bC)$ is a quantum principal
bundle over its subalgebra generated by~$t$.

For $\tau \: SL_2(\bC) \to GL_2(\bC) \ox_B GL_2(\bC)$ with
$\tau(h) = \chi^{-1}(1 \ox h)$, one finds that
\begin{align*}
\tau(a) &= \chi^{-1}(1 \ox a) = \chi^{-1}((ad - bc) \ox a)
= d \ox_B a - b \ox_B c,
\\
\tau(b) &= \chi^{-1}(1 \ox b) = \chi^{-1}((ad - bc) \ox b)
= d \ox_B b - b \ox_B d,
\\
\tau(c) &= \chi^{-1}(1 \ox c) = \chi^{-1}((ad - bc) \ox c)
= a \ox_B c - c \ox_B a,
\\
\tau(d) &= \chi^{-1}(1 \ox d) = \chi^{-1}((ad - bc) \ox d)
= a \ox_B d - c \ox_B b,
\end{align*}
on the generators of $SL_2(\bC)$.
\end{example}

\begin{example} 
\label{eg:pi-GL2-to-SL2-two}
We now consider the noncommutative deformations of $SL_2(\bC)$ and
$GL_2(\bC)$ (see Examples~\ref{eg:SLq2} and~\ref{eg:Mq2} above). We
take $H' = SL_q(2) = M_q(2)/\bigl( (ad - q^{-1}bc) -1 \bigr)$ and
$H = GL_q(2) = M_q(2)[t]/\bigl( t(ad - q^{-1}bc) -1 \bigr)$. As in
Example~\ref{eg:pi-GL2-to-SL2}, we define the map
$\pi \: H = GL_q(2) \to H' = SL_q(2)$ by $\pi(a) = a$,
$\pi(b) = b$, $\pi(c) = c$, $\pi(d) = d$, $\pi(t) = 1$. Recall that
$SL_q(2)$ and $GL_q(2)$ have the same comultiplication and counit as
$SL(2)$ and $GL(2)$, respectively, with antipodes
\[
S_{SL_q(2)} \twobytwo{a}{b}{c}{d} 
= \twobytwo{d}{-qb}{-q^{-1}c}{a},  \quad
S_{GL_q(2)} \begin{pmatrix}
a & b & 0 \\ c & d & 0 \\ 0 & 0 & t \end{pmatrix}
= \begin{pmatrix}
td & -qtb & 0 \\ -q^{-1}tc & ta & 0 \\ 0 & 0 & ad - q^{-1}bc
\end{pmatrix}.
\]
Thus, to check that $\pi$ is a Hopf algebra homomorphism, it remains
only to verify the relation
$S_{SL_q(2)} \circ \pi = \pi \circ S_{GL_q(2)}$:
\begin{align*}
\pi \circ S_{GL_q(2)}(a) 
&= \pi(td) = d = S_{GL_q(2)}(a) = S_{GL_q(2)}(\pi(a)),
\\
\pi \circ S_{GL_q(2)}(b) 
&= \pi(-qtb) = -qb = S_{GL_q(2)}(b) = S_{GL_q(2)}(\pi(b)),
\\
\pi \circ S_{GL_q(2)}(c) 
&= \pi(-q^{-1}tc) = -q^{-1}c = S_{GL_q(2)}(c) = S_{GL_q(2)}(\pi(c)),
\\
\pi \circ S_{GL_q(2)}(d) 
&= \pi(ta) = a = S_{GL_q(2)}(d) = S_{GL_q(2)}(\pi(d)),
\\
\pi \circ S_{GL_q(2)}(t) 
&= \pi(ad - q^{-1} bc) = ad - q^{-1} bc = 1 = S_{GL_q(2)}(1)
= S_{GL_q(2)}(\pi(t)),
\end{align*}
because $ad - q^{-1} bc = 1$ in $SL_q(2)$. As a result, the same
computations as in Example~\ref{eg:pi-GL2-to-SL2} show that $GL_q(2)$
is a quantum principal bundle over $B$, its subalgebra generated
by~$t$.

To compute $\tau$ now we use the \textit{quantum determinant}
$ad - q^{-1}bc$ to observe that the map
$\tau \: SL_q(2) \to  GL_q(2) \ox_B GL_q(2)$, with
$\tau(h) := \chi^{-1}(1 \ox h)$, is given by
\begin{align*}
\tau(a) &= \chi^{-1}(1 \ox a) = \chi^{-1}((ad - q^{-1}bc) \ox a) 
= d \ox_B a - q b \ox_B c,
\\
\tau(b) &= \chi^{-1}(1 \ox b) = \chi^{-1}((ad - q^{-1}bc) \ox b) 
= d \ox_B b - q^{-1}b \ox_B d,
\\
\tau(c) &= \chi^{-1}(1 \ox c) = \chi^{-1}((ad - q^{-1}bc) \ox c) 
= a \ox_B c - q^{-1}c \ox_B a,
\\
\tau(d) &= \chi^{-1}(1 \ox d) = \chi^{-1}((ad - q^{-1}bc) \ox d) 
= a \ox_B d - q^{-1}c \ox_B b,
\end{align*}
on the generators of $SL_q(2)$.
\end{example}

Around the same time as \cite{Brzezinski94} and
\cite{BrzezinskiMajid}, the same ideas were addressed in
\cite{Durdevic95}, working on $*$-algebras: ``In the framework of
noncommutative differential geometry, the Hopf--Galois extensions are
understandable as analogues of principal bundles''.


\section{Conclusions} 

There are two main approaches to the problem of finding an adequate
generalization of a principal bundle in noncommutative geometry: the
noncommutative algebraic approach and the $C^*$-algebraic formulation.
The noncommutative algebraic setting of Hopf--Galois extensions has
limitations when trying to extend principal bundles beyond the compact
case; nevertheless, it is the consolidated point of view. It extends
the injectivity of the canonical map
$\can \: X \x G \to X \x_G X : (x,g) \mapsto (x,x\. g)$, 
in the form of the surjectivity of the homomorphism 
$\can^* \: C(X \x_G X) \to C(X \x G) : f \mapsto f \circ \can$. As a
map acting between algebras, it can be generalized in such a way that
the choice of the type of algebra shows the kind of extension one
wants to achieve. The venue explored in these notes was the one that
replaces the group $G$ by a more general structure, namely a Hopf
algebra.

Recall that a driving idea in noncommutative geometry is to replace a
space $X$ by a commutative algebra of functions $A(X)$ from $X$ to
$\bC$, in such a way that knowledge of $A(X)$ allows one to recover
$X$ up to an isomorphism in the appropriate category. One can then
replace $A(X)$ by a possibly noncommutative algebra $\sA$ and regard
it as `the algebra of functions on a putative noncommutative space'.
Furthermore, if $f \: X \to Y$ is a bijection between spaces, the map
$f^* \: A(Y) \to A(X) : h \mapsto h\circ f$ is an isomorphism of
algebras. Conversely, having an algebra map $F \: \sA \to \sB$, one
may choose to regard it as $F = f^*$ where $f \: X \to Y$ is a
(possibly non-existing) map between (possibly non-existing) spaces $X$
and~$Y$.

The article \cite{Schneider1990} stands as the point of departure for
seeking a noncommutative substitute for a principal bundle, providing
the main tool -- the dualization process -- for the passage from
modules to comodules and from algebras to coalgebras. The works
\cite{Durdevic1993, Durdevic95, Brzezinski94, BrzezinskiMajid, Pflaum,
HajacThesis} extend this vision. The viewpoint of a Hopf--Galois
extension as being a dual notion to a principal bundle was
consolidated in \cite{BrzezinskiMajid2, Schauenburg1998, Brzezinski99,
HajacMajid}. Recently \cite{AschieriFioresiLatini,
AschieriBieliavskyPaganiSchenkel, AschieriLandiPagani,
AschieriLandiPagani2, BrzezinskiSzymanski} gave a refreshed treatment
to the subject. For a coaction $\dl \: A \to A \ox H$ of the
bialgebra~$H$ on the algebra~$A$, its coinvariant subalgebra is
$A^{\co H} := \set{a \in A : \dl(a) = a \ox 1_H}$, and $A$ is a
coalgebra extension of $A^{\co H}$ by $H$. If $A$ and $B$ are
$k$-algebras with $B \subset A$ and some algebra coaction $\dl$ of a
bialgebra $H$ on~$A$ leaves $B$ coinvariant, then $A$ is a coalgebra
extension of $B$ by~$H$. The balanced tensor product is the quotient
$A \ox_B A := (A \ox A)/J_B$ with $J_B$ the ideal generated by the
elements of the form $ab \ox c - a \ox bc$ for $a,c \in A$, $b \in B$.
The canonical map for~$\dl$ is the $k$-linear map
$\can \: A \ox_B A \to A \ox H : a \ox_B c \mapsto a\,\dl(c)$. The
extension is Hopf--Galois if this map is bijective.

A bialgebra $H$ is actually a Hopf algebra, provided that its
canonical map is invertible (see
Proposition~\ref{pr:Hopf-and-can-map}). More generally, if a bialgebra
$H$ admits a coaction whose canonical map is bijective, then $H$ has a
Hopf algebra structure (see
Proposition~\ref{pr:bialgebra-with-coaction-and-bijective-map} for the
finite-dimensional case and
Proposition~\ref{pr:Hopf-and-faithfully-flat} for the flat case).

The very first definition of quantum group that accommodates our
algebraic approach is the one due to Drinfel'd: a \textit{quantum
group} is a noncommutative noncocommutative Hopf algebra. Some
authors, like \cite{Pareigis}, call them \textit{proper quantum
groups}. For every proper quantum group~$H$, $k \subseteq H$ is a
Hopf--Galois extension of~$k$ by the bialgebra $H$. Any Hopf algebra
$H$ satisfies the following properties:
\begin{enumerate}
\item
$(\Dl \ox \id_H) \Dl = (\id_H \ox \Dl)\Dl : H \to H \ox H \ox H$, and
\item
$(H \ox 1)\Dl(H) = H \ox H = \Dl(H)(1 \ox H)$,
\end{enumerate}
which are algebraic versions of the requirements for a compact quantum
group. The cancellation property (b) is satisfied by every matrix
quantum group.

In \cite{Brzezinski94} and \cite{BrzezinskiMajid}, the notion of a
translation map is dualized to the noncommutative setting, showing
that the notion of a `quantum principal bundle' is equivalent to the
existence of such a generalized translation map. A refreshed treatment
is given in \cite{AschieriBieliavskyPaganiSchenkel,
AschieriLandiPagani}: a \textit{quantum subgroup} of a given Hopf
algebra $H$ is another Hopf algebra $H'$ together with a surjective
Hopf algebra homomorphism $\pi \: H \to H'$ which induces on $H$ the
structure of a right $H'$-comodule algebra. The subalgebra
$B := H^{\co H'} \subseteq H$ of coinvariants is called a
\textit{quantum homogeneous $H$-space}. When the associated canonical
map $\chi \: H \ox_B H \to H \ox H'$ is bijective, i.e.,
$B \subseteq H$ is a Hopf--Galois extension, $H$ is called a
\textit{quantum principal bundle} over the quantum homogeneous
space~$B$. This is less restrictive than the requirement in
\cite{Brzezinski94} of exactness of the coaction. In
\cite{AschieriLandiPagani} the authors study conditions for the
canonical map $\chi$ to be an algebra map in the setting of
coquasitriangular Hopf algebras, showing that the canonical map is an
isomorphism, whose inverse is determined by its restriction to
$1 \ox H$, named the \textit{translation map}.


\subsection*{Acknowledgments}

Support from the Vicerrectoría de Investigación of the Universidad de
Costa Rica is acknowledged. Special thanks to Joseph C. Várilly for
constructive criticism.


\begin{center}
Last updated: \today
\end{center}

\end{document}